%% file: main.tex
\documentclass[12pt,a4paper]{amsart}
\makeatletter
\renewcommand\normalsize{%
    \@setfontsize\normalsize{11.7}{12pt plus .3pt minus .3pt}%
    \abovedisplayskip 10\p@ \@plus4\p@ \@minus4\p@
    \abovedisplayshortskip 6\p@ \@plus2\p@
    \belowdisplayshortskip 6\p@ \@plus2\p@
    \belowdisplayskip \abovedisplayskip}
\renewcommand\small{%
    \@setfontsize\small{9.5}{10\p@ plus .2\p@ minus .2\p@}%
    \abovedisplayskip 8.5\p@ \@plus4\p@ \@minus1\p@
    \belowdisplayskip \abovedisplayskip
    \abovedisplayshortskip \abovedisplayskip
    \belowdisplayshortskip \abovedisplayskip}
\renewcommand\footnotesize{%
    \@setfontsize\footnotesize{8.5}{9.25\p@ plus .1pt minus .1pt}%%
    \abovedisplayskip 6\p@ \@plus4\p@ \@minus1\p@
    \belowdisplayskip \abovedisplayskip
    \abovedisplayshortskip \abovedisplayskip
    \belowdisplayshortskip \abovedisplayskip}
\ifdefined\pdfpagewidth
\else
\fi
\calclayout
\makeatother
\usepackage[dvipsnames]{xcolor}
\usepackage{amsfonts}
\usepackage{amscd}
\usepackage[utf8]{inputenc}
\usepackage[T1]{fontenc}
\usepackage{placeins}
\usepackage{subcaption}
\usepackage{soul}
\usepackage{float}
\usepackage[normalem]{ulem}

\usepackage{t1enc}
\usepackage[mathscr]{eucal}
\usepackage{indentfirst}
\usepackage{graphicx}
\usepackage{graphics}
\usepackage{pict2e}
\usepackage{epic}
\numberwithin{equation}{section}
\usepackage{epstopdf} 
\usepackage{amsmath}
\allowdisplaybreaks
\usepackage{amssymb}
\usepackage{amsthm}
\usepackage{mdframed}
\usepackage{bbm}
\usepackage{tikz}
\usepackage{paralist}
\usepackage[onehalfspacing]{setspace}
\usepackage{amsfonts}
\usepackage{color}
\usepackage{nicefrac}
\usepackage{mathrsfs}
\usepackage{multirow}
\usepackage{mathtools} %serve
\usepackage{tikz-cd}

\newcounter{dummy}
\usepackage{hyperref}
\usepackage{enumitem}
\makeatletter
\newcommand\myitem[1][]{\item[#1]\refstepcounter{dummy}\def\@currentlabel{#1}}
\makeatother
\newtheorem{thm}{Theorem}
\numberwithin{thm}{section}
\newtheorem{lemma}[thm]{Lemma}
\newtheorem{prop}[thm]{Proposition}
\newtheorem{definition}[thm]{Definition}
\newtheorem{coro}[thm]{Corollary}
\newtheorem{conjecture}[thm]{Conjecture}

\newtheorem*{thm*}{Theorem}
\newtheorem*{prop*}{Proposition}
\numberwithin{equation}{section}
\theoremstyle{remark}
\newtheorem{remark}[thm]{Remark}

\newtheorem{problem}[thm]{Problem}
\usepackage{soul}

\input{abbrev}

 \definecolor{Hanks}{rgb}{0.7,0.3,0.1}
 
\begin{document}
\title[On Shnirelman's inequality in 2D]{On Shnirelman's inequality in 2D: Irreversible behavior and Euler-Lagrange variational formulation}
\author[Schiffer \& Zizza]{Stefan Schiffer and Martina Zizza}
\address{Max-Planck Institute for Mathematics in the Sciences} 
\email{stefan.schiffer.math@gmail.com}
\email{martina.zizza@mis.mpg.de}
\subjclass[2020]{76B75,05C21,35Q31}
\keywords{fluid flows, discrete networks, group of volume-preserving diffeomorphisms, action functional, generalized flows, turbulence} 

\begin{abstract}
In this paper we prove a version of Shnirelman's inequality in the two-dimensional square $[0,1]^2$. More precisely, given a target fluid configuration $f$ (a volume-preserving diffeomorphism) we construct a divergence-free velocity field $v\in L^q_tL^p_x$ whose flow connects $f$ to the identity  and whose $L^q_{t} L^p_x$-norm can effectively be bounded by the $L^p$-difference of $f$ and $id$. However, higher regularity of the vector field might be affected.
The main difference with the higher dimensional case is that, because of the topological obstructions of dimension $\nu=2$, we have to allow that different fluid trajectories 'intersect' in space. We also observe that this choice leads to the emergence of irreversible behaviors in the Eulerian dynamics. %{\color{red} In this preliminary version, we chose to keep technicalities to their minimum to enhance the readability of the result, built upon the ideas of [Schiffer S., Zizza, M. \emph{On incompressible flows in discrete networks and Shnirelman's inequality}, (2024)].}

\end{abstract}
\maketitle

\input{sec1}

\input{sec2_new}
\input{interchangetoyproblem}

\input{sec3_new}

%\input{sec4_new}

\input{sec5_new}

\input{appendix}
%\input{sec2}
%\input{sec3}
%\input{sec4}

% \pagebreak

\bibliography{biblio.bib}
\bibliographystyle{abbrv}

\end{document}

%% file: abbrev.tex
\newcommand{\R}{\mathbb{R}}
\newcommand{\N}{\mathbb{N}}

\newcommand{\dt}{\,\textup{d}t}

\newcommand{\dH}{\,\textup{d}\mathcal{H}}
\newcommand{\Haus}{\mathcal{H}}

\everymath{\displaystyle}

\DeclareMathOperator{\dist}{dist}

\DeclareMathOperator{\divergence}{div}

\DeclareMathOperator{\spt}{spt}

\DeclareMathOperator{\id}{id}

\DeclareMathOperator{\BV}{BV}

\newcommand{\diverg}{\divergence}
\newcommand{\q}{\mathrm{Q}}

\newcommand{\Dcal}{\mathcal{D}}
\DeclareMathOperator{\Id}{Id}
\newcommand{\un}{\mathbf{e}}

\newcommand{\Rcal}{\mathcal{R}}

\definecolor{Gump}{rgb}{0,0.6,0.4}
\definecolor{Hanks}{rgb}{0.7,0.3,0.1}

%% file: sec1.tex
\section{Introduction}
When describing the flow of an incompressible, inviscid fluid in a vessel $M \subset \R^\nu$, $\nu \geq 2$, which is a compact, connected domain with Lipschitz boundary, there are two classical viewpoints: one way is to describe it via the \emph{Eulerian} formulation, i.e. looking at the fluid's velocity given at a fixed point $x \in M$ at times $t\geq 0$. The Eulerian formulation for example gives rise to the incompressible Euler equations
\begin{equation} \label{eq:Euler}
    \begin{cases}
        \partial_t u + (u \cdot \nabla) u = -\nabla p, \\
        \mathrm{div}~ u =0, \\
        u_{t=0}=u_0,\\
        u\cdot \hat n=0 \text{ a.e. on }\partial M.
    \end{cases}
\end{equation}
Here,  $u:[0,T]\times M\rightarrow \R^\nu$ is the fluid velocity, $p:[0,T]\times M\rightarrow\R$ is the hydrodynamic pressure, $u_0:M\rightarrow \R^\nu$ is the initial fluid velocity, $T>0$ is the final time, and $\hat n$ is the exterior unit normal to $M$. 

This viewpoint is contrasted by the \emph{Lagrangian} formulation, where for a particle at an initial point $x \in M$, we track its position at any time through the flow map $\varphi(t,x)=\varphi_t(x)$. Due to the incompressibility constraint the map $x \mapsto \varphi_t(x)$ has to be volume-preserving for any time, namely $\mathcal{L}^\nu(A)=\mathcal{L}^\nu(\varphi_t^{-1}(A))$ for any $A$ measurable set, where $\mathcal{L}^{\nu}$ denotes the Lebesgue measure.  For every time $t$, the map $\varphi_t$ is a \emph{configuration} of the fluid: it is a volume-preserving diffemorphism $\varphi_t\in\text{SDiff}(M)$, i.e. a bijective, smooth map with smooth inverse, invariant on the boundary. 

In the Eulerian formulation we ask the following Cauchy Problem: given an initial data $u_0$, establish whether there exists (unique) smooth solution of Euler Equations with initial data $u_0$. In contrast, the Lagrangian formulation asks, given an initial fluid configuration $f\in\text{SDiff}(M)$ and a final fluid configuration $g\in\text{SDiff}(M)$, whether it is possible to reconstruct the dynamics of the incompressible fluid starting in $f$ and leading to the final state $g$ while minimizing the total kinetic energy (\emph{two-point boundary value} problem). 

More generally, we can formulate, in the context of the Lagrangian formulation, the following question: 

\begin{problem} \label{problem:intro} {Given two fluid configurations $f,g\in\text{SDiff}(M)$, is it possible to connect them via a flow $\varphi: [0,T] \times M \to M$ in an optimal way?}
\end{problem}
\smallskip
In the scope of this problem, a \emph{flow} $\lbrace \varphi_t\rbrace\subset \text{SDiff}(M)$ is a piecewise-smooth path in the group of volume-preserving diffeomorphisms. Two fluid configurations $f,g\in\text{SDiff}(M)$ are connected if there exists a flow $\lbrace\varphi_t\rbrace$ with $\varphi_{t=0}=f$ and $\varphi_{t=T}=g$. The \emph{optimality}, which Problem \ref{problem:intro} demands, is that the associated velocity field $v(t,x)=\dot\varphi(t,\varphi_t^{-1}(x))$ should be a minimizer of a certain functional, e.g.
\begin{equation} \label{intro:energy:p:q}
\mathcal{A}_{q,p} \{\varphi_t\}_0^T = \int_0^T \Vert \dot{\varphi}(t) \Vert_{L^p}^{q} \dt=\int_0^T \Vert v(t) \Vert_{L^p}^{q} \dt,
\end{equation}
where the last equality follows by the volume-preserving condition on the flow $\varphi$.
We call this the \emph{smooth setting}, as both the end configuration as well as the path and the velocity field between them need to be smooth (i.e. $C^{\infty}$ or sometimes the Sobolev space $H^s$ for $s$ large enough, as established in the work of Ebin and Marsden \cite{EbinMarsden}). In such a smooth context, the geometric formulation of ideal hydrodynamics due to Arnold \cite{Arnold1966} connects Problem \eqref{problem:intro} to the natural question of the existence of solutions to Euler Equations \eqref{eq:Euler} whose flow connects $f$ to the configuration $g$. Indeed, any critical point of the $L^2_t L^2_x$ action functional 
\begin{equation} \label{intro:energy}
\mathcal{A} \{\varphi_t\}_0^T=\mathcal{A}_{2,2} \{\varphi_t\}_0^T = \int_0^T \Vert \dot{\varphi}(t) \Vert_{L^2}^{2} \dt
\end{equation}
is a smooth solution of Euler equations.

\medskip

Even if the initial state $\varphi_0 = id$ and the target configuration $\varphi_T=f$ are both smooth, it is unknown whether a critical point to \eqref{intro:energy} exists in a classical sense. In dimension $\nu\geq 3$ Shnirelman proved that there exist target fluid configurations $f\in\text{SDiff}(M)$ connected to the identity for which a \emph{minimizer} of the action functional does not exist. However, the example constructed in \cite{Shnirelman},\cite{Shnirelman2} does not imply non-existence of \emph{critical points} of the action functional.

On the contrary, Brenier \cite{Brenier1} proved that, in a generalized, non-smooth sense a minimizer always exists. For this we need to expand the space of classical smooth flows and consider the class of \emph{generalized incompressible flows}. Such generalized flows describe a probabilistic behaviour of fluid particles: instead of following one single trajectory, each fluid particle may 'choose' one trajectory with some probability. Stated differently, the fluid particle may split into a continuum of particles traveling along different trajectories (see also \cite{Shnirelman2}). 

\subsection{The two point problem in two space dimensions}  

The picture changes drastically when, instead of considering volume-preserving diffeomorphisms of three- (or higher-) dimensional manifolds, we consider area-preserving diffeomorphisms of surfaces. More precisely, we will denote by $\text{SDiff}([0,1]^2)$ the set of volume-preserving diffeomorphisms of $(0,1)^2$ that can be extended continuously up to the boundary $\partial (0,1)^2$ and such that the extension is the identity at the boundary. In this group, there are examples of diffeomorphisms that cannot be connected to the identity with finite action. More precisely, Shnirelman proved the following results:

\begin{thm}[Non-attainable diffeomorphism \cite{Shnirelman2}]
    There exists $\xi\in \text{SDiff}([0,1]^2)$ which is not attainable with finite action.
\end{thm} 

Moreover, among the class of attainable diffeomorphisms there is a sequence $f_n$, such that the ratio between the minimal action and the $L^2$-distance to the identity diverges, i.e.:
\begin{thm}[\cite{Shnirelman},\cite{Shnirelman2}]\label{thm:comparison}
    For every $c,C>0$ there exists $f\in\text{SDiff}([0,1]^2)$ with the following property: $\text{dist}_{\text{SDiff}([0,1]^2)}(f,Id)>C$ and $\|f-Id\|_{L^2([0,1]^2)}\leq c.$
\end{thm}

Here, we have denoted by
\begin{equation}\label{eq:distance}
    \text{dist}_{\text{SDiff}([0,1]^\nu)}(f,g)=\inf_{\varphi_t: \varphi_0=f,\varphi_1=g} \mathcal{A}_{1,2}\lbrace\varphi_t\rbrace.
\end{equation}
Here, the infimum is taken among all piecewise smooth paths $\varphi_t$ connecting $f$ to $g$. The functional $\mathcal{A}_{1,2}$ is called \emph{Length functional}, and $\text{dist}_{\text{SDiff}([0,1]^\nu)}$ is a metric (\emph{geodesic distance}).
Notice that in $\nu\geq 3$ the result of Theorem \ref{thm:comparison} does not hold. Indeed it has been proved in the previous paper by the same authors \cite{Schiffer_Zizza25} (improving the results in \cite{Shnirelman},\cite{Shnirelman2},\cite{Zizza24}) that 
\begin{thm}[Sharp Shnirelman's Inequality]\label{thm:sharp:shnirelman:3d}
    Let $\nu \geq 3$, $M=[0,1]^{\nu}$. Then there exists a positive constant $C=C(\nu)$ such that for all $f,g\in\text{SDiff}(M)$, it holds
    \begin{equation}\label{eq:sharp:shnirelman}
    \text{dist}_{\text{SDiff}([0,1]^\nu)}(f,g)\leq C\|f-g\|_{L^2}.
  \end{equation}
\end{thm}
 Theorem \ref{thm:comparison} shows that it is impossible to have a sharp Shnirelman's Inequality in $\nu=2$. However, if we relax the assumption on the flows under consideration, we ask if it is still possible to obtain a version of the sharp inequality. 

\medskip 

The aim of the present paper is to understand the validity of Shnirelman's inequality \eqref{eq:sharp:shnirelman} in dimension $\nu=2$ under relaxed assumptions on the regularity of the flows and to provide a step towards a deeper understanding of the two-dimensional dynamics from a variational perspective.

\medskip 
Without loss of generality we consider any $f\in\text{SDiff}([0,1]^2)$. Is it possible to find $\lbrace \varphi_t\rbrace $, \emph{not necessarily smooth}, connecting $f$ to the identity and a positive constant $C>0$ (depending only on $\nu=2$) such that
\begin{equation}\label{eq:minimizer:up:to:constant}
    \mathcal{A}_{1,2}\lbrace\varphi_t\rbrace_0^T\leq C\ \underset{\phi_t}{\inf}\mathcal{A}_{1,2}\lbrace\phi_t\rbrace_0^T?
\end{equation}
\smallskip
Here the requirement is that the constant $C$ does not depend on the choice of the diffeomorphism $f$, and the infimum is taken over all paths of measure-preserving maps $S([0,1]^2)$ satisfying the following assumptions:

\begin{enumerate}
    \item\label{hyp:1} $\phi\in  AC_{\text{loc}}([0,T), S([0,1]^2))$;
    \item\label{hyp:2} there exists a sequence $t_n\to T^-$ of times such that $\phi_{t_n}\rightarrow f$ strongly in $L^2$;
    \item\label{hyp:3} $\dot \phi_t\circ \phi_t^{-1}\in L^1_{\text{loc},t}([0,T), L^2([0,1]^2))$.
\end{enumerate}

We denote by $\mathcal{C}_{1,2}(Id,f)$ the set of flows satisfying hypotheses \ref{hyp:1},\ref{hyp:2} and \ref{hyp:3}. \footnote{We recall that the strong closure in $L^2$ of the group of volume-preserving diffeomorphisms is  
\begin{equation}\label{eq:closure}
    S([0,1]^2)=\overline{\text{SDiff}([0,1]^2)}^{\|\cdot\|_{L^2}}=\lbrace \varphi:[0,1]^2\rightarrow [0,1]^2\text{ measurable, area-preserving}\rbrace
\end{equation}
and has been proved in \cite{Shnirelman},\cite{Shnirelman2},\cite{BrenierGangbo} in any dimension $\nu\geq 2$.}

Our main result then reads as follows:

   \begin{thm}[Shnirelman's Inequality in $\nu=2$]\label{thm:main:intro}
    Let $\nu=2$, then there exists a positive constant $C=C(\nu)>0$ such that, for every $f\in\text{SDiff}([0,1]^2)$ there exists a flow $ \varphi_t\in \mathcal{C}_{1,2}(Id,f) \cap L^1_{\text{loc}}([0,T),\BV([0,1]^2))$ such that the following Shnirelman's Inequality holds
        \begin{equation}
        \mathcal{A}_{1,2}\lbrace\varphi_t\rbrace\leq C\|f-Id\|_{L^2([0,1]^2)}.
        \end{equation}
   \end{thm}
Notice that the constant depends only on the dimension $\nu=2$. As a consequence, the flow constructed $\dot\varphi_t\circ\varphi_t^{-1}$ lives in $L^1([0,T],L^2([0,1]^2))$.

This statement, which ultimately proves Inequality \eqref{eq:minimizer:up:to:constant}, is not in contradiction with the example constructed by Shnirelman and mentioned above, nor with the results on the diameter problem \cite{EliashbergRatiu}, being the flow constructed highly non regular. Indeed $\varphi_t$ is not in general a smooth diffeomorphism for all times, but only a measure-preerving map. Moreover, the statement is false if in addition we require to have $v=\dot\varphi_t\circ\varphi_t^{-1} \in L^1([0,1],\BV([0,1]^2)$ (without the loc) in general:  indeed we construct a velocity field in $L^1_t BV_x$ for any time interval $[0,T]$ for $T<1$. A more general version for this theorem with time variable in $L^q_t$ and space variable in $L^p_x$ with $q,p\in[1,+\infty]$ is also available (see Theorem \ref{thm:main:chapter:p:q}). Notice that, in view of the results \cite{Shnirelman},\cite{EliashbergRatiu} we might consider the result sharp. 
%{\color{red}Another interesting feature is that the constant $C$ in the right hand side does not blow up, while it seems possible to obtain a local version of the Shnirelman inequality with regular velocity fields, with constant $C$ eventually blowing up \cite{Patrickdiam}.} \textcolor{magenta}{I don't understand this sentence.} {it's a reference to Patrick work, I will put it in the definitive version once they also have their result out.}

\medskip

One of the main differences with the proof of Theorem \ref{thm:sharp:shnirelman:3d} is that in dimension $\nu\geq 3$ trajectories have codimension two, so that one can always 'twist' two trajectories around each other, allowing for fluid particles not to cross. This was one of the main ideas for the pipe-flow constructed in \cite{Schiffer_Zizza25}: indeed we can always arrange fluid particles to travel along pipes that do not intersect. While in dimension $d=2$ we should take into account that fluid particles may \emph{collide}. The second reason that makes the analysis different is that, in order to give the proof of the previous theorem, one needs to prove a discrete-type inequality for permutations of a tiling $\mathcal{R}_N$ of $[0,1]^\nu$ and then to apply a smoothing procedure in order to recover the inequality for volume-preserving diffeomorphisms \cite{Arnold:khesin},\cite{Shnirelman},\cite{Schiffer_Zizza25}. Unfortunately, this can be achieved only in dimension $\nu\geq 3$ because of purely topological obstructions. The main reference for this reasoning is in Lemma IV.7.19 in \cite{Arnold:khesin}.  This second obstacle is impossible to overcome in dimension $\nu=2$, and this is the main reason why the original Shnirelman's inequality is not available in $\nu=2$. While the purpose of this work is to show that the first obstacle can be rectified, giving surprising results that, in our opinion, go beyond the statement of the inequality.

Indeed, a compelling motivation for this study arises from the special case of two-dimensional ideal hydrodynamics, where the scalar vorticity field is advected along the geodesic path of volume-preserving diffeomorphisms. As predicted by Onsager \cite{Onsager49}, regions of equally signed vorticity tend to aggregate into large-scale vortex condensates. However, because vorticity is strictly advected rather than dissipated, this transition does not occur through diffusion. Instead, it proceeds via mixing—an intricate process of thinning and folding where large-scale structures emerge at the cost of increasing entanglement at infinitesimal scales.
\begin{figure}[H]
    \centering
    \includegraphics[scale=0.5]{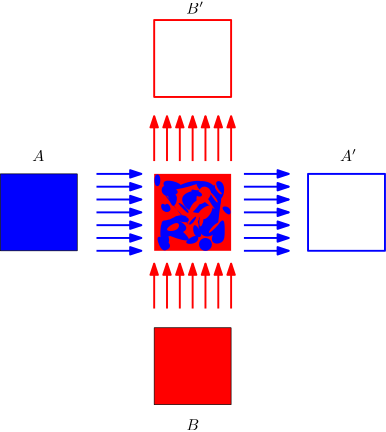}
    \caption{Two shear flows that merge and collide into a \emph{mixing} region in order to reach their final destination minimizing the total kinetic energy.}
    \label{fig:mixing:intro}
\end{figure}
In this work, we encounter a similar behavior when trying to solve the variational problem in dimension $\nu=2$. Assume, indeed, to have two fluid boxes $A$ (the blue one) and $B$ (the red one) such that they need to reach their final destinations $A'$ and $B'$, respectively, \emph{minimizing} the total kinetic energy (see Figure \ref{fig:mixing:intro}). If we impose the fluid particles to follow the shortest path (i.e. straight lines), then a possibility is that they merge and collide into an \emph{intersection region}. 

As we will see, this gives the formation of small scales structures where the gradient of the velocity field increases. Heuristically, when the fluid particles enter the intersection region, we lose the information about their starting position ($A$ or $B$), so that fluid particles are then free to move to their final destination. We point out that this behavior has been observed numerically in \cite{MerigotMirebeau2016} and we do not know if it is merely related to the two-point boundary value formulation of Euler Equations, or it can be used to partially explain the dynamics of Euler Equations in $2d$. 

We will analyse this behavior in great detail in the next section, where we will give the backbone of the proof of Theorem \ref{thm:main:intro}.

\medskip 
We point out that this version of the paper is intended to provide a readable presentation of the main ideas and results. We have therefore kept the technical details of the construction to a minimum and enlarged the heuristic explanations of the construction in Section \ref{Section:2}, while keeping the consistency of the results. 

\subsection{Plan of the paper} 
In Section \ref{Section:2} we describe the ideas and the proof of Theorem \ref{thm:main:intro}, that will be developed throughout the following sections. Section \ref{Section:3} is devoted to the discrete problem formulation, and the building blocks for the pipe-flow. In particular Lemma \ref{lem:transport:tot} is the two-dimensional version of the Lemmas of \cite{Schiffer_Zizza25}, while Lemma \ref{lemma:cornerII} is new and improves the one in \cite{Schiffer_Zizza25}. Finally, Lemma \ref{lem:swap} is a fundamental Lemma for the construction, and has been proved in \cite{Zizza24}. In Section \ref{sec:intersection} we analyze the dynamics inside an intersection region. We first construct the building block for the flow in Lemma \ref{lem:inter:flow}. Then in Proposition \ref{prop:inter:flow:subdivisions} we use the building block of the previous lemma at a finer scale. Finally Proposition \ref{prop:inter:flow:subdivisions:II} gives the general case of Proposition \ref{prop:inter:flow:subdivisions} without imposing any additional constraint on the size of the intersection region. In Section \ref{Section:4} we state the construction of the pipe-flow: a key difference is the introduction of interchange regions where the fluid particles may cross, while the argument is identical to the one in \cite{Schiffer_Zizza25} with small technical differences. 
Finally, in \ref{S:section:6} we use the results of Sections \ref{sec:intersection} and \ref{Section:4} to prove the main theorem. In the Appendix \ref{S:appendix} we put the main proof of the statements in \cite{Schiffer_Zizza25} that will be used throughout the paper.

\subsection*{Acknowledgments} The authors would like to thank Patrick Heslin and Giacomo del Nin for interesting discussions arising while writing this paper. They would also like to thank Theo Drivas, Gerard Misio{\l}ek and Patrick Heslin for insights that improved the presentation of this manuscript. 

\subsection*{AI declaration} The mathematical work in this paper was carried out by the authors without the use of AI.

\section{Dynamics of two-dimensional motion: a two-point problem approach}\label{Section:2}

We devote this section to the explanation of the proof of Theorem \ref{thm:main:intro} and to connections to some open problems in the field. As the strategy of proof is rather similar to the higher dimensional case in \cite{Schiffer_Zizza25}, there is definitely overlap with the construction of \cite{Schiffer_Zizza25}. For the sake of completeness we still explain the most important steps and only refer to the aforementioned work for details that we do not prove again.
\subsection{Strategy of Proof}
\begin{figure}[H]
    \centering
    \includegraphics[width=0.5\linewidth]{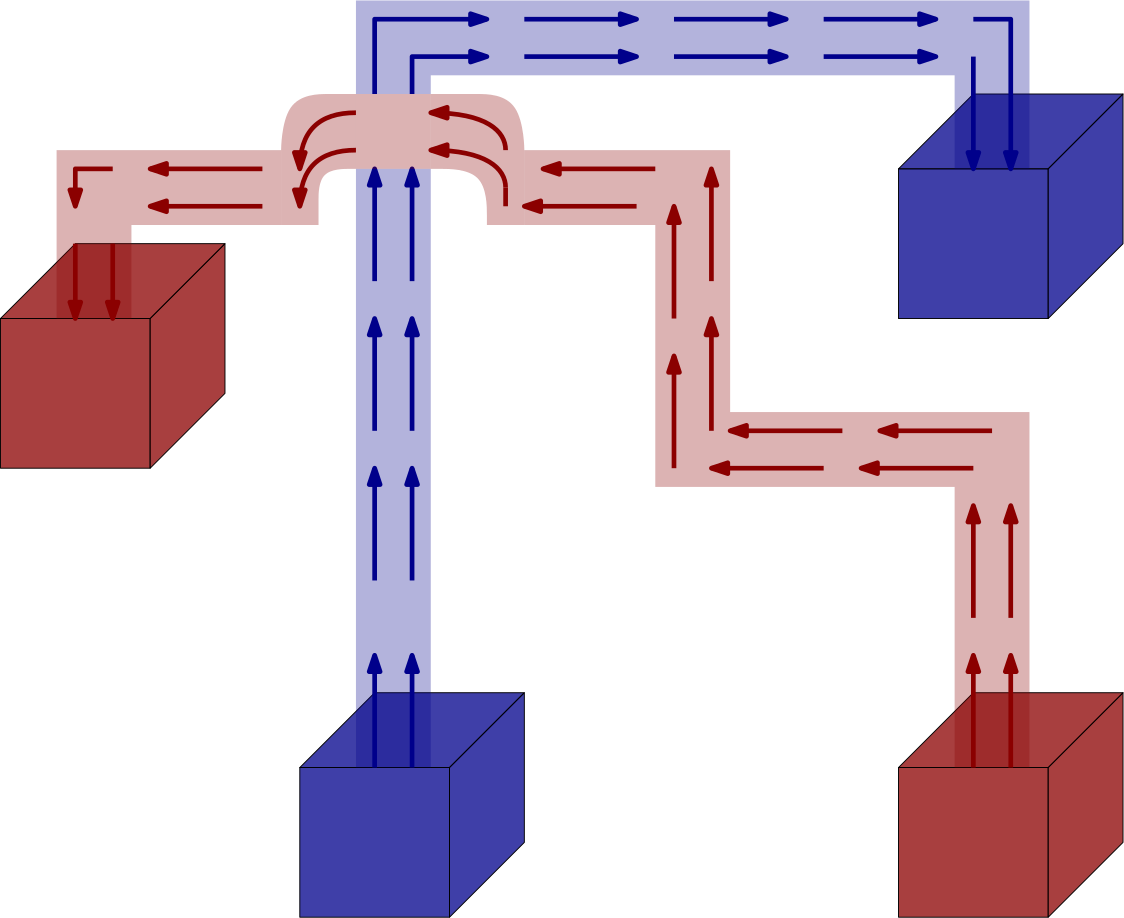}
    \caption{The basic idea of constructing the flow with pipes. The cubes are connected with these tubes, in which the velocity field is stationary. Note that in three or  more dimensions, the pipes essentially have codimension two or higher, so in case that there is an intersection in a baseline construction, we can "twist" the trajectories around each other (by building over- and underpasses). This is not possible any longer in dimension 2. }
    \label{fig:3dPipes}
\end{figure}
Consider a tiling $\mathcal{R}_N$ of $[0,1]^2$ into $N^2$ squares of sidelength $N^{-1}$. By a theorem of Lax \cite{Lax}, we can approximate any fluid configuration (volume-preserving diffeomorphism) via a permutation of subcubes of the tiling in any $L^p$-norm.

To find a flow map connecting one fluid configuration (or, in the discrete case, a permutation of subcubes) to another, two different strategies are generally imaginable: the first strategy is based on the Lagrangian viewpoint, always keeping track of the permutation. This approach is chosen in \cite{Shnirelman,Zizza24} where the final permutation is realised as the discrete flow of permutations (or \emph{elementary movements}), each one swapping couples of adjacent cubes at a discrete time step. As a consequence, the main difficulty here lies in the estimate of the minimum number of elementary movements, which directly links the problem to the combinatorial estimate of the diameter of the Cayley Graph associated to transposition trees \cite{Zizza24}. 

In contrast, in \cite{Schiffer_Zizza25} a different approach is chosen that leans more towards the Eulerian formulation. In particular, we constructed a (stationary) vector field $v$ and showed that the solution to 
\[
\partial_t \phi(t,x) = v(\phi(t,x))
\]
at $t=1$ is exactly our desired configuration. Roughly speaking, this can be achieved through sending fluid particles from cube to cube via \emph{pipes} (cf. Figure \ref{fig:3dPipes}).

In three or higher dimensions those pipes essentially have codimension at least two. Consequently, we can arrange them so that they \emph{do not} need to intersect. This is different in dimension two and, in order to achieve a result, we need a more elaborate construction.

As a consequence, our construction is more of a mixture of Lagrangian and Eulerian viewpoint. The baseline stays the same as in \cite{Schiffer_Zizza25}: we construct pipes connecting a subcube $Q$ to its image $\phi(Q)$. Unfortunately, pipes coming from different cubes need to intersect. The essential idea is now to split up the flow into two phases: In the first phase (of time step $s>0$) we correct the flow
that helps us deal with intersections of pipes. In the second phase, again performed in time step $s>0$, we follow a regular "pipe flow". In the end, after a time step $\tau>0$ (which will be clear from the construction), the flow behaves as if we had a non-intersecting pipe flow. Repeating this procedure all over leads to the correct result for time $t=1$.

\medskip 

Going a little bit more into detail, given a permutation $P:M\rightarrow M$ on the squares of the tiling, we aim to construct an incompressible velocity field whose flow map at time $t=1$ is precisely $P$. 
\begin{itemize}
    \item \textbf{Decomposition of the motion on the pipe-flow}: we observe that the motion can be conveniently thought as occurring into pipes: fluid particles travel along thick pipes and the thickness of pipes is not constant along the motion. This decomposition transforms the continuous analysis of the incompressible flow into a discrete problem on a graph. Each pipe represents a bundle of fluid trajectories, and its volume flux must be conserved despite changes in its cross-sectional area (thickness), rigorously maintaining the incompressibility condition. This continuous path is then mapped onto the discrete network, which is a graph $G=(V,E)$ where the set of vertices $V$ is given by the cubes of the tiling $\mathcal R_N$. Two vertices $v, w \in V$ are connected by an edge $e=\lbrace v,w\rbrace \in E$ if their corresponding cubes are \emph{adjacent} (i.e., they share a face) on the tiling. The continuous motion of a fluid particle through a pipe is thereby translated into a \emph{path} $\gamma$ on the graph $G$.
    
    \medskip 
    \item \textbf{Discrete Setup}: we formulate Shnirelman's inequality in the discrete network and we setup a discrete problem on a graph, cf. Subsection \ref{Ss:discrete:problem}.
    The discrete flow is established by assigning a positive integer weight or capacity (the \emph{thickness parameter} $\rho(e,\gamma)$) to the couple $(e,\gamma)$, representing the volume flux passing between the adjacent cubes along the edge $e$. The incompressibility condition for the continuous flow translates into a conservation law at every edge $e\in E$, namely $\rho(e,\gamma)\cdot u(e,\gamma)=1$, where $u(e,\gamma)$ denotes the constant velocity of the pipe along the edge $e$. We also ask that
    \begin{itemize}
        \item \textbf{Capacity constraint}: the sum of thicknesses of the pipes $\gamma$ visiting an edge $e$ is bounded (since we want our flow to stay confined in the square $M$);
        \item \textbf{Time Constraint}: the fluid particles travel at most in time $t=1$ along the pipe $\gamma$. 
    \end{itemize}
    \smallskip

    The sharp Shnirelman's inequality is now recast as a combinatorial optimization problem. The objective is to find a set of paths in the graph, corresponding to the fluid trajectories, satisfying the capacity and the time constraint. 

    \medskip 

    \item \textbf{Building block for the pipe flow}: we build up the main tools for translating the discrete problem into the pipe-flow continuous formulation. In particular, we prove that the motion is a concatenation of \emph{source-sink} flows: the computations for the velocity field in each source-sink flow is explicit (Section \ref{Section:3}, Lemmas \ref{lem:transport:tot},\ref{lemma:corner:1},\ref{lemma:cornerII}).  
\end{itemize}
After this, following the procedure of \cite{Schiffer_Zizza25}, we arrive at a picture resembling Figure \ref{fig:2dPipes}: Any subcubes $Q$ and $P(Q)$ are connected via a velocity field $v_Q$; but for different cubes Q those
velocity fields might intersect as in Figure \ref{fig:2dPipes}. One therefore has to ensure the following:
  \medskip 
\begin{itemize}
    \item \textbf{Alignment of pipes:} First, we need to acknowledge that the basic pipes that we construct need to intersect, but only under very specific circumstances. To do this, it is possible to introduce a (lexicographic) ordering of the pipes and afterwards show that if two pipes intersect, they do orthogonally (Section \ref{Section:4}). This construction has its natural counterpart in \cite{Schiffer_Zizza25}, where we need to "bend" fluid flows around each other.
    
 %   \st{\textbf{Continuous construction}: in the continuous construction we have to take into account that pipes may intersect, differently from the $3$-dimensional case. We first introduce a lexicographic order and then we prove that, if two pipes intersect, they intersect orthogonally (Section \ref{Section:4}). This part has to been taken into account more carefully than the one in $\nu\geq 3$ \cite{Schiffer_Zizza25} because of the topological obstructions caused by the two-dimensional topology (Figure \ref{fig:interchange:region}).}\label{intersection}

    \medskip 

    \item \textbf{Intersection of two pipes}: Having arrived at the very specific configuration, at which two pipes intersect, we adapt our construction, so that we obtain a "classical" flow that, seen over a time step $\tau$, achieves the same as if we could just ignore the intersection of trajactories. This point is further explained in the Subsection below and is the content of Section \ref{sec:intersection}.
    
\end{itemize}    

% \begin{figure}
%     \centering
%     \includegraphics[width=0.5\linewidth]{interchange.png}
%     \caption{A look inside an interchange region: thanks to the lexicographic order, we can design a flow in the interchange region in such a way that every couple of flows intersects at most twice in a perpendicular way.}
%     \label{fig:interchange:region}
% \end{figure}

\subsection{Intersection of flows and emergence of probabilistic behavior} 
\begin{figure}[H]
    \centering
    \includegraphics[width=0.5\linewidth]{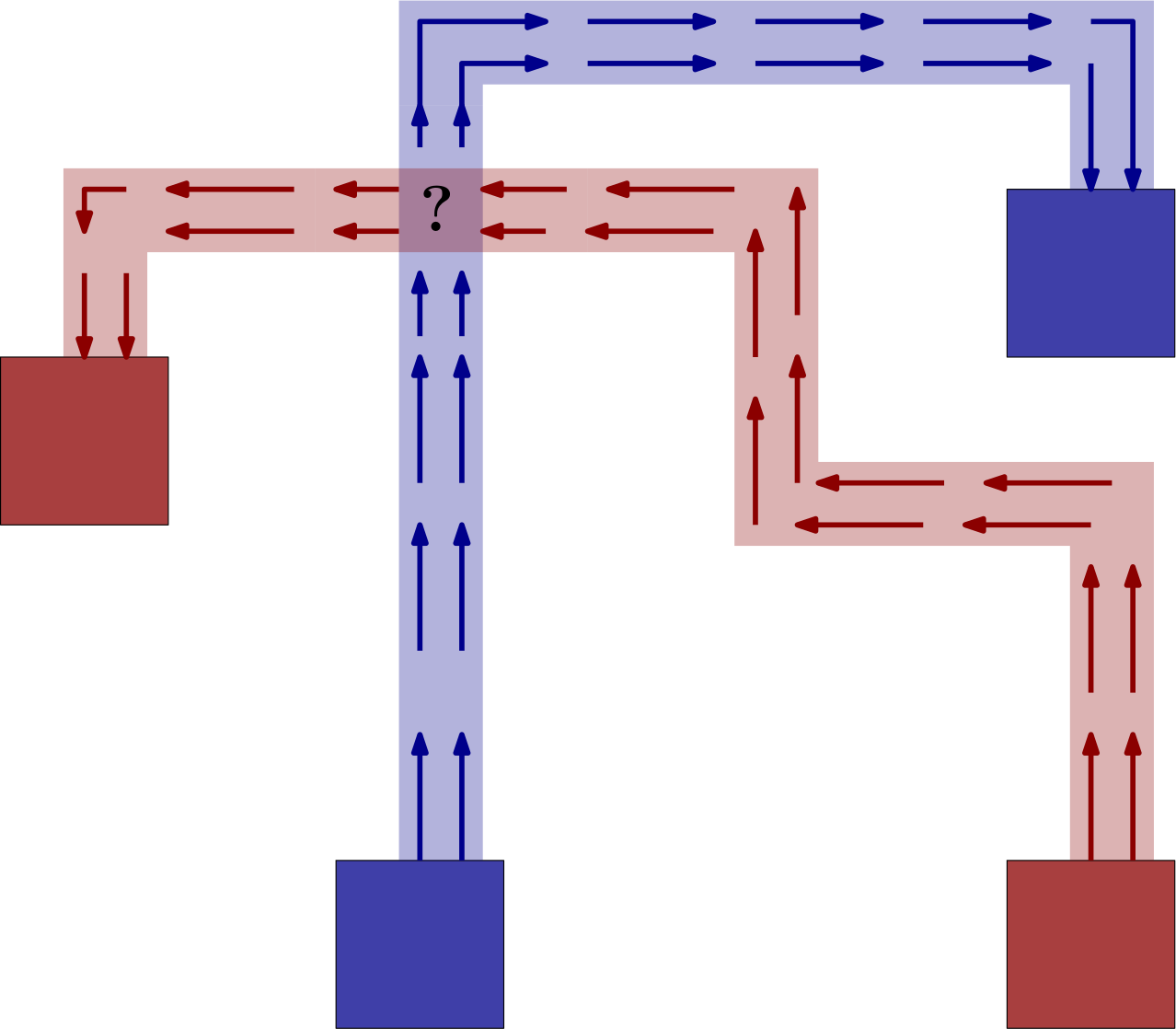}
    \caption{The same baseline construction is used compared to the 3D case. However, due to the present dimension, the trajectories forming the pipes need to intersect. Therefore, one needs to figure out how to arrange the flow in the intersection region.}
    \label{fig:2dPipes}
\end{figure}
As already mentioned, the main novelty of the present paper is to deal with regions where the two pipe-flows intersect (cf. Figure \ref{fig:2dPipes}). 

We explain here the heuristic of our result, that led to the formulation of Theorem \ref{thm:main:intro} and further connections with the theory of Generalized Incompressible Flows \cite{Brenier1},\cite{Brenier99} and Shnirelman's two-point conjecture.

\medskip 
Assume that an horizontal pipe intersects a vertical pipe, as in Figure \ref{fig:traffic}. In the pipes, the motion is given by horizontal (resp. vertical) shear flows. We want particles to reach their final destination possibly following \emph{the shortest path}, in order to minimize the amount of the total kinetic energy. In order to reach their final position, the fluid particles can  
\begin{enumerate}
    \item wait for some time (you have to think like a traffic light in order to avoid collisions between fluid particles, Figure \ref{fig:traffic});
    \item go straight and collide.
\end{enumerate}
In the first hypothesis, if fluid particles wait for some amount of time $T>0$, since they can potentially intersect $N$ pipes (recall that $N\in\N$ is the size of the tiling), then, as $N\to\infty$ the cost functional $ \|v\|_{L^2_tL^2_x}\to \infty$: the amount of time to reach the final destination would be $\mathcal{O}(NT)$. The second possibility is that fluid particles collide (or \emph{merge}): see Figure \ref{fig:Intersections:configurations}.

\begin{figure}
    \centering
    \includegraphics[width=0.5\linewidth]{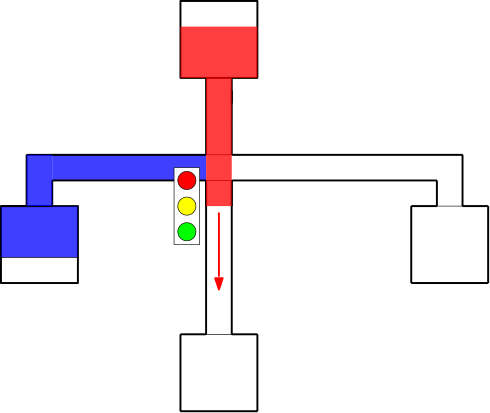}
    \caption{If the blue fluid particles wait at every crossing, the cost for reaching their final destination tends to infinity.}
    \label{fig:traffic}
\end{figure}
\begin{figure}[!tbp]
  \centering
  \begin{minipage}[b]{0.4\textwidth}
    \includegraphics[width=\textwidth]{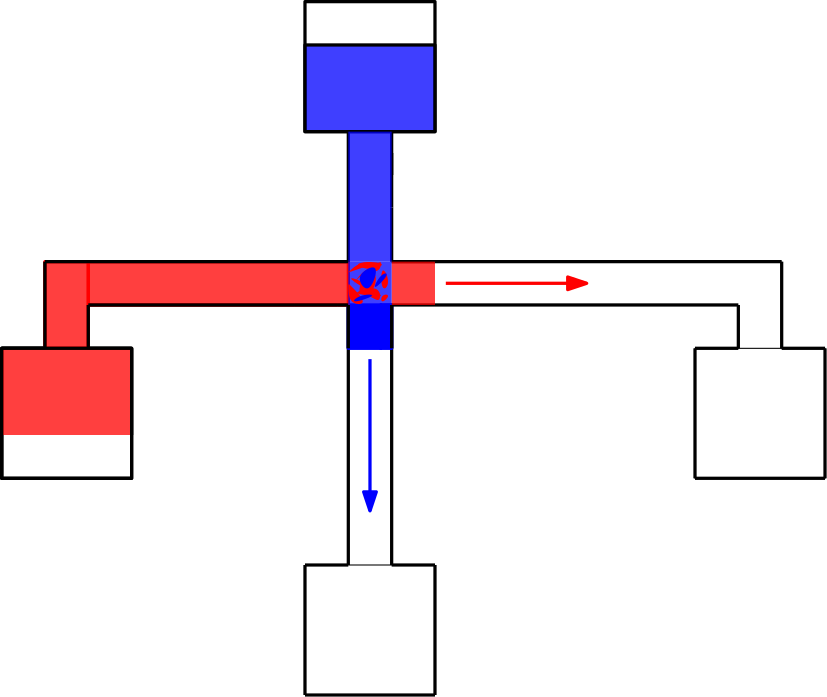}
  \end{minipage}
  \hfill
  \begin{minipage}[b]{0.4\textwidth}
    \includegraphics[width=\textwidth]{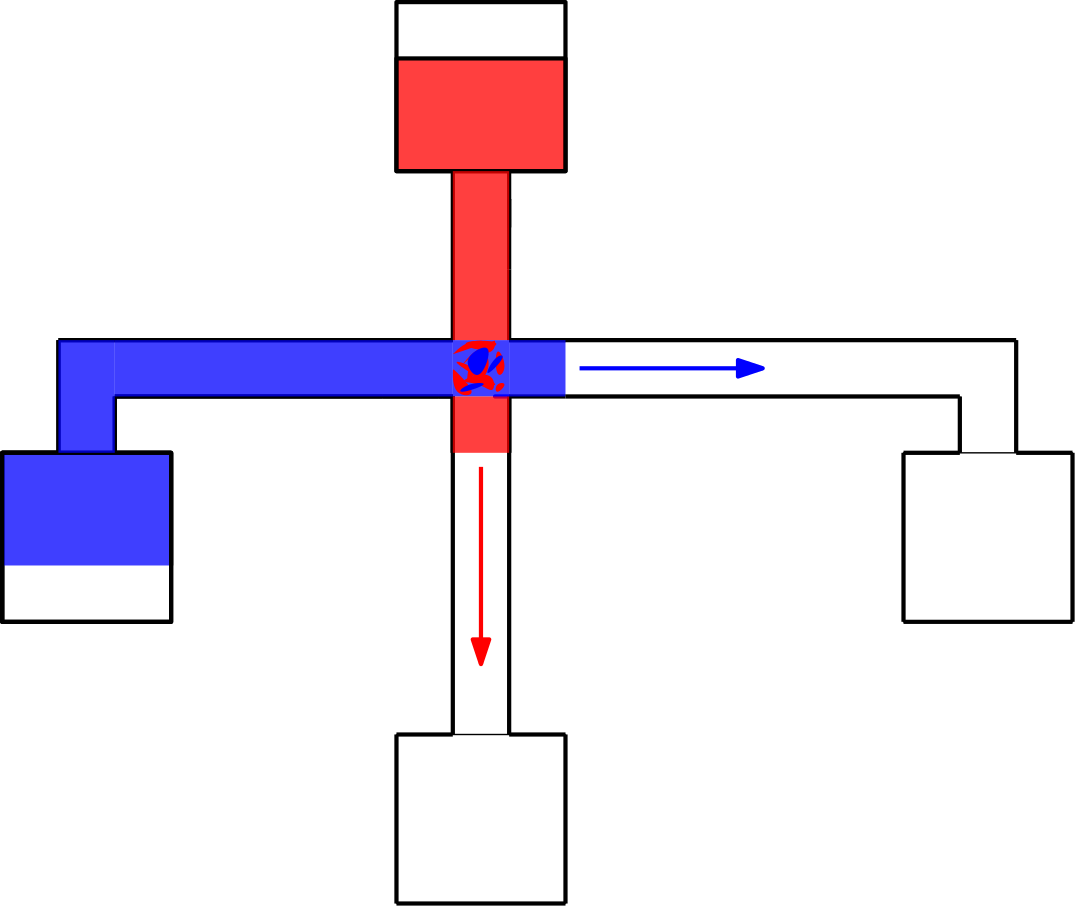}
     \end{minipage}
  \caption{Here we have two different fluid configurations: in the first one blue fluid particles are traveling vertically, while in the second one they are traveling horizontally. We assume now that in the intersection region the fluid particles are mixed. If we look at the intersection region only, then we can not distinguish if we are in the first or in the second configuration (irreversibility).}
  \label{fig:Intersections:configurations}
\end{figure}

We picture the collision area as a \emph{mixing region}: indeed, if in the intersection region we cannot distinguish the blue particles and the red particles (i.e.: they are mixed) then we can let them travel straight to their final destination. Somehow, in the intersection region we lose the information about the initial fluid configuration, observing an \emph{irreversible} behavior.

\subsection{Summary of the proof ideas}

We first analyze what happens in the case of a single intersection region and subsequently we adjust the construction to deal with multiple intersections simultaneously.

\subsubsection{Single intersection}The content of this subsection is the proof of Lemma \ref{lem:inter:flow}, Section \ref{sec:intersection}.

Consider two pipes intersecting perpendicularly. We first assume that the two pipes have identical thickness $a>0$ (the general case is then discussed in Proposition \ref{prop:inter:flow:subdivisions:II}). In order to realize the permutation,  we need to transfer the fluid volume contained in the region $F_1$ into the region $G_2$ and the fluid volume in the region $G_1$ into the region $F_2$, starting from the configuration of Figure \ref{fig:swap:intro:1}. We also assume that at the boundary of the region our velocity field is $v_0$. We perform this procedure in two steps:
\begin{enumerate}
    \item We swap the fluid of the two regions $F_1$ and $G_1$ in the same amount of time $\tau$, proportional to $\tau\sim a/v_0$ (Lemma \ref{lem:swap}), Figure \ref{fig:swap:intro:2}.
    \item We move the fluid along the pipe. The construction is arranged so that the quantity of fluid in a cube is identical to the quantity of fluid in the corner region, Figure \ref{fig:swap:intro:3}. In particular, this step is performed in the same amount of time $\tau$. Notice that the region $F_1$ is now filled with red fluid and the region $G_1$ with blue fluid.
    \end{enumerate}
We repeat the previous two steps: Figures \ref{fig:swap:intro:4}, \ref{fig:swap:intro:5}. After an amount of time $4\tau$ we have achieved our goal.
\begin{figure}[H]
     \centering
     % Prima immagine
     \begin{subfigure}[b]{0.25\textwidth}
         \centering
         \includegraphics[width=\textwidth]{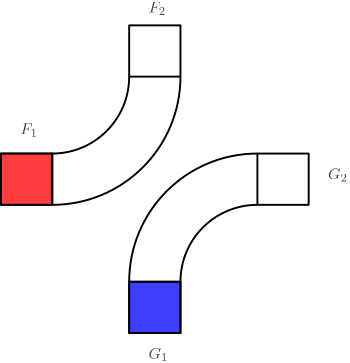}
         \caption{We want to move the blue fluid in the region $G_1$ into the region $F_2$ and the red fluid in $F_1$ into $G_2$.}
         \label{fig:swap:intro:1}
     \end{subfigure}
    \hfill
     \begin{subfigure}[b]{0.25\textwidth}
         \centering
         \includegraphics[width=\textwidth]{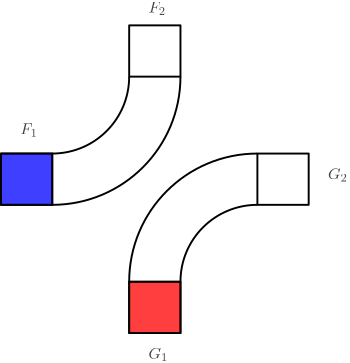}
         \caption{We swap the content of the region $F_1$ with the region $G_1$. This is performed in time $\tau\sim a/v_0$, thanks to Lemma \ref{lem:swap}.}
         \label{fig:swap:intro:2}
     \end{subfigure}
     \hfill % Spazio elastico tra le colonne
     % Terza immagine
     \begin{subfigure}[b]{0.25\textwidth}
         \centering
         \includegraphics[scale=0.3]{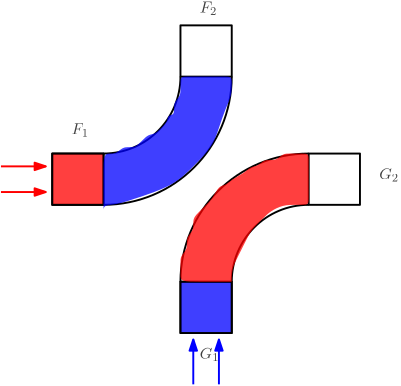}
         \caption{We move the flow along the pipe. This is performed again in time $\tau$. Notice that now in $F_1$ we have again red particles and in $G_1$ blue particles.}
         \label{fig:swap:intro:3}
     \end{subfigure}
     \caption{The first two steps of the construction in Lemma \ref{lem:inter:flow}.}
     
\end{figure}
\begin{figure}[H]
     \centering
     % Prima immagine
     \begin{subfigure}[b]{0.3\textwidth}
         \centering
         \includegraphics[width=\textwidth]{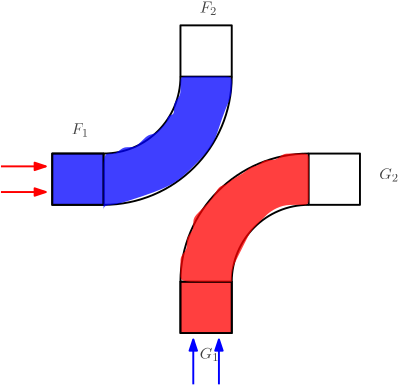}
         \caption{We swap again the content of the region $F_1$ with the content of region $G_1$.}
         \label{fig:swap:intro:4}
     \end{subfigure}
     \hfill % Spazio elastico tra le colonne
     % Seconda immagine
     \begin{subfigure}[b]{0.3\textwidth}
         \centering
         \includegraphics[width=\textwidth]{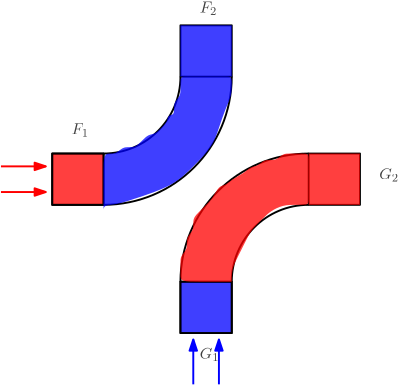}
         \caption{We move again the flow in the pipe and we obtain the desired configuration.}
         \label{fig:swap:intro:5}
     \end{subfigure}
     \caption{How the flow looks like in the case of a single intersection. Notice that, even if we \emph{imagine} an intersection region, the pipes do not really intersect. Notice that the amount of fluid in the square should be the same quantity of fluid in the \emph{corner} area.}
     \label{fig:swap:intro:11}
\end{figure}

When we compute the $L^\infty_tL^p_x$-estimate of the velocity field constructed at this step, it is, up to a constant independent of the construction, identical to the value of the $L^\infty_tL^p_x$-norm of the velocity field as if the two pipes do not intersect, that is 

\begin{equation}\label{eq:cost:intro}
    \|u\|_{L^p}\quad \sim \quad\text{ (Area of the intersection region)}^{\frac{1}{p}}\cdot\text{|v| }\sim\quad  a^{\frac{2}{p}}v_0.
\end{equation}
Indeed, by Lemma \ref{lem:swap}, it holds
\begin{equation}\label{eq:cost:intro:2}
    \text{ cost of the swapping velocity field}\sim \quad  \frac{1}{\tau}a^{1+\frac{2}{p}}=a^{\frac 2p}{v_0}.
\end{equation}
 We can explain the previous quantity as follows: the $L^p$-norm of a velocity field that swaps two cubes in time $1$ of volume \emph{$\text{vol}$} and at a distance \emph{$\text{dist}$} is precisely \emph{$\text{vol}^{\frac{1}{p}}\cdot\text{dist}$} (this has also been observed in \cite{Zizza24} and it has been used as starting point for the analysis carried out there).

\subsection{Heuristic: swaps versus small scale creation}\label{Ss:heuristics}
\begin{figure}[H]
    \centering
    \includegraphics[scale=0.45]{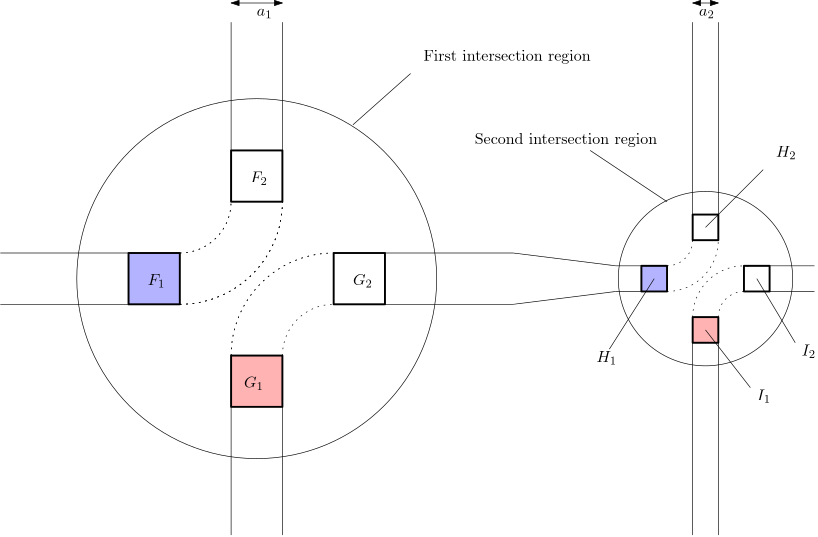}
    \caption{A single pipe might intersect multiple pipes.}
    \label{fig:double:intersection}
\end{figure}

Unfortunately, the construction above is no longer viable when a pipe traverses multiple intersection regions with different thicknesses. Indeed, consider Figure \ref{fig:double:intersection}. Applying the reasoning of the previous subsection to the particular case of two (or more) intersections, we would define the parameters $\tau_1=a_1/v_{0,1}$ and $\tau_2=a_2/v_{0,2}.$ However, since each pipe is connected, we must swap $F_1$ with $G_1$ and $H_1$ with $I_1$ in the same amount of time $s>0$, and then move the flow along the pipe again in time $s>0$ (the second step in the previous construction).
Inspired by the heuristics of Figure \ref{fig:Intersections:configurations}, which dictates the behavior we expect, we hypothesize the creation of smaller scales, in order to find a new time-parameter $s$ independent of the intersection considered.

\begin{figure}[H]
    \centering
    \includegraphics[scale=0.4]{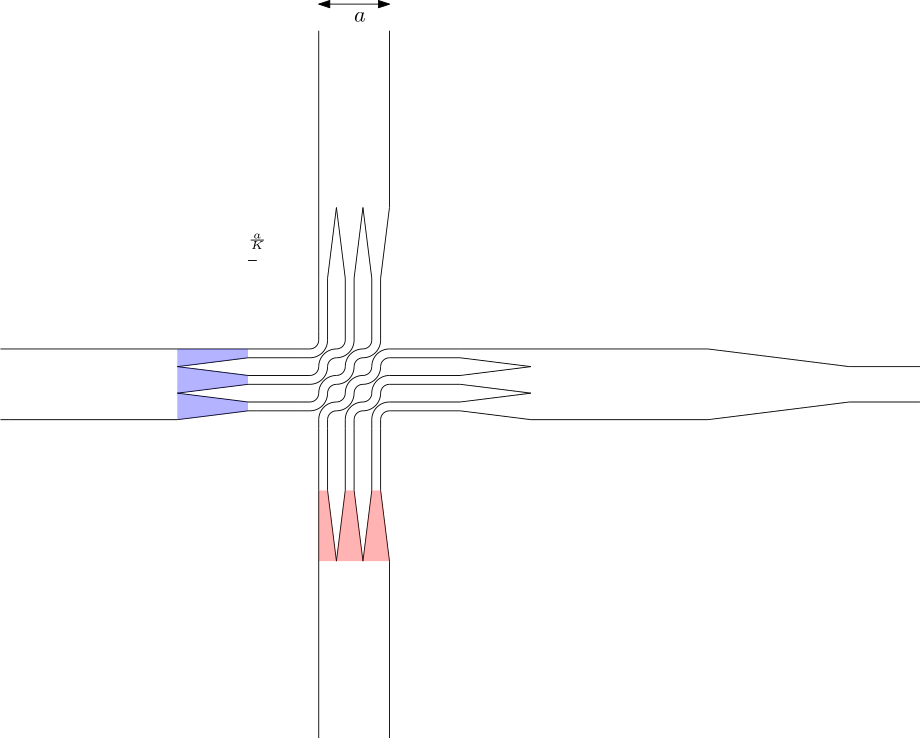}
    \caption{We subdivide each pipe of size $a$ into $\bar K$ subpipes and we iterate in each subpipe the steps performed in Figure \ref{fig:swap:intro:1}.}
    \label{fig:swap:sintro:2}
\end{figure}
\begin{figure}[H]
    \centering
    \includegraphics[scale=0.5]{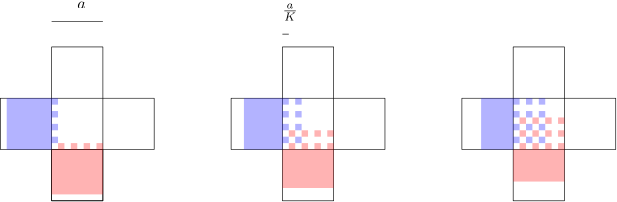}
    \caption{How the fluid flow looks like in an intersection region.}
    \label{fig:swap:small:scale}
\end{figure}
We therefore introduce a new construction parameter $\bar{K} = \bar{K}(a)$ (see Propositions \ref{prop:inter:flow:subdivisions} and \ref{prop:inter:flow:subdivisions:II}), representing the number of subpipes of size $a/\bar{K}$ into which the flow has been subdivided (Figures \ref{fig:swap:sintro:2} and \ref{fig:swap:small:scale}). In each subpipe, we apply the steps described in the previous subsection (Lemma \ref{lem:inter:flow}) $\bar{K}$ times. Notice that, starting with an initial velocity $v_0$, the time parameter scales as 
\begin{equation}
    \text{time parameter } s \sim \frac{\text{thickness}}{\text{velocity}} = \frac{a}{\bar{K} v_0}.
\end{equation}
In particular, by tuning the new parameter $\bar{K}$, $s$ can be chosen independently of the quantities $a$ and $v_0$. 

Computing the total cost of the new velocity field, we observe that 
\begin{equation}\label{eq:cost:intro:3}
    \text{cost of the swapping velocity field at small scales} \sim \left(\frac{K}{\tau}\right) \cdot \frac{a}{K} \cdot a^{\frac{2}{p}} = a^{\frac{2}{p}} v_0;
\end{equation}
indeed, the only change lies in the scale of the velocity field, while the volume and time estimates remain invariant under this procedure. Notably, this quantity coincides precisely with the one computed previously, corresponding to the $L^\infty_t L^p_x$-norm estimate of the velocity field \emph{without} intersections.

\medskip 

Observe that, in the limit $s \to 0$ (corresponding to the grid size $N\to\infty$), the velocity field formally becomes a Young measure. This limit represents a transition from deterministic fluid motion to a purely probabilistic regime where, instead of discrete subpipes, one considers particle-by-particle motion (i.e. generalized flows, \cite{Brenier1}). However, we choose not to investigate this regime further, as it lies beyond the scope of the present paper.

\medskip 
To summarize, in order to solve a two-point problem in the plane (up to a constant, see Theorem \ref{thm:main:intro}), we assume that fluid particles follow \emph{the shortest path}. This assumption leads to collisions, which are observed within a mixing region. Such behavior naturally raises the following questions: 
\emph{Is this a phenomenon observable in Euler motion (i.e., among minimizers of the action functional), or is it merely a mathematical artifact?}

\subsection{Open Problems and relations with the existing literature}

The preceding discussion is heuristic, yet we hope it may provide further insight into the dynamics of $2d$ incompressible Euler motion. In this section, we briefly outline several open problems in the field, though the following list is by no means exhaustive.

\subsubsection{Shnirelman's two-point conjecture}
We recall a long-standing question in the field, namely Shnirelman's two-point conjecture, which we state here in two versions:

\begin{conjecture}[Shortest-path two-point conjecture]\label{conj:1}
    Let $f \in \text{SDiff}([0,1]^2)$ and assume there exists a path $\{\phi_t\} \subset \text{SDiff}([0,1]^2)$ connecting $id$ to $f$ such that $\mathcal{A}\{\phi_t\} < +\infty$. Then, there exists a minimizer $\{\bar{\phi}_t\}$ connecting $id$ to $f$ such that
    \begin{equation}
        \min_{\phi_t: \phi_0=id, \phi_1=f} \mathcal{A}\{\phi_t\} = \mathcal{A}\{\bar{\phi}_t\}. 
    \end{equation}
\end{conjecture}

\begin{conjecture}[Perfect fluid two-point conjecture]\label{conj:2}
    Let $f \in \text{SDiff}([0,1]^2)$ and assume there exists $\{\phi_t\} \subset \text{SDiff}([0,1]^2)$ connecting $id$ to $f$ with $\mathcal{A}\{\phi_t\} < +\infty$. Then, there exists a smooth solution $u: [0,1] \times [0,1]^2 \to \mathbb{R}^2$ of the Euler equations such that $\phi^u_{t=1} = f$, where $\phi^u$ denotes the flow of the vector field $u$. 
\end{conjecture}

Our findings suggest that whenever fluid particles follow the shortest path (a straight line) to reach their destination, the resulting velocity field may no longer be deterministic, potentially disproving Conjecture \ref{conj:1}. However, this perspective is perhaps too simplistic and necessitates further investigation. Consider, for instance, the rotation $R_\pi$ on the solid disk $B_1(0)$. If we consider the straight lines connecting each particle $x$ to $R_\pi(x)$, as prescribed by our relaxed model, they all intersect at the center of the ball. Nevertheless, a well-known result by Brenier \cite{Brenier1} proves that the rotation flow, where particles move along circular arcs rather than straight lines, is indeed a minimizer of the action functional. See also \cite{AF} for open questions about gap phenomena of the metric.

\subsubsection{Connections to Generalized Flows}
Regarding the configurations in Figure \ref{fig:Intersections:configurations}, we believe that generalized least-action principles, such as those formulated by Brenier \cite{Brenier1}, could be of particular relevance. A primary question is how the irreversible behavior observed in Theorem \ref{thm:main:intro} relates to Generalized Incompressible Flows \cite{Brenier1}. Secondly, it remains an open question to understand whether this behavior originates from the two-point boundary value formulation or if it is also expected in the Cauchy problem formulation of Euler. 

\subsubsection{Inequalities for Discrete Fluid Configurations}

An inequality similar to the one in Theorem \ref{thm:main:intro} can be formulated by considering discrete approximations instead of smooth configurations, as discussed at the beginning of this chapter. Specifically, we consider a tiling $\mathcal{R}_N$ of the square into $N^2$ subsquares of sidelength $N^{-1}$, where configurations are permutations $P$ of these squares.

Building on this discretization approach, Shnirelman \cite{Shnirelman} proved a discrete version of the aforementioned inequality with a H\"older exponent $\alpha = 1/64$, a result subsequently improved in \cite{Zizza24}. We state the latter for completeness:

\begin{thm}[Theorem 1.4 \cite{Zizza24}]
    Let $\mathcal{R}_N$ be a tiling of $M=[0,1]^2$. For any target permutation $P$ of the tiles in $\mathcal{R}_N$, there exists a constant $C > 0$ (independent of $N$) and a divergence-free vector field $v \in L^\infty([0,1], \text{BV}(M)) \cap L^2([0,1], L^2(M))$ whose flow $\varphi_t^v$ connects $id$ to $P$ at time $t=1$, satisfying:
    \begin{equation}
        \|v\|_{L^2_t L^2_x} \leq C \|P - Id\|_{L^2_x}^{2/7}.
    \end{equation}
\end{thm}

The results in \cite{Shnirelman} and \cite{Zizza24} rely on approximating fluid motion via the swapping of adjacent cubes. It remains an open question whether the sharp H\"older exponent can be obtained through swaps of adjacent cubes (bridging the theory of permutons \cite{PermutonsDauvergne},\cite{PermutonsVirag} and sorting algorithms with incompressible fluid motion) or it is a purely fluid dynamics feature.

%% file: sec2_new.tex
\section{Preliminaries}\label{Section:3}

    \subsection*{General Notation}
Below we give a non-exhaustive overview over frequently used notation; some of it will be introduced in further detail below.

\subsubsection*{The square and its discretization}
\begin{itemize}
    \item $M=[0,1]^2$ is the reference domain;
    \item $\Dcal(M)=\text{SDiff}(M)$ is the space of volume-preserving diffeomorphisms on $M$;
    \item $\mathcal{R}_N$ is a tiling of $M$ into $N^2$ cubes of sidelength $N^{-1}$;
     \item $\Dcal_N$ denotes the set of permutations of $\mathcal{R}_N$;
     \item $P(M)$ denotes the set of permutations, i.e. $P(M)=\cup_{N=1}^\infty \Dcal_N$.
     \item for a multiindex $v \in \{0,1,\ldots,N-1\}^{2}$ denote by $Q_v= N^{-1} v + N^{-1} [0,1)^{2}$ the cubes of $\Rcal_N$.

\end{itemize}
\subsubsection*{Discrete Mathematics}
    \begin{itemize}
    \item $G=(V,E)$ denotes an (undirected) graph, where $V$ is the set of vertices and $E$ the set of edges;
    \item for our purpose on the grid, $V=\{0,1,2,\ldots,N-1\}^{2}$ is the set of vertices;
    \item in that instance, $e = \{v,w\} \in E$ if and only if $\vert v- w \vert=1$;
    \item $\mathfrak S_{N,2}$ denotes the symmetric group over the set $V$, i.e. the set of bijective maps $V \to V$;
    \item $\sigma$ denotes the generic element of $\mathfrak S_{N,2}$;
    \item a path $\gamma$ is a sequence of disjoint vertices $v_1, \ldots,v_k$ such that $v_i$ and $v_{i+1}$ are connected by an edge;
    \item $\Gamma$ denotes a set of paths;
    \item for a path $\gamma$, $E(\gamma)$ denotes the set of edges belonging to a path $\gamma$ and $V(\gamma)$ the set of vertices in $\gamma$.
    \end{itemize}
\subsubsection*{The discrete problem and its continuous counterpart}
\begin{itemize}
    \item $\psi_{\sigma} (x)= x+ N^{-1}(\sigma(v) -  v) \quad \text{if } x \in \q_v,~v \in V$, and $\sigma \in \mathfrak S_{N,2}$;
    \item for an edge $e$ and a path $\gamma$ the 'thickness' of a pipe is denoted by  $\rho(e,\gamma)$;
    \item $u(e,\gamma)$ is the velocity of the fluid through a pipe;
   % \item $\omega(\gamma)\in (0,1]$ is the weight of a path;
    \item $q_v$ is a subcube of $Q_v$ of sidelength $\ell$.
\end{itemize}
\subsubsection*{Flows of divergence-free velocity fields}
\begin{itemize}
    \item Without loss of generality we can always assume $T=1$;
    \item If $v\in L^q([0,1],L^p([0,1]^2))\cap L^1([0,1],\text{BV}([0,1]^2))$ divergence-free (here we mean \emph{distributional divergence}), the map $\phi^v_t:M\rightarrow M$ will denote the Regular Lagrangian Flow of $v$.
\end{itemize}
\subsubsection*{Parameters and relevant constants}
\begin{itemize}
    \item The constant $C>0$ denotes a generic constant, whose value is not interesting to us. The value of such a constant may change from line to line;
    \item the constant $C_{\text{swap}}$ is the constant in the swap flow, see Lemma \ref{lem:swap}.
\end{itemize}

The content of this section is rather similar to the corresponding section of \cite{Schiffer_Zizza25}. In order to give a self-contained exposition, in the following we restate the main results. 

\medskip 

\begin{figure}
    \centering
    \includegraphics[width=1.0\linewidth]{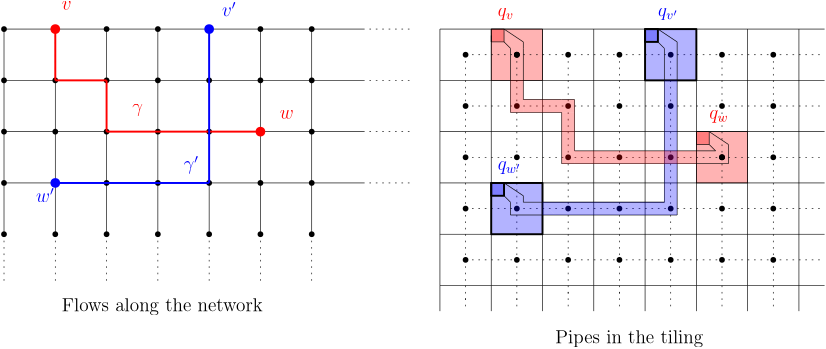}
    \caption{On the left side of the picture, we have two paths $\gamma,\gamma'$ connecting $v,w$ and $v',w'$ respectively. On the right side we have the corresponding pipe flow on the tiling $\mathcal{R}_N$. The pipes connect the square $q_v$ to $q_w$ and the square $q_{v'}$ to $q_{w'}$. Observe that the subsquares $q_v,q_v'$ are first connected to the network, then the pipes follow the path of $\gamma$ and $\gamma'$ along the network, and finally the pipes reconnect to the subsquares $q_w$ and $q_{w'}$. We also remark that the squares $q_v,q_{v'},q_w,q_{w'}$ have sidelength $\ell <N^{-1}$, while the sidelength of the tiling is $N^{-1}$. Moreover, notice that in dimension $\nu=2$ pipes are forced to intersect.}
    \label{fig:network:pipes}
\end{figure}

With a reference to Figure \ref{fig:network:pipes} we first introduce the discretization of the square $[0,1]^2$ with a tiling $\mathcal{R}_N$, $N\in\N$. More precisely, we consider $N^{2}$  squares $$ Q_v:= N^{-1}v + N^{-1} [0,1)^{2},$$ where $v \in \{{0},\ldots,N-1\}^2$ is a multi-index.

We denote by $V= \{0,...,N-1\}^2$ the vertex set. A \emph{permutation} $\sigma$ is a bijective map $V \to V$. For a given $\sigma$, we associate a map $\psi_{\sigma} \colon M \to M$ as follows:
\begin{equation} \label{def:psisigma}
    \psi_{\sigma} (x)= N^{-1} \sigma(v) + (x- N^{-1} v) \quad \text{if } x \in Q_v,~v \in V. 
\end{equation}
Given a permutation $\sigma:V\rightarrow V$, we denote by $\sigma_1,\sigma_2$ the first and the second coordinate of $\sigma$.
Observe that we have the following equivalence of norms:
\begin{equation} \label{def:equiv}
    \Vert \psi_{\sigma} - \Id_M \Vert_{L^p(M)} \sim N^{-1} N^{-2/p} \Vert \sigma - \Id_V \Vert_{l^p(V)},
\end{equation}
where we have denoted by
\begin{equation}
    \Vert\sigma-\Id_V \Vert_{l^p(V)}=\left(\sum_{v\in V} |\sigma(v)-v|^p\right)^{\frac{1}{p}},
\end{equation}
and $|\sigma(v)-v|=|\sigma_1(v)-v_1|+|\sigma_2(v)-v_2|$. In the case $p=\infty$ we denote by
\begin{equation*}
     \Vert\sigma-\Id_V \Vert_{l^\infty(V)}=\max\lbrace |\sigma(v)-v|,  v\in V\rbrace.
\end{equation*}

\subsection{Discrete Problem and incompressible flows on networks}\label{Ss:discrete:problem}
Our aim is to first formulate a Shnirelman-type inequality in a purely discrete context. In this section, we mainly follow the setup from \cite{Schiffer_Zizza25}.

We consider a lattice graph $G=(V,E)$ over the set of $N^2$ vertices $V= \{v=(i,j) \colon i,j =0,\ldots N-1\}$ and set of edges $E\subset V^2$ along coordinate lines, i.e.
\[
E=\{ \{v,w\} \colon v,w \in V \text{ and }\vert v - w \vert =1 \}
\]
where $\vert v- w \vert = \vert v_1 - w_1 \vert + \vert v_2 -w_2 \vert$ is the usual Manhattan distance.
A \emph{path} $\gamma$ from a vertex $v_1$ to $v _{k+1}$ is a sequence $ \gamma= \lbrace v_1,...,v_{k+1}\rbrace$ with $v_i \neq v_j$ for $i \neq j$, with $e_i := \{ v_i, v_{i+1}\} \in E$. We can understand a path as a subgraph of $G$ and, consequently, write $V(\gamma) = \{v_1,\ldots v_{k+1}\}$ and write $e \in E(\gamma)$ if $e= \{v_i,v_{i+1}\}$ and $v_i,v_{i+1} \in V(\gamma)$. Usually we will denote a set of paths by $\Gamma$.

For a path $\gamma$ define the \emph{length} $\ell(\gamma)$ of $\gamma$ as the number of edges contained in $E(\gamma)$ and define
\[
    \dist(v,w) =\min_{\gamma \text{ path from } v \text{ to } w } \ell(\gamma).
 \]
If $\gamma_1$ is a path from $v_1$ to $v_{k+1}$ and $\gamma_2$ is a path from $v_{k+1}$ to $v_{m}$, such that $v_i \neq v_j$ for $1 \leq i,j \leq m$, then we may define $\gamma ={\lbrace}v_1,...,v_m{\rbrace}$ as the concatenation of the paths $\gamma_1$ and $\gamma_2$. 

\subsubsection{Problem formulation}
 Given a path $\gamma\in \Gamma$ connecting two vertices $v,w\in V(\gamma)$, we can imagine the fluid particles to travel along pipes visiting the edges $e$ of $\gamma$. The thickness of the pipe, that we will indicate by $\rho$, is inversely proportional to the fluid velocity $u$ (\emph{incompressibility condition}). More precisely, consider an edge $e\in E(\gamma)$. We then define by
\begin{itemize}
    \item $\rho(e,\gamma)$ the thickness of the pipe $\gamma$ along the edge $e$;
    \item $u(e,\gamma)$ the fluid velocity along the edge $e$. 
\end{itemize}
 The $\emph{flow}$ through the pipe per time unit then is the product 
\[
f(e,\gamma) = u(e,\gamma) \cdot \rho(e,\gamma).
\]
We say that the flow is \emph{incompressible} if, for every $\gamma\in\Gamma$, for each neighbouring edges $e_i,e_{i+1}\in E(\gamma)$, (meaning $e_i \cap e_{i+1} \neq \emptyset$) in the path the flow is the same, i.e. the value of the flow \emph{does not} depend on $e$:
\[
f(e,\gamma) = f(\gamma).
\]
Assuming that 
\[
f(\gamma) =1
\]
for all paths $\gamma$, we deduce the following \emph{incompressibility condition}
\begin{equation}\label{eq:incompressibility}
u(e,\gamma) = \rho(e,\gamma)^{-1} \quad \text{if } e \in E(\gamma).
\end{equation}
%\textcolor{green}{In order to obtain a Shnirelman-type inequality for the two-dimensional setting, we formulate the discrete problem as}
% \textcolor{magenta}{I changed this a little bit, so that it is not too close to the treatment in the previous paper.}

Our aim is the following: Given a permutation $\sigma \colon V(G) \to V(G)$ of vertices in the graph, find paths $\lbrace \gamma_v\rbrace$ connecting vertices $v$ to $\sigma(v)$ so that \begin{itemize}
    \item the flow is incompressible along $\gamma_v$, for every vertex $v\in V$;
    \item we can transport the fluid from $v$ to $\sigma(v)$ in a given time span;
    \item at no edge we violate a \emph{capacity} constraint (i.e. the total width of pipes is bounded from above).
\end{itemize}
In classical flow problems (e.g. \cite{AMO,KV}) and in the problem analysed in \cite{Schiffer_Zizza25} (corresponding to the sharp Shnirelman's inequality in $\nu\geq 3$) in principle we can have multiple \emph{augmenting paths} from $v$ to $\sigma(v)$, i.e. fluid particles traveling from $v$ to $\sigma(v)$ can  take different \emph{trajectories} with some probability $\omega(\gamma)$ (meaning that the fluid particles can split along multiple trajectories, see also \cite{Brenier1},\cite{Brenier99}). In the two space dimension, there is a natural choice for a single path $\gamma_v$ from $v$ to $\sigma(v)$ of shortest length.
We may summarise the problem therefore as follows.
\begin{problem}[Incompressible flow along one path] \label{problem:inc1}
     ~\\
     \emph{Given: } A graph $G=(V,E)$, an exponent $1\leq p \leq \infty$ and a bijective map $\sigma: V \to V$. \\
    \emph{Task: } Find paths $\gamma_v$ from $v$ to $\sigma(v)$, $\Gamma$ denoting the set of paths, and $\rho \colon E \times \Gamma \to [0,1]$ obeying
    \begin{enumerate} [label=(\roman*)]
        \item $\rho(e,\gamma) =0$ if $e \notin E(\gamma)$;
        \item time constraint: for each $\gamma \in \Gamma$ we have
            \begin{equation} \label{eq:volume}
                \sum_{e \in E(\gamma)} \rho(e,\gamma) \leq 1;
            \end{equation}           
        \item capacity constraint: for each $e \in E$ we have
            \begin{equation} \label{eq:capacity}
                \sum_{\gamma \colon e \in E(\gamma)} \rho(e,\gamma) \leq 1
            \end{equation}
        \end{enumerate}
    such that the $l_p$ cost functional
    \begin{equation} \label{def:cp}
         \tilde{c}_p(\rho):= \left(\sum_{\gamma \in \Gamma} \sum_{e \in E(\gamma)} \rho(e,\gamma)^{-p+1} \right)^{1/p}
    \end{equation}
    and, for $p=\infty$,
    \begin{equation} \label{def:cinfty}
        \tilde{c}_\infty(\rho):= \sup_{\gamma \in \Gamma} \sup_{e \in E(\gamma)} \rho(e,\gamma)^{-1}
    \end{equation}
    is minimized.
\end{problem}
~\\
This problem deals with the construction of a flow along \emph{one} path according to a permutation as the bijective function $\sigma:V\rightarrow V$, $\sigma\in\mathfrak{S}_{N,\nu}$. Here, we recall that
\begin{itemize}
    \item the time constraint \eqref{eq:volume} guarantees that each flow particle travels in time $1$ to its final destination;
    \item the capacity constraint \eqref{eq:capacity} ensures that the amount of pipes in the discrete network problem is controlled: this condition in particular guarantees that our construction can be adapted to the continuous setting;
    \item the $\ell_p$ cost functional is related to the $L^p$-norm of the velocity field, which is given by $$ \left(\sum_{\gamma\in\Gamma}\sum_{e\in E(\gamma)}u(e,\gamma)^p\rho(e,\gamma)\right)^{\frac{1}{p}}.$$ 
\end{itemize}

\medskip 
As done in \cite{Schiffer_Zizza25}, in order to prove Shnirelman's inequality we give a solution that attains the minimum up to a constant. Let us call paths $\gamma$ and thicknesses $\rho$ \emph{admissible} if they satisfy the constraints given by Problems \ref{problem:inc1}. Finally, we formalize the discrete Shnirelman's inequality as a discrete problem on network as follows:

\begin{problem} \label{realproblem2} ~ \\
  \emph{Input: } A graph $G=(V,E)$, a number $1\leq p \leq \infty$ and a bijective map $\sigma: V \to V$. \\
  \emph{Task: } Find admissible $\gamma'$, $\rho'$ such that
  \[
  c_p(\rho') \leq C(p,G) \Vert \dist(v,\sigma(v)) \Vert_{l_p(V)}=\|\sigma-\id_V\|_{\ell^p(V)}.  
  \]
\end{problem}
\smallskip
\subsubsection{The solution to the discrete problem} 
\label{sec:discrete:2D}
Given $v=(v_1,v_2)$ to $w=(w_1,w_2) = \sigma(v)$ we choose the path $\gamma= \gamma_v$ defined by the following procedure. Consider $z=(w_1,v_2)$. There is a unique shortest path from $v$ to $z$ and a unique shortest path from $z$ to $w$. Those paths are edge-disjoint and define $\gamma$ as the concatenation of those paths. In other words, we choose the paths between $v$ and $w$ in such a way that we first adjust the first coordinate and then the second, as in Figure \ref{fig:discrete:choice}.

\begin{figure}
    \centering
    \includegraphics[width=0.5\linewidth]{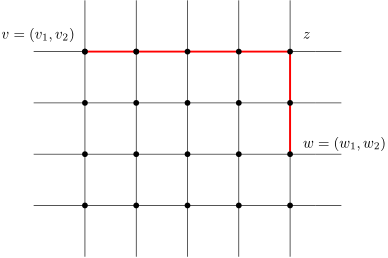}
    \caption{The choice of the path $\gamma$ connecting two vertices $v=(v_1,v_2)$ and $w=(w_1,w_2)$.}
    \label{fig:discrete:choice}
\end{figure}
Define $\Gamma$ to be the set of all those paths $\gamma_v$ and for an edge $e$ let 
\[ 
F(e) = \# \{ \gamma \in \Gamma \colon e \in E(\gamma) \}.
\]
Then we define
\begin{equation} \label{def:flow:2D}
    \rho(e,\gamma) = \begin{cases}
        \Bigl(\max \left( F(e), \ell(\gamma) \right)\Bigr)^{-1} &\text{ if } e\in E(\gamma), \\
        0 & \text{ otherwise.}\end{cases}
\end{equation}
We collect here the relevant statements.  The proofs can be found in Appendix \ref{Ss:appendix:discrete}. 
\begin{lemma} [Lemma 3.9 \cite{Schiffer_Zizza25}]\label{lemma:2D:1}
   Consider the previous $2D$ setup of $G=(V,E)$, paths $\gamma$ and $\rho$ chosen as above. Then
   \begin{enumerate} [label=(\roman*)]
       \item \label{lemma:2D:1:a} If $e=\{(i,k),(i+1,k)\}$ then $F(e) = \# \{ j >i \colon \sigma_1(j,k) \leq i \} + \# \{ j \leq i \colon \sigma_1(j,k) > i\}$;
       \item \label{lemma:2D:1:b} If $e=\{(i,k),(i,k+1)\}$ then $F(e) = \# \{ j>k \colon (i,j) = \sigma(v) \text{ and } v_2 \leq k \} +  \# \{ j<k+1 \colon (i,j) = \sigma(v) \text{ and } v_2>k \};$
       \item \label{lemma:2D:1:c} $(\gamma,\rho)$ is admissible (for Problem \ref{problem:inc1}).
   \end{enumerate}
\end{lemma}

\begin{lemma}[Lemma $3.10$ \cite{Schiffer_Zizza25}]\label{lemma:2D:2}
    Let $G=(V,E)$, $\gamma$ and $\rho$ be as before. Then:
    \begin{enumerate}  [label=(\roman*)]
        \item We have $\Vert F(\cdot) \Vert_{l^1(E)} = \Vert\dist(\cdot,\sigma(\cdot)) \Vert_{l^1(V)}=\|\sigma-\id_V\|_{\ell^1(V)}$;
        \item $\Vert F(\cdot) \Vert_{l^{\infty}(E)} \leq 2 \Vert \dist(\cdot,\sigma(\cdot)) \Vert_{l^{\infty}(V)}=2\|\sigma-\id_V\|_{\ell^\infty(V)}$;
        \item for any $1 \leq p \leq \infty$ we have $\Vert F(\cdot) \Vert_{l^p(E)} \leq C_p \Vert \dist(\cdot, \sigma) \Vert_{l^p(V)}=C_p\|\sigma-\id_V\|_{\ell^p(V)}$.
    \end{enumerate}
\end{lemma}
The two results together give the proof of the following 
\begin{thm} \label{thm:2D}
    Let $G=(V,E)$ be the $2$-dimensional grid, let $\gamma_v$ be the path constructed above connecting $v$ to $\sigma(v)$ and let $\rho$ be as in \eqref{def:flow:2D}. Then the set of paths $\gamma$ and $\rho$ is admissible for the Problem \ref{problem:inc1} and there is a constant $C>0$, depending only on $p$, such that for any $1 \leq p \leq \infty$
    \begin{equation}\label{eq:discrete:inequality}
        \tilde{c}_p(\rho) \leq C \Vert \dist(\cdot ,\sigma(\cdot))\Vert_{l^p(V)}.
    \end{equation}
    
\end{thm}
\begin{proof}
    The proof of this theorem can be easily obtained by the one of the $1$-dimensional case, Theorem $3.8$ in \cite{Schiffer_Zizza25}.
\end{proof}
 
\begin{remark}\label{rmk:space:between:pipes}
    If a set of paths and thicknesses $\rho$ are admissible, observe that then the same set of paths with thicknesses $\alpha \cdot \rho$ is admissible for a scalar $\alpha<1$. Inequality \eqref{eq:discrete:inequality} then (up to a change of constant depending on $\alpha$) is still valid. Consequently, if needed in the implementation of the construction, we may assume a sharper capacity constraint
    $$\sum_{\gamma \colon e \in E(\gamma)} \rho(e,\gamma) \leq \alpha.
    $$
\end{remark}
In \cite{Schiffer_Zizza25} the proof of the corresponding version of Theorem \label{thm:2D} was more involved as we treated any space dimension. The main difference between the present article and \cite{Schiffer_Zizza25} however does not lie in this different solution to this problem (which probably comes down to also cleverly choosing an appropriate set of paths in higher dimensions), but in the implementation of solutions as flows, cf. Section \ref{Section:4}.

\subsection{Sources, Sinks \& Checkpoints} 
For the remainder of the section, we focus on some necessary building block vector fields that split the construction of the actual two-dimensional flows into smaller substeps.
The subsequent results may formally be justified via the $\BV$ regularity of the constructed vector fields. More precisely, let $u\in L^1([0,1],\BV(M))$ be a divergence-free vector field, then there exists a unique flow $\psi\in C([0,1], L^1(M,M))$ such that, for a.e. $x\in M$, the function $\psi$ is an integral solution of 

 \begin{equation*}
     \begin{cases}
         \dot\psi(t,x)=u(t,\psi(t,x)), \\
         \psi(0,x)=x.
     \end{cases}
 \end{equation*}
Moreover, $\psi$ is measure preserving, that is $\mathcal{L}^\nu(\psi^{-1}(A))=\mathcal{L}^\nu(A)$ for any $A$ measurable set.  This unique flow is also called the \emph{Regular Lagrangian flow} (RLF) of the velocity field $u$.  Existence, uniqueness and stability properties of Regular Lagrangian Flows for $\BV$ vector fields have been established in \cite{Ambrosio:BV}. In our setting, RLFs are of fundamental importance, since they give the theoretical framework in order to connect permutations to the identity (see also \cite{Zizza24},\cite{Schiffer_Zizza25}). In the following, given a velocity field $u\in L^1([0,1],\BV(M))$ divergence-free, we say that the Regular Lagrangian Flow $\psi$ of the velocity field $u$ connects a measure-preserving map $f$ to $id$ if $\psi_{t=1}=f$. 

Instead of directly giving a flow vector field that transports a cube $Q$
into another cube $\sigma(Q)$, we instead divide this flow into multiple subflows until we arrive at $\sigma(Q)$. We therefore design flows that go from segments $S_{out}$ (the source) into a segment $S_{in}$ (the sink).  As a simplification, sources and sinks (which are defined as line segments) are always oriented along the coordinate directions $\un_1$ and $\un_2$. Consequently, we focus on two situations: The two segments $S_{in}$ and $S_{out}$ are either parallel (cf. Lemma \ref{lem:transport:tot}) or perpendicular (cf. Lemma \ref{lemma:cornerII}).

We begin with a formal definition of the concept.

\begin{definition} \label{def:sourcesink}
Suppose that $S_{in}, S_{out} \subset M$ are two open disjoint segments (with disjoint here we mean that $\overline{S_{in}}\cap\overline{S_{out}}=\emptyset$) with normals the unit vectors $\nu_{S_{in}},\nu_{S_{out}} \in \R^2$. Let $u_{S_{in}}$ and $u_{S_{out}}$ be two constant vectors such that $u_{S_{in}} = \mu_{S_{in}}\nu_{S_{in}}$ and $u_{S_{out}} = \mu_{S_{out}} \nu_{S_{out}}$, with $\mu_{S_{in}}, \mu_{S_{out}}>0$ and 
\[
\int_{S_{in}} \nu_{S_{in}} \cdot u_{S_{in}} \dH^1 = \int_{S_{out}}  \nu_{S_{out}}  \cdot u_{S_{out}}  \dH^1 \quad \Longleftrightarrow \quad \mu_{S_{in}} \Haus^{1}(S_{in}) = \mu_{S_{out}}  \Haus^1({S_{out}} ).
\]
We say that a piecewise smooth $u \colon M \to \R^2$ induces a flow from $S_{out}$ to $S_{in}$ if 
\begin{enumerate} [label=(\roman*)]
    \item \label{def:sourcesink:1}$\diverg(u) = - \Haus^{1} \llcorner S_{out} \cdot u_{S_{out}}+ \Haus^1 \llcorner S_{in} \cdot u_{S_{in}}$ in the sense of distributions. 
    \item There is a time $t_{S_{out},S_{in}}$ such that we have: for any $z \in S_{out}$ the solution to the ordinary differential equation
    \begin{equation} \label{eq:ODE}
        \begin{cases}
            \partial_t \psi(z,t) = u(\psi(z,t)) & t \in [0,t_{S_{out},S_{in}}),
            \\
            \psi(z,0) = z,
        \end{cases}
    \end{equation}
    satisfies $\psi(z,t_{S_{out},S_{in}}) \in S_{in}$ and the map $\psi \colon z \mapsto \psi(z,t_{S_{out},S_{in}})$ is affine.
    \item $\spt(u) = \psi([0,t_{S_{out},S_{in}}]\times S_{out})$.
\end{enumerate}
We call $\psi$ the \textbf{source-sink flow} from  the \textbf{source} $S_{out}$ to the sink \textbf{sink} $S_{in}$ and, correspondingly, $u$ the \textbf{source-sink vector field.}
\end{definition}

\begin{remark}[Union and concatenation of source-sink flows]\label{rmk:checkpoints}
    With a reference to Lemma $4.4$ in \cite{Schiffer_Zizza25} consider the segments $S_1,\ldots,S_k$ and $Z_1,\ldots,Z_k$ and the source-sink vector fields $u_j$ from $S_j$ to $Z_j$, $j=1,2,\dots, k$. If the supports of $u_j$ are pairwise disjoint then we may define
    \[
    u \colon M \to \R^2, \quad u(x)=u_j(x) \quad x \in \spt(u_j),
    \]
    and we still have
    \begin{enumerate} [label=(\alph*)]
        \item $\diverg(u)  = -\sum_{j=1}^k \Haus^2 \llcorner S_j \cdot u_{S_j} + \sum_{j=1}^k \Haus^2 \llcorner Z_j \cdot u_{Z_j}$ in the sense of distributions.
        \item For any $z \in S_j$ the solution to the ordinary differential equation
          \[  \begin{cases}
            \partial_t \psi^z(t) = u(\psi^z(t)) & t \in [0,t_{S_j,Z_j}),
            \\
            \psi^z(0) = z,
        \end{cases}
        \]
        is the same as for the source-sink vector field $u_j$.
    \end{enumerate}
    Suppose further that $S_1,\ldots,S_k$ are disjoint segments and that there are source-sink flows $u_j$ from $S_j$ to $S_{j+1}$ with normals $\nu_{j,S_j}$ and $\nu_{j,S_{j+1}}$ and velocities $u_{j,S_j}$ and $u_{j,S_{j+1}}$.  Suppose that
    \begin{equation*}
         \nu_{j-1,S_j} = \nu_{j,S_j}, \quad   u_{j-1,S_j} =  u_{j,S_{j}}, \quad j=2,...,n-1,
    \end{equation*}
    and that $\spt(u_j) \cap \spt(u_{j+1}) = S_{j+1}$, $\spt(u_i) \cap \spt(u_j) = \emptyset$ if $\vert i- j \vert \geq 2$. 
    Then we may define
    \[
    u \colon M \to \R^2, \quad u(x)=u_j(x) \quad x \in \spt(u_j)
    \]
   and $u$ is a source-sink flow from $S_1$ to $S_k$. Moreover, we will call $S_2,...,S_{k-1}$ \textbf{checkpoints} of the source-sink flow.
\end{remark}

%%%%%%%%%%%%%%%%%%%%%%
We state the lemmas proved in \cite{Schiffer_Zizza25} in the case of $\nu=2$.

%\begin{figure}
 %   \centering
  %  \includegraphics[width=0.5\linewidth]{sourcesinkparallel.png}
   % \caption{A divergence-free vector field moving the mass from the segment $S_{out}$ (the source) to the segment $S_{in}$ (the sink), Lemma \ref{lem:transport:tot}.}
    %\label{fig:placeholder}
%\end{figure}

        \begin{lemma}\label{lem:transport:tot} 
Let us fix $S_{out} = (0, s_1) \times \{0\}$ and $S_{in} = (z, z + s_2) \times \{h\}$ with $s_1,z,s_2>0$ and consider the triples $(S_{out}, \nu_{S_{out}}, u_{S_{out}})$ and $(S_{in}, \nu_{S_{in}}, u_{S_{in}})$ obeying the properties of Definition \ref{def:sourcesink}.

Denote by $\Psi : S_{out} \to S_{in}$ the affine map
$$\Psi(x, 0) = \left( \frac{s_2}{s_1} x + z, h \right), \quad x \in (0, s_1) .$$

Then there exists a positive time $t_{S_{out},S_{in}}>0$ and  $\psi : [0, t_{S_{in}, S_{out}}] \times [0,1]^2 \to [0,1]^2$ from $S_{out}$ to $S_{in}$ satisfying the following properties: if we denote by $u : [0, t_{S_{out},S_{in}}] \times [0,1]^2 \to \mathbb{R}^2$ the associated velocity field, then it is constant in time and
\begin{itemize}
\item  $supp(u) \subset (0,\max\lbrace{s_1,z+s_2}\rbrace)\times(0,h)$;
\item  $\|u\|_\infty \leq \max\{ \mu_{S_{out}}, \mu_{S_{in}} \} \left( 1 + \frac{2}{h} \max\{ z+s_1, z+s_2 \} \right)$;
\item  $\|u\|_{L^p}^p \leq \left( 1 + \frac{2}{h} \max\{ z+s_1, z+s_2 \} \right)^p \max\{ \mu_{S_{out}}^p \cdot s_1, \mu_{S_{in}}^p \cdot s_2 \} \cdot h$.
\end{itemize}
Moreover, the flow $\psi$ satisfies
$$\psi(t_{S_{in}, S_{out}}, x, 0) = \Psi(x, 0) \quad x \in (0, s_1)$$ with the time $t_{S_{in}, S_{out}}$ satisfying
$$h \min \left\{ \frac{1}{\mu_{S_{in}}}, \frac{1}{\mu_{S_{out}}} \right\} \leq t_{S_{out},S_{in}} \leq h \max \left\{ \frac{1}{\mu_{S_{in}}}, \frac{1}{\mu_{S_{out}}} \right\} .$$
\end{lemma}

\begin{proof}
    The proof is a combination of simple transport and the change of pipe-width result, see Lemmas 5.1 and 5.2 in \cite{Schiffer_Zizza25}. In Appendix \ref{S:appendix} we decided to put the proof of Lemma 5.2 for completeness, see the proof \ref{proof:lemma:52}.
\end{proof}
\begin{remark}
    The proof works similarly in the case we have rigid transformations of $S_{in}$ and $S_{out}$.
\end{remark}
We now define two types of flows that connect intervals $S_{in}$ and $S_{out}$ that are perpendicular. The first features a construction of a divergence-free vector field that 'takes' the corner with constant absolute velocity. In the second lemma we will adjust the construction so that the fluid particles arrive in the sink $S_{in}$ at the same time. 
%\begin{figure}
%    \centering
 %   \includegraphics[width=0.4\linewidth]{corner_1.png}
%    \caption{Corner flow from $S_{in}$ to $S_{out}$. Notice that fluid particles hit $S_{out}$ in different times.}
%    \label{fig:corner:1}
%\end{figure}
\begin{lemma}[Corner Flow with time difference] \label{lemma:corner:1} 
 Let $S_{out}=(0,s)\times \lbrace0\rbrace$ and $S_{in}=\lbrace 2s \rbrace\times(s,2s)$, let $\mu\in\R_+$. Let $\mu>0$ be a positive constant. Then there exists $ t_{S_{out},S_{in}}>0$ and a divergence-free velocity field $u:[0,t_{S_{out},S_{in}}]\times[0,2]^2\rightarrow\R^2$ constant in time, with  $u\in L^\infty([0,t_{S_{out},S_{in}}],L^\infty([0,1]^2))\cap L^\infty([0,t_{S_{out},S_{in}}], \BV([0,1]^2))$ satisfying
    \begin{enumerate}
        \item $\text{supp}(u)\subset (0,2s)^2\setminus (s,2s)\times(0,s)$
        \item $|u|=\mu$ a.e. on $\spt(u)$;
        \item $u=\mu \un_2$ on $\lbrace (x,y): x\in(0,s), y\leq -x+2s\rbrace$;
        \item $u=\mu\un_1$ on $\lbrace (x,y): y\in(s,2s), x\geq -y+2s\rbrace$;
        \item if $\psi:[0,+\infty)\times[0,1]^2\rightarrow [0,1]^2$ denotes the flow of $u$, and
        $$f(x)=\frac{2s}{\mu}+(s-x)\frac{2}{\mu}$$ for $x\in(0,s)$ then
        \begin{equation}\label{eq:condition:time:difference}
        \psi(f(x),x,0)\in S_{in}.
        \end{equation}
    \end{enumerate}
\end{lemma}
Observe that condition \eqref{eq:condition:time:difference} implies that the fluid particle starting at $(x,0)\in S_{out}$ hits the segment $S_{in}$ in time $f(x)$, in particular the particles arrive to their final destination in different times. We will call  \emph{time difference} the quantity $\Delta f= f(0)-f(s)=2s/\mu$, corresponding to the difference of time for particles at the endpoints. The function $f$ will be called \emph{time-difference function} associated with $S_{out}$ and $S_{in}$. Notice that, in the case we choose a different orientation of the corner, as the corner flow from $S_3$ to $S_4$ in Figure \ref{fig:corner:2}, then the time-difference will be negative. 

We would also like to remark that, since the corner flows are building blocks for flows coming from a parallel flow, we need that the velocity field is constant when restricted to $S_{in}$ and $S_{out}$, this is the reason that justifies our choice for the vector field. 
Combining Lemmas \ref{lem:transport:tot} and \ref{lemma:corner:1} we construct a \emph{Corner Flow} where the particles move from $S_{out}$ to $S_{in}$ perpendicular segments but the particles starting simultaneously in $S_{out}$ hit the segment $S_{in}$ \emph{at the same time}. To do so, we propose a variant of Lemma $5.4$ in \cite{Schiffer_Zizza25} that is also valid in any dimension, but we state here in dimension $\nu=2$ for simplicity.

\begin{prop}[Corner Flow]\label{lemma:cornerII}
Let $S_{out} = (0, s) \times \{0\}$ and $S_{in} = \{2s\} \times (s, 2s)$ with $s > 0$. Finally let $\mu > 0$ be a positive constant and consider the triples $(S_{out}, \un_2, \mu \un_2)$, $(S_{in}, \un_1, \mu \un_1)$ satisfying Definition \ref{def:sourcesink}. Then there exists a source-sink flow $\psi : [0, t_{S_{out}, S_{in}}] \times [0, 1]^2 \to [0, 1]^2$ with $t_{S_{out}, S_{in}} > 0$ satisfying the following properties:

\begin{enumerate}
    \item call $u : [0, t_{S_{out}, S_{in}}] \times [0, 1]^2 \to \mathbb{R}^2$ the source-sink vector field associated with $\psi$, then $u$ is constant in time and 
    
    $u \in L^\infty([0, t_{S_{out}, S_{in}}], L^\infty([0, 1]^2)) \cap L^\infty([0, t_{S_{out}, S_{in}}], \BV([0, 1]^2))$;
    \item $supp(u) \subseteq [0, 2s]^2$;
    \item it holds $\|u\|_\infty \leq 36 \mu$, and in particular
    $$\|u\|_{L^p}^p\leq 2(36 \mu)^p  s^2 ;$$
    \item the time $t_{S_{out}, S_{in}} > 0$ satisfies
    $$\frac{9}{8}\frac{s}{\mu} \leq t_{S_{out}, S_{in}} \leq \frac{9}{4}\frac{s}{\mu}. $$
\end{enumerate}

and the map
$$\psi(t_{S_{out}, S_{in}}, x, 0) = \Psi(x, 0), \text{ for } x \in (0, s)$$
with $\Psi(x, 0) = (2s, 2s-x)$.
\end{prop}

\begin{figure}
    \centering
    \includegraphics[width=0.5\linewidth]{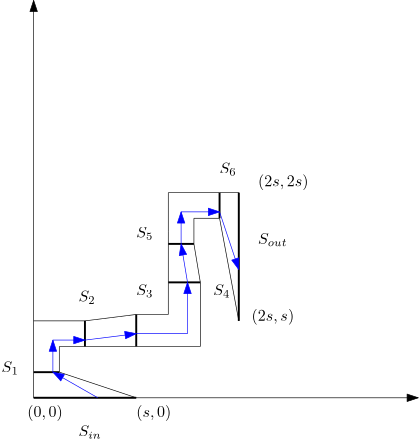}
    \caption{We use Lemmas \ref{lem:transport:tot} and \ref{lemma:corner:1} in order to arrange the fluid flow so that particles hit the segment $S_{out}$ simultaneously. Notice that we have to combine at least three corner flows in order to adjust the time difference.}
    \label{fig:corner:2}
\end{figure}
\begin{proof}
 With a reference to Figure \ref{fig:corner:2}, we use a concatenation of the following flows:
 \begin{itemize}
     \item flow from $S_{out}=S_0 = (0,s) \times \{0\}$ to $S_1 = (0,s/4) \times \{s/4\}$ (Lemma \ref{lem:transport:tot});
     \item flow from $S_1$ to $S_2= \{s/2\} \times (s/2,3s/4)$ (Lemma \ref{lemma:corner:1});
     \item flow from $S_2$ to $S_3 = \{s\} \times (s/2,s/2+s/\sqrt{8})$ (Lemma \ref{lem:transport:tot});
     \item flow from $S_3$ to $S_4 =(s + s/\sqrt{8},s+s/\sqrt{2}) \times \{s/2+s/\sqrt{2}\}$ (Lemma \ref{lemma:corner:1});
     \item flow from $S_4$ to $S_5= (s + s/\sqrt{8},5s/4 +  s/\sqrt{8}) \times \{3s/2\}$ (Lemma \ref{lem:transport:tot});
     \item flow from $S_5$ to $S_6=  \{3s/2+s/\sqrt{8}\} \times (7s/4,2s)$ (Lemma \ref{lemma:corner:1});
     \item flow from $S_6$ to $S_{out}=S_7= \{2s\} \times (s,2s)$ (Lemma \ref{lem:transport:tot}).
 \end{itemize}

 To compute the total time $t_{S_{out},S_{in}}$ consider the following time-differences:
 \begin{itemize}
     \item $f_1$ the time-difference function associated with $S_1$ and $S_2$, so that $\Delta f_1=\frac{1}{4\mu}\frac{s}{2}$;
     \\
     \item $f_2$ the time-difference function associated with $S_3$ and $S_4$, so that $\Delta f_2=-\frac{1}{\sqrt{8}\mu}\frac{s}{\sqrt{2}}=-\frac{s}{4\mu}$;
     \\
     \item $f_3$ the time-difference function associated with $S_5$ and $S_6$, so that $\Delta f_2=\Delta f_1$.
 \end{itemize}
 In particular notice that 
$$\Delta f_1+\Delta f_2+\Delta f_3=0,$$
implying that fluid particles starting in $S_{out}$ ends in $S_{in}$ simultaneously. 

Finally, putting together the results of Lemmas \ref{lem:transport:tot} and \ref{lemma:corner:1} the bounds on the $L^p$-norm and on $t_{S_{out},S_{in}}$ are straightforward.

\end{proof}
 \subsection{Swap of squares}\label{Ss:swapping}

The last ingredient for the two-dimensional pipe flow is the swapping of squares.  Given two squares $Q_1,Q_2$ of sidelength $a>0$ (without loss of generality we may assume $a<<1$), we define $\phi_{1,2}(t)$ the flow that swaps $Q_1$ and $Q_2$ in the time interval $[0,\tau]$. 
More precisely, if we define $c_{Q_1},c_{Q_2}\in[0,1]^2$ the centers of $Q_1,Q_2$, then

\begin{equation*}
    \phi_{1,2}(\tau)(x,y)=(x,y)+(c_{Q_2}-c_{Q_1}),\quad\forall (x,y)\in Q_1,
\end{equation*}
\begin{equation*}
    \phi_{1,2}(\tau)(x,y)=(x,y)+(c_{Q_1}-c_{Q_2}),\quad\forall (x,y)\in Q_2.
\end{equation*}
This flow can be easily defined following the proof in \cite{Zizza24}, Appendix A. The divergence-free vector field whose flow is given by $\phi_{1,2}(t)$ is the quantity  
    \begin{equation}\label{eq:swap:velocity:field}
        \dot\phi_{1,2}(t,\phi^{-1}_{1,2}(t,x,y))=v_{1,2}(t,x,y).
    \end{equation}
    For convenience, we put here the statement about the properties of the velocity field $v_{1,2}$.
    \begin{lemma}\label{lem:swap}
    Let $\tau>0$ and $p\in[1,+\infty]$. Given two disjoint squares $Q_1,Q_2$ of sidelength $a>0$, with $Q_1,Q_2\subset[0,1]^2$, there exists a positive constant $C=C(p)>0$ and a velocity field $v_{1,2}:[0,\tau]\times[0,1]^2\rightarrow\R^2$, divergence-free with $v_{1,2}\in L^\infty([0,\tau],\text{BV}[0,1]^2)\cap L^\infty([0,\tau],L^p[0,1]^2)$ satisfying the following properties: 
    \begin{enumerate}
        \item the flow $\phi_{1,2}(t)$ associated with $v_{1,2}$ swaps $Q_1$ and $Q_2$ within the time interval $[0,\tau]$, that is
        \begin{align*}
        \phi_{1,2}(\tau)(x,y)=\begin{cases}
    (x,y)+(c_{Q_2}-c_{Q_1}), &(x,y)\in Q_1, \\
    (x,y)+(c_{Q_1}-c_{Q_2}), &  (x,y)\in Q_2, \\
    {(x,y)} & (x,y) \notin Q_1 \cup Q_2.
    \end{cases}
        \end{align*}
        \item The velocity field satisfies 
$$\|v_{1,2}\|_{L^\infty_t L^p_x}\leq \frac{C}{\tau}|c_{Q_1}-c_{Q_2}|a^{\frac{2}{p}},$$
and
$$\|v_{1,2}\|_{L^\infty_t L^\infty_x}\leq \frac{C}{\tau}|c_{Q_1}-c_{Q_2}|.$$
    \end{enumerate}
        Moreover, the constant $C$ can be chosen to be independent on $\tau, a$ and any choice of $Q_1,Q_2$.
    \end{lemma}
    \begin{proof}
        The proof is a time-rescale of the flow in the proof in Appendix A \cite{Zizza24} and it relies on the idea that the $L^p$-cost for swapping the squares is bounded by the distance of their centers times the area of a square to the power $1/p$.
    \end{proof}
    Since we will extensively use this Lemma throughout the paper, we will refer to $C=C_{\text{swap}}$ as the \emph{swap constant}.

    \begin{remark}[Swaps with different orientations]\label{rmk:rotation:swap}
       Given a square $Q=(0,a)^2$ we define the \emph{rotation flow} $r_{t}:Q\rightarrow Q$ for $t\in[0,1]$ in the following way: call
\begin{equation*}
	V(x)=\max \left\{ \left| x_1-\frac{a}{2}\right|, \left| x_2-\frac{a}{2}\right|\right\}^2, \quad (x_1,x_2)\in Q.
\end{equation*}
Then the \emph{rotation field} is $r:Q\rightarrow \R^2$
\begin{equation}
	\label{rotation field}
	{r}(x)=\nabla V^\perp(x),
\end{equation}
where $\nabla^\perp =(-\partial_{x_2},\partial_{x_1})$ is the orthogonal gradient. Finally the rotation flow $r_{t}$ is the flow of the vector field $r$, i.e., the unique solution to the following ODE system:
\begin{equation}
\label{rot:flow:square}
	\begin{cases} 
	\dot {r}_{t}(x)={r}(r_{t}(x)), \\
	r_{0}(x)=x.
	\end{cases}
\end{equation}
This flow rotates the cube counterclockwise of an angle ${\pi}/{2}$ in a unit interval of time. Notice that we can easily modify the flow of Lemma \ref{lem:swap} (with minor adjustments in the involved constants) so that, in the time interval $[0,\tau]$, the cubes $Q_1$ and $Q_2$ have been rotated counterclockwise of an angle ${\pi}/{2}$ (simply apply the rotation flow $r_{\frac{t}{\tau}}$ to each cube, then the swap flow and rescale the time). In this case we will denote the new resulting flow (resp. vector field) by $\phi_{1,2,\frac{\pi}{2}}$ (resp. $ v_{1,2,\frac{\pi}{2}}$).
    \end{remark}
    

%% file: interchangetoyproblem.tex
\section{Building blocks for the intersection of pipes}\label{sec:intersection}

In this section we provide estimates and building blocks necessary to describe the intersection of two pipes. To give the reader an idea how such a result generally looks like, we could construct a flow $\Phi \colon [0,1] \times \R^2 \to \R^2$ with uniformly bounded velocity field such that 
\begin{equation*}
    \Phi(1,(x,y)) = \begin{cases}
            (x+1,y) &\text{if } y \in[0,1], x \notin [-1,1], \\
            (x+2,y) & \text{if } y \in [0,1], x \in [-1,0], \\
            (x,y+1) & \text{if } x \in [0,1], y \notin [-1,1], \\
            (x,y+2) & \text{if } x \in [0,1], y \in [-1,0],\\
            (x,y) & \text{else.}
    \end{cases}
\end{equation*} 
In other words, two laminar flows meet in the square $[0,1]^2$. After time $1$, however, it looks like both flows "passed through a wormhole" (e.g. from the segment $\{0\} \times [0,1]$ to $\{1\} \times [0,1]$) and did not interfere in the region $[0,1]^2$.

The goal of this section is to formalize and generalize such a statement and to provide a proof. The focus lies on the following:
\begin{itemize}
    \item Provide a statement that works for different pipe sizes (in above example both are 1) and different velocities.
    \item By repeating the flow in the above example, one may get a time-periodic vector field that is periodic of period $\tau=1$ (the constant $1$ which depended on the size of the intersection region $[0,1]^2$ and the velocity field $v$). To synchronize different intersection regions, that a fluid flow must pass, this periodicity must be the same for any region, i.e. must be an independent parameter.
\end{itemize}

We use the term intersection region for a region like $[0,1]^2$ above (precise definition below); in the final (Regular Lagrangian) flow that we construct of course no trajectories actually intersect.

We shortly remind the reader of some frequently used notation and facts, that we may gather from the previous section:
\begin{itemize}
   \item We consider two pipes that meet perpendicularly. The pipes have width $a$ and $b$ and the fluid is generally moving at velocity $v_0$ and $w_0$ with $a \cdot v_0 = b \cdot w_0$.
   \item We may assume (cf. Remark \ref{rmk:space:between:pipes}) that, by slight modification of pipe width, parallel pipes are \emph{separated}: that is, if we have two parallel pipes of width $a$ and $b$ respectively, with $a<b$, then their distance $>\frac{a+b}{5}$. This means that in an intersection region $\mathcal{A}$ two and only two pipes meet. Moreover, 
   we might estimate all pipe widths from below by a power of $4$ while the solution only loses a constant in the cost inequality \eqref{eq:discrete:inequality}. Therefore we might also assume that the corresponding velocity is a power of $2$. Observe further that due to the explicit form given in \eqref{def:flow:2D}, all thicknesses can then be bounded from below by $(4N)^{-1}$. After suitable rescaling, we therefore might assume that all thicknesses $a$ or $b$ that we deal with are bounded from below by some $4^{-k_{min}}$ (which only depends on $N$)
   and that the velocities correspondingly satisfy $v_0 \cdot a = w_0 \cdot b = C N$ with some dimensional constant $C$.
    \item If two pipes meet, we will denote an intersection region by $\mathcal{A}_{a,b}$ (precise definition below).
    \item We aim to make the flow time-periodic of interval length $2s$ (which might depend on $N$) with three time phases where the corresponding velocity is constant-in-time.

\end{itemize}

\subsection{{Two pipes with the same width}} \label{subsec:samewidth}
As a first building block, we consider a toy model of two pipes of same thickness $a>0$ that \emph{meet} inside an intersection region $\mathcal A=\mathcal{A}_a=[0,5a]^2$ that we may assume, after translation, to be $\mathcal A=\mathcal{A}_a=[0,5a]\times [-2a,3a]$ (of size $5a$). To this aim, we fix $v_0>0$ representing the value of the velocity field at the boundary of the intersection region. Moreover, we denote by  
       \begin{equation}\label{Eq:relation:parameters}
           \tau=\frac{a}{v_0},
       \end{equation}
       the time-variable, which represents, up to a constant factor, the amount of time a fluid particle spends in the region $\mathcal{A}$.

     \begin{figure}
         \centering
         \includegraphics[scale=0.5]{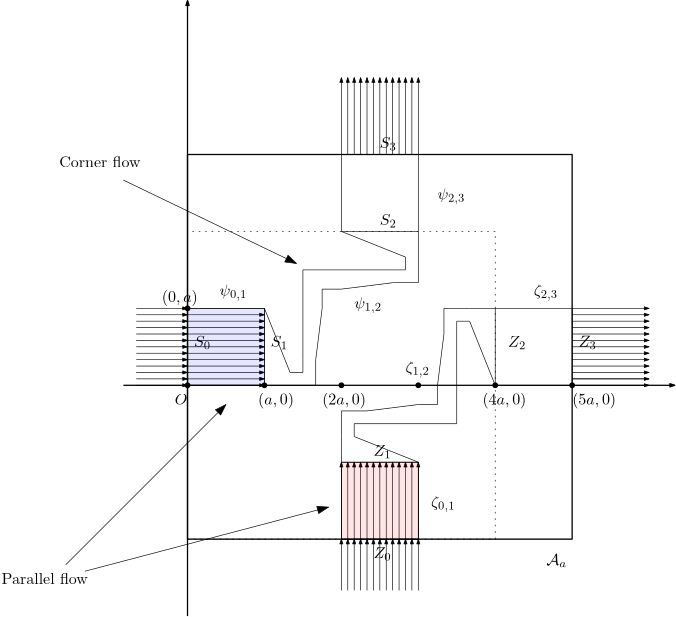}
         \caption{A look inside an intersection region. We first swap the blue and the red squares, then we move the flow in the corner. The time spent inside the corner is the same time a fluid particle starting from the segment $S_0$ takes to hit the segment $S_1$, thus each corner and each square contain the same amount of fluid. 
         }
         \label{fig:building:1}
     \end{figure}
    \begin{lemma}[Building block for the intersection flow]\label{lem:inter:flow}
    
       Let us fix $a,v_0>0$ positive parameters. Then there exists a divergence-free vector field $v:[0,+\infty)\times \R^2 \rightarrow \R^2$ with
        $v\in L^\infty([0,+\infty),L^\infty(\R^2))\cap L^\infty([0,+\infty),\BV(\mathcal{A}))$ with the following properties:
        \begin{enumerate}
            \item the velocity field is time-periodic of period $2\tau$;
            \item \label{time:flow:lemma:1}denote by $\phi:[0,+\infty)\times \R^2\rightarrow \R^2$ the flow associated with the velocity field $v$. Then 
            \begin{equation*}\label{eq:1}
                \phi(4\tau,x,y)=(4a+x,y),\quad \text{a.e. } (x,y)\in (0,a)\times(0,a);
            \end{equation*}
            \begin{equation*}\label{eq:2}
                \phi(4\tau,x,y)=(x,y+4a),\quad \text{a.e. } (x,y)\in (2a,3a)\times(-2a,-a);
            \end{equation*}
            \item the velocity field $v$ satisfies the following (boundary) conditions:
            \begin{equation*}
                v(t,0,y)\cdot\un_1= v(t,a,y)\cdot\un_1= v(t,4a,y)\cdot\un_1= v(t,5a,y)\cdot\un_1=v_0, 
            \end{equation*}
            for $\quad y\in(0,a), \quad t\in[j\tau, (j+1)\tau), \quad j\text{ odd}$,
            \begin{equation*}
                v(t,x,-2a)\cdot\un_2= v(t,x,-a)\cdot\un_2= v(t,x,2a)\cdot\un_2= v(t,x,3a)\cdot\un_2=v_0, 
            \end{equation*}
            for $\quad x\in(2a,3a), \quad t\in[j\tau, (j+1)\tau), \quad j\text{ odd}$;
            \medskip
          \item the velocity field satisfies the following bounds
        \begin{equation}
          \|v\|_{L^\infty([0,+\infty),L^\infty([0,1]^2))}\leq 36C_{\text{swap}} v_0,
      \end{equation}
      and for $p\in[1,+\infty)$
     \begin{equation}
           \|v\|^p_{L^\infty([0,+\infty),L^p(\mathcal{A}))}\leq 25\cdot(36)^p C_{\text{swap}}^p (v_0^pa^2),
      \end{equation}
      where $C_{\text{swap}}$ is the swap constant of Lemma \ref{lem:swap} and it is independent of any choice of $a,v_0$.
      \end{enumerate}
        \end{lemma}
    
  \begin{proof}
  With a reference to Figure \ref{fig:building:1}, consider the following notation:
  \begin{itemize}
   \item  $S_0=\lbrace 0\rbrace\times(0,a)$, $S_1=\lbrace a\rbrace\times(0,a)$, $S_2=(2a,3a)\times\lbrace 2a\rbrace$ and $S_3=(2a,3a)\times\lbrace 3a\rbrace$; 
   \item $Z_0=(2a,3a)\times\lbrace -2a\rbrace$,$Z_1=(2a,3a)\times\lbrace -a\rbrace$, $Z_2=\lbrace 4a\rbrace\times(0,a)$ and $Z_3=\lbrace 5a\rbrace\times (0,a)$;
   \end{itemize}

       Call:
       \begin{itemize}
           \item $\psi_{0,1}$ the source-sink flow transporting $(S_0,\un_1, v_0\un_1)$ to $(S_1,\un_1,v_0\un_1)$ (Lemma \ref{lem:transport:tot}), and $u_{0,1}$ its associated velocity field;
           \item $\psi_{1,2}$ the source-sink flow transporting $(S_1,\un_1, v_0\un_1)$ to $(S_2,\un_2,v_0\un_2)$ (Lemma \ref{lemma:cornerII}), and $u_{1,2}$ its associated velocity field;
           \item $\psi_{2,3}$ the source-sink flow transporting $(S_2,\un_2, v_0\un_2)$ to $(S_3,\un_2,v_0\un_2)$ (Lemma \ref{lem:transport:tot}), and $u_{2,3}$ its associated velocity field.
       \end{itemize}
       Observe that $u_{0,1},u_{1,2}, u_{2,3}$ are autonomous velocity fields by construction.
       Denote by
       \begin{itemize}
           \item $t_{S_0S_1}={a}/{v_0}=\tau,$ the time that a fluid particle starting in $S_0$ spends to hit the segment $S_1$;
       \item and by $t_{S_1S_2}\in\left[{9}/{8}\tau,{9}/{4}\tau\right]$, the time that a fluid particle starting in $S_1$ spends to hit the segment $S_2$.
       \end{itemize}
       
       For the sake of clarity and without loss of generality, we hereafter assume $t_{S_1S_2}=t_{S_0S_1}=\tau$. While the actual ratio ${t_{S_1S_2}}/{t_{S_0S_1}}\in\left[{9}/{8}\tau,{9}/{4}\tau\right]$ any required adjustment simply involves a rescaling of the intersection region (as illustrated in Figure \ref{fig:building:11}, Appendix \ref{appendix:B}) by a factor independent of the core construction. To avoid unnecessary technical complexity, we maintain this uniform notation in all forthcoming results.
       
       Similarly, denote by
       \begin{itemize}
           \item $\zeta_{0,1}$ the source-sink flow transporting $(Z_0,\un_2, v_0\un_2)$ to $(Z_1,\un_2,v_0\un_2)$ (Lemma \ref{lem:transport:tot}), and $w_{0,1}$ its associated velocity field;
           \item $\zeta_{1,2}$ the source-sink flow transporting $(Z_1,\un_2, v_0\un_2)$ to $(Z_2,\un_1,v_0\un_1)$ (Lemma \ref{lemma:cornerII}), and $w_{1,2}$ its associated velocity field;
           \item $\zeta_{2,3}$ the source-sink flow transporting $(Z_2,\un_1, v_0\un_1)$ to $(Z_3,\un_1,v_0\un_1)$ (Lemma \ref{lem:transport:tot}), and $w_{2,3}$ its associated velocity field.
       \end{itemize}
       Notice that $t_{Z_0Z_1}=t_{S_0S_1}=\tau$ and $t_{Z_0Z_1}=t_{Z_1Z_2}=t_{Z_2Z_3}$, $t_{S_0S_1}=t_{S_1S_2}=t_{S_2S_3}$.
       
        Next, denote by $Q_1=(0,a)^2$ and $Q_2=(2a,3a)\times(-2a,-a)$ and by  $\xi_{1,2,\frac{\pi}{2}}:[0,\tau]\times\R^2\rightarrow\R^2$ the flow that swaps the squares $Q_1$ and $Q_2$ in time $\tau$ and rotate them counterclockwise of an angle $\pi/{2}$ (see Lemma \ref{lem:swap} and Remark \ref{rmk:rotation:swap}). 
      Similarly, call $v_{1,2,\frac{\pi}{2}}:[0,\tau]:\R^2\rightarrow\R^2$ the velocity field that swaps $Q_1$ and $Q_2$ and rotate them of $\pi/{2}$. We recall that
      \begin{equation}
          \|v_{1,2,\frac{\pi}{2}}\|_{L^\infty([0,\tau],L^\infty([0,1]^2))}\leq \frac{C_\text{swap}}{\tau}\sqrt{5}a,
      \end{equation}
      and
      \begin{equation}
          \|v_{1,2,\frac{\pi}{2}}\|^p_{L^\infty([0,\tau],L^p(\mathcal A))}\leq \frac{C_\text{swap}^p}{\tau^p}(\sqrt{5})^p a^{p+2},
      \end{equation}
      where $C_\text{swap}$ is the swap constant of Lemma \ref{lem:swap}.
      Call
      \begin{equation*}
          \tilde v(x,y)=u_{0,1}(x,y)+u_{1,2}(x,y)+u_{2,3}(x,y)+v_{0,1}(x,y)+v_{1,2}(x,y)+v_{2,3}(x,y)
          \end{equation*}
          and notice that $\tilde{\phi}(t_Z,x,y)\in S_2$ whenever $(x,y)\in S_0$ and similarly $\tilde\phi(t_Z,x,y)\in Z_2$ whenever $(x,y)\in Z_0$. Moreover, by Lemma \ref{lem:transport:tot} and Lemma \ref{lemma:cornerII} we have that 
      \begin{equation}\label{eq:lemma:5:2:1}
          \|\tilde v\|_{L^\infty([0,1]^2)}\leq\max\lbrace 36 v_0, 3v_0\rbrace\leq 36v_0,
      \end{equation}
      \begin{equation}\label{eq:lemma:5:2:1}
          \|\tilde v\|^p_{L^p(\mathcal{A})}\leq (2\cdot (2\cdot 36)^p+4\cdot 3^p)v_0^pa^2,
      \end{equation}
      and  \begin{equation}\label{eq:lemma:5:2:3}
          \|\tilde v\|^p_{L^p((0,4a)\times(-2a,2a))}\leq (2\cdot (2\cdot 36)^p+2\cdot 3^p)v_0^pa^2,
      \end{equation}
      (this last formula will be useful for the following computations).
      We are finally ready to define the velocity field $v:[0,+\infty)\times\R^2\rightarrow \R^2$ of the statement. 
      \begin{equation*}
          v(t,x,y)=\begin{cases}
              v_{1,2,\frac{\pi}{2}}(t-j\tau,x,y), \quad &t\in[j\tau,(j+1)\tau), \quad j=0,2,4,\dots, \quad (x,y)\in\mathcal{A}, \\
              \tilde v(x,y),  \quad& t\in[j\tau,(j+1)\tau), \quad j=1,3,5,\dots, \quad (x,y)\in\mathcal{A}, \\
              v_0\un_1, &(x,y)\in\left((-\infty,0)\cup(5a,+\infty)\right)\times(0,a),\\ & \quad t\in[j\tau,(j+1)\tau),  j=1,3,5,\dots, \quad  \\
              v_0\un_2, &(x,y)\in(2a,3a)\times\left((-\infty,-2a)\cup(3a,+\infty)\right), \\ &\quad t\in[j\tau,(j+1)\tau), j=1,3,5,\dots,  \\
              0 &\text{otherwise}.
          \end{cases}
      \end{equation*}
      It is clear that $v$ is divergence-free and time-periodic of period $2\tau$ and all the requirements of the statement are satisfied. Moreover, we have that
      \begin{equation*}
          \|v\|_{L^\infty([0,+\infty),L^\infty[0,1]^2)}\leq C_{\text{swap}}\max\left\lbrace \sqrt{5}v_0,36 v_0\right\rbrace,
      \end{equation*}
      and consequently
     \begin{align*}
          \|v\|^p_{L^\infty([0,+\infty),L^p[(\mathcal A))}&\leq 25\cdot(36)^p C_{\text{swap}}^p (v_0^pa^2).
      \end{align*}
         \end{proof}
   
Given an intersection region $\mathcal{A}=(0,5a)\times(-2a,3a)$ as before , we will call \emph{size} of the intersection region the number $a>0$. If we also have a positive parameter $v_0>0$,  we will denote by $\phi^\mathcal{A}:[0,+\infty)\times\R^2\rightarrow \R^2$ the flow constructed in the previous lemma and we will denote by $v^{\mathcal{A}}:[0,+\infty)\times \R^2\rightarrow \R^2$ its associated velocity field, supported in the intersection region. Notice that, because of the definition of the velocity field $v^\mathcal{A}$, we can easily glue multiple intersection regions. 

To fix the notation, consider $J$ intersection regions $\mathcal{A}_1=(0,5a)\times(-2a,3a), \mathcal{A}_2=(4a, 9a)\times(-2a,3a),\dots,\mathcal{A}_J=(4(J-1)a,(4(J-1)+5)a)\times(-2a,3a)$ with same size $a$, for some $J\in\N$. 

\begin{coro}[Superposition of intersection regions]\label{coro:superposition:intersection}
    Let us fix $a,v_0>0$ positive parameters and consider $\mathcal A_1,\mathcal A_2,\dots,\mathcal A_J$, with $J\in\N$, the intersection regions defined above. Then there exists a divergence-free vector field $v:[0,+\infty)\times \R^2 \rightarrow \R^2$ with
        $v\in L^\infty([0,+\infty),L^\infty(\R^2))\cap L^\infty([0,+\infty),\BV(\cup_{j=1}^J\mathcal{A}_j))$ with the following properties:
        \begin{enumerate}
            \item the velocity field is time-periodic of period $2\tau$, with $\tau=\frac{a}{v_0}$;
            \item\label{cor:condition:2} denote by $\phi:[0,+\infty)\times \R^2\rightarrow \R^2$ the flow associated with the velocity field $v$. Then 
            \begin{equation*}\label{eq:1}
                \phi(4J\tau,x,y)=((4+3(J-1))a+x,y),\quad \text{a.e. } (x,y)\in (0,a)\times(0,a);
            \end{equation*}
            \begin{equation*}\label{eq:2}
                \phi(4\tau,x,y)=(x,y+4a),\quad 
                \end{equation*}
                \begin{equation*}\text{a.e. } (x,y)\in ((2+4(j-1))a, (3+4(j-1))a)\times(-2a,-a),\quad j=1,\dots J;
            \end{equation*}
            \item the velocity field $v$ satisfies the following (boundary) conditions:
            \begin{equation*}
                v(t,4(j-1)a,y)\cdot\un_1= v(t,(4(j-1)+5)a,y)\cdot\un_1=v_0,
            \end{equation*}
            for $j=1,2,\dots, J$ and $\quad y\in(0,a), \quad t\in[h\tau, (h+1)\tau), \quad h\text{ odd}$,
            \begin{equation*}
                v(t,x,-2a)\cdot\un_2= v(t,x,-a)\cdot\un_2= v(t,x,2a)\cdot\un_2= v(t,x,3a)\cdot\un_2=v_0, 
            \end{equation*}
            for $\quad x\in ((2+4(j-1))a, (3+4(j-1))a), \quad j=1,2,\dots J,\quad \quad t\in[h\tau, (h+1)\tau), \quad h\text{ odd}$;
          \item the velocity field satisfies the following bounds
        \begin{equation}
          \|v\|_{L^\infty([0,+\infty),L^\infty([0,1]^2))}\leq 36C_{\text{swap}} v_0,
      \end{equation}
      and for $p\in[1,+\infty)$
     \begin{equation}
           \|v\|^p_{L^\infty([0,+\infty),L^p(\mathcal{A}_1\cup\dots\mathcal{A}_J))}\leq 25\cdot(36)^p C_{\text{swap}}^p (v_0^pa^2J).
      \end{equation}
      \end{enumerate}      
\end{coro}
\begin{proof}
    If $v_j:[0,+\infty)\times\R^2\rightarrow\mathbb{R}^2$ is the velocity field of Lemma \ref{lem:inter:flow}, for $j=1,2,\dots,J$ we observe that $v^j=v^{j+1}$ on $\mathcal{A}_j\cap\mathcal{A}_{j+1}$, thus we can define the velocity field of the statement as
    \begin{equation}
        v(t,x,y)=\begin{cases}
              v^j(t,x,y), \quad & (x,y)\in\mathcal{A}_j, \\
              v_0\un_1, &(x,y)\in\left((-\infty,0)\cup((4(J-1)+5)a,+\infty)\right)\times(0,a), \\ & \quad t\in[j\tau,(j+1)\tau),  j=1,3,5,\dots,\\
              v_0\un_2, &(x,y)\in \cup_{j=1}^J ((2+4(j-1))a, (3+4(j-1))a)\times \\ &\left((-\infty,-2a)\cup(3a,+\infty)\right), \\& \quad t\in[j\tau,(j+1)\tau),  j=1,3,5,\dots,\\
              0 &\text{otherwise}.
          \end{cases}
    \end{equation}
\end{proof}

When considering $\mathcal{A}_1,\dots,\mathcal{A}_J$ intersection regions as in the statement, we will denote by $v^{1,2,\dots,J}:[0,+\infty)\times\R^2\rightarrow\R^2$ the vector field obtained by concatenation via Corollary \ref{coro:superposition:intersection}. 
\begin{remark}\label{rmk:superposition}
    While Corollary \ref{coro:superposition:intersection} is formulated for subsequent regions in $\un_1$-direction, it is obviously also valid in $\un_2$-direction.
\end{remark}

\subsection{{Refinement of intersection regions}} \label{subsec:refinement}

Up to this point, the constructed flow depends on the size $a$ of the intersection region $\mathcal{A}_a$ and on the entry velocity $v_0$ into the region. To define the flow simultaneously across all intersection regions, we must ensure that the swap-shear alternation period, previously denoted as $\tau$, becomes independent of the intersection region. We will denote this period by $2s$ and observe that it depends on the tiling size $N$. Consequently, we will show that it is possible to design a $\BV$ velocity field that \emph{immerses} two shears and the cost of this operation is in the period $s$ (see also Subsection \ref{Ss:heuristics}).

To this aim, we subdivide the flows into multiple parallel pipes. In doing so, we shift the dependency on the intersection region (if chosen small enough) from the time-period to the number of pipes; this adjustment allows us to define a uniform time-parameter $s>0$ (corresponding to $\tau$ from the previous subsection) that is completely independent of the intersection region.
The idea is to use the result of Lemma \ref{lem:inter:flow} and Corollary \ref{coro:superposition:intersection} as building blocks for the flow, that will be given as a concatenation of smaller intersection regions. The idea is that, thanks to the subdivision of the flow into smaller pipes, we can perform the steps of Lemma \ref{lem:inter:flow} at smaller scales, using the superposition of smaller intersection regions. 

\subsubsection{{Subdivision for same pipe width}}

    \begin{figure}
        \centering
        \includegraphics[scale=0.5]{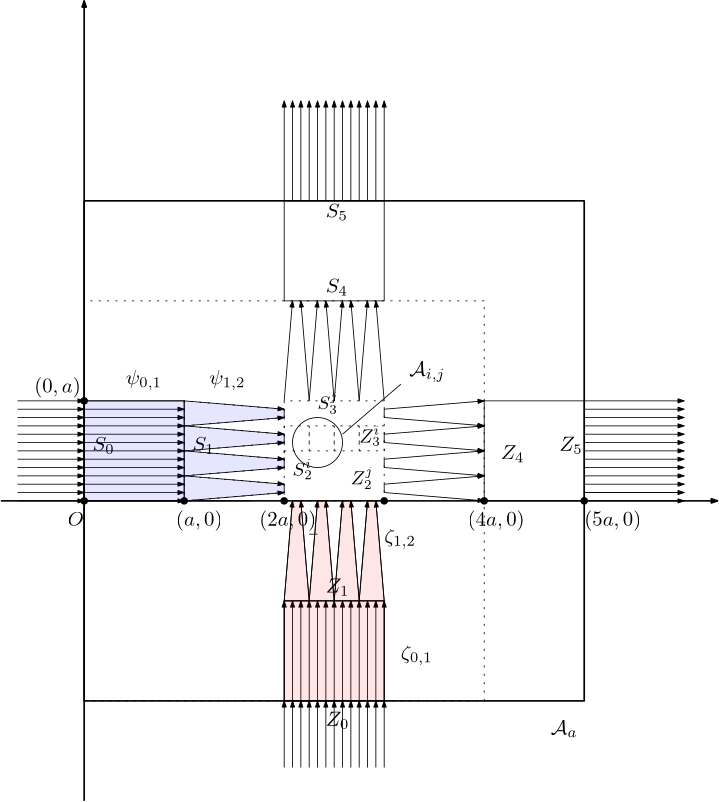}
        \caption{The flow splits into $\bar K$ pipes in order to keep the construction independent of the intersection region.}
        \label{fig:build:block:2}
    \end{figure}

 With a reference to Figure\ref{fig:build:block:2},  we fix $a>0$  and we consider a finer subdivision inside a region $\mathcal A=(0,5a)\times(-2a,3a)$ (see Figure \ref{fig:build:block:2}). We introduce the following parameters: 
\begin{itemize}
    \item $a>0$ the size of $\mathcal{A}$ (which is also the thickness of the main pipes); 
    \item $v_0>0$ is the velocity parameter;
    \item $\tau>0$ is the time parameter, with $\tau=
    {a}/{v_0}$. It clearly depends on the intersection region and indeed the time to cross the intersection region is of the order of $\tau$.
\end{itemize}
Let $s$ be a positive parameter sufficiently small ($s<a/(5v_0)$) be fixed. Then we can construct a sequence of $s_{\bar K}<s$, with $s_{\bar K}\to 0$ such that 
\begin{equation}
    s_{\bar K}=\frac{a}{v_0}\frac{\bar K}{(1+4\bar K)^2}>0
\end{equation}
for some parameter $\bar K=\bar K(\mathcal{A},v_0,s)\in\N$ (the choice of the parameter $\bar K$ is motivated by the proof of the next result). 
Define the following thickness parameters as
\begin{itemize}
    \item $L=5{a}/({1+4\bar K})$;
    \item $L'=a/{\bar K}$.
\end{itemize}
 We will see that the quantity $L/5$ represents the thickness of the new subpipes. 
\begin{prop}[Intersection Flow, I]\label{prop:inter:flow:subdivisions}

        Let us fix $a,v_0,s>0$ positive parameters and consider $\mathcal{A}, \bar K(\mathcal{A},v_0)$ and $L$ defined as before. 
Then, for every $s_{\bar K}<s$ constructed as above there exists and a divergence-free vector field $v:[0,+\infty)\times \R^2 \rightarrow \R^2$ with
        $v\in L^\infty([0,+\infty),L^\infty(\R^2))\cap L^\infty([0,+\infty),\BV(\mathcal{A}))$ satisfying the following properties:
        \begin{itemize}
             \item $v$ is time-periodic of period $2s_{\bar K}$;
              \item denote by $\phi:[0,+\infty)\times \R^2\rightarrow \R^2$ the flow associated with the velocity field $v$. Then  there exists $\bar t>0$ with $\bar t\in \left[3\tau, 11\tau\right]$ such that
            \begin{equation*}\label{eq:1}
                \phi(\bar t,x,y)=(4a+x,y),\quad \text{a.e. } (x,y)\in (0,a)\times(0,a);
            \end{equation*}
            \begin{equation*}\label{eq:2}
                \phi(\bar t,x,y)=(x,y+4a),\quad \text{a.e. } (x,y)\in (2a,3a)\times(-2a,-a);
            \end{equation*}
            \item the velocity field $v$ satisfies the following (boundary) conditions:
            \begin{equation*}
                v(t,0,y)\cdot\un_1= v(t,a,y)\cdot\un_1= v(t,4a,y)\cdot\un_1= v(t,5a,y)\cdot\un_1=v_0, 
            \end{equation*}
            for $\quad y\in(0,a), \quad t\in[js_{\bar{K}}, (j+1)s_{\bar{K}}), \quad j\text{ odd}$;
            \begin{equation*}
                v(t,x,-2a)\cdot\un_2= v(t,x,-a)\cdot\un_2= v(t,x,2a)\cdot\un_2= v(t,x,3a)\cdot\un_2=v_0, 
            \end{equation*}
            for $\quad x\in(2a,3a), \quad t\in[js_{\bar{K}}, (j+1)s_{\bar{K}}), \quad j\text{ odd}$;
            \item the velocity field satisfies the bounds
      \begin{equation}
    \|v\|_{L^\infty([0,+\infty), L^\infty(\R^2))}\leq 36C_{\text{swap}}v_0.
\end{equation}
 Moreover
\begin{equation}
   \|v\|^p_{L^\infty(0,+\infty),L^p(\mathcal{A}))}\leq (C_{\text{swap}}\cdot 36\cdot 25)^p\cdot 12 \cdot v_0^pa^2.
\end{equation}
        \end{itemize}
\end{prop}
\begin{remark}\label{rmk:notation:s}
    Without loss of generality, we can assume $s_{\bar K}=s$. We will keep the notation $s_{\bar K}$ in the next proof, but throughout the paper we will use $s$ as period parameter.
\end{remark}
\begin{remark}
    Note that by the time $\bar t$ given by the proposition, the two shear flows have successfully merged and passed through the intersection region, behaving exactly as they would have if the two pipes had not intersected at all (Figure \ref{fig:building.block2.1}).
    \begin{figure}
        \centering
        \includegraphics[scale=0.4]{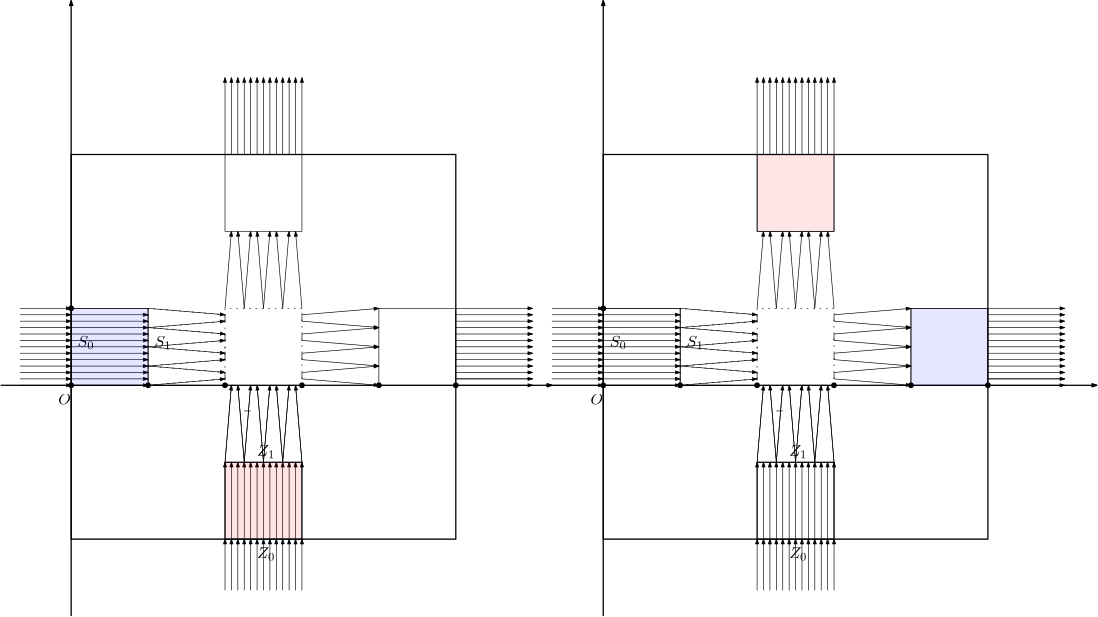}
        \caption{In time $\bar t$ given by Proposition \ref{prop:inter:flow:subdivisions} the two shear flows have successfully merged and recombined as if the two pipes were not intersecting at all.}
        \label{fig:building.block2.1}
    \end{figure}
\end{remark}
        \begin{proof}
            Define the following segments:
\begin{itemize}
    \item $S_0 = \{0\} \times (0,a)$;
    \item $S_1 = \{a\} \times (0,a) = \bigcup_{i=1}^{\bar K} S_1^i$, where 
    
    $S_1^i = \{a\} \times \left(L'(i-1), L'i\right)$, for $i=1,2,\dots, \bar K$;
    
    \medskip
    \item $S_2^i = \{2a\} \times \left( \frac{L}{5}\left(2+4(i-1)\right), \frac{L}{5}(3+4(i-1)) \right)$, for $i=1, \dots, \bar K$;
    \medskip
    \item $S_3^j = \left( 2a+\frac{L}{5}\left(2+4(j-1)\right), 2a+\frac{L}{5}(3+4(j-1))\right) \times \{a\}$, for $j=1,2,\dots,\bar K$;
    \medskip
    \item $S_4 = (2a, 3a) \times \{2a\} = \bigcup_{j=1}^K S_4^j$,
    
    where $S^j_4=\left(2a+\frac L5 (2+4(j-1)),2a+\frac L5 (3+4(j-1))\right)$, for $j=1,2,\dots,\bar K$;
\medskip
    \item $S_5 = (2a, 3a) \times \{3a\}$.
\end{itemize}

Call:
\begin{itemize}
    \item $\psi_{0,1}$ the source-sink flow transporting $(S_0, \un_1, v_0 \un_1)$ to $(S_1, \un_1, v_0 \un_1)$, and $u_{0,1}$ its associated velocity field;
    \item $\psi_{1,2}^i$ the source-sink flow transporting each $(S_1^i, \un_1, v_0 \un_1)$ to $(S_2^i, \un_1, w_0 \un_1)$ with $w_0$ chosen in such a way $v_0 L' = w_0 \frac{L}{5}$ (respecting Definition \ref{def:sourcesink}). Call $u_{1,2}^i$ their associated velocity fields;
     \item $\psi_{3,4}^j$ the source-sink flow transporting each $(S_3^j, \un_2, w_0 \un_2)$ (with $w_0$ chosen as before) into $(S_4^j, \un_2, v_0 \un_2)$; call $u_{3,4}^j$ their associated velocity fields.
    \item $\psi_{4,5}$ the source-sink flow transporting $(S_4, \un_2, v_0 \un_2)$ into $(S_5, \un_2, v_0 \un_2)$; call $u_{4,5}$ its associated velocity field.
\end{itemize}

Similarly, call:
\begin{itemize}
    \item $Z_0 = (2a, 3a) \times \{-2a\}$;
    \item $Z_1 = (2a, 3a) \times \{-a\} = \bigcup_{j=1}^K Z_1^j$, 
    
    where 
    $Z_1^j = \left(2a+L'(j-1), 2a+L'j\right) \times \{-a\}$, for $j=1,2, \dots, \bar K$;

    \medskip 
    \item $Z_2^j = \left( 2a+\frac{L}{5}\left(2+4(j-1)\right), 2a+\frac{L}{5}(3+4(j-1))\right) \times \{-a\}$, for $j=1,2,\dots,\bar K$;
    \medskip
    \item $Z_3^i = \{3a\} \times \left( \frac{L}{5}\left(2+4(i-1)\right), \frac{L}{5}(3+4(i-1)) \right)$, for $i=1, \dots, \bar K$;
\end{itemize}
\begin{itemize}
\item $Z_4 = \{4a\} \times (0, a) = \bigcup_{i=1}^K Z_4^i$, where  $Z^i_4=\lbrace 4a\rbrace\times (L'(i-1),L'i)$, for $i=1, \dots, \bar K$;
    \item $Z_5 = \{5a\} \times (0, a)$.
\end{itemize}

Again, call:
\begin{itemize}
    \item $\zeta_{0,1}$ the source-sink flow from $(Z_0, \un_2, v_0 \un_2)$ to $(Z_1, \un_2, v_0 \un_2)$, and $w_{0,1}$ its associated velocity field;
    \item $\zeta_{1,2}^j$ the source-sink flows transporting each $(Z_1^j, \un_2, v_0 \un_2)$ to $(Z_2^j, \un_2, w_0 \un_2)$ (again $v_0 L' = w_0 L/5$) and call $w_{1,2}^j$ their associated velocity fields;
    \item $\zeta_{3,4}^i$ the source-sink flow transporting each $(Z_3^i, \un_1, w_0 \un_1)$ into $(Z_4^i, \un_1, v_0 \un_1)$ and $w_{3,4}^i$ their associated velocity fields;
    \item $\zeta_{4,5}$ the source-sink flow transporting $(Z_4, \un_1, v_0 \un_1)$ into $(Z_5, \un_1, v_0 \un_1)$ and call $w_{4,5}$ its associated velocity field.
\end{itemize}

We denote by $\mathcal A_{ij}$, with $i,j=1, \dots, \bar K$, the following \emph{small intersection regions}:
\begin{equation*}
    \mathcal{A}_{ij}=\left(\frac 45 (i-1)L, \frac{4i+1}{5}L\right)\times \left(2a+\frac 45 (j-1)L, 2a+\frac{4j+1}{5}L\right),  
\end{equation*}
for $i,j=1,2,\dots,\bar K$. Finally call $$v^{\mathcal A_{ij}}=v^{i,j}: [0, +\infty) \times \R^2 \to \R^2$$ the velocity fields constructed in Lemma \ref{lem:inter:flow}, where now the parameters to be used in the lemma are 
    $$\text{size of }\mathcal{A}_{i,j} \longrightarrow L/5,\quad \text{ initial velocity } \longrightarrow w_0,$$ 
meaning that the flow acts on smaller intersection regions of area $L^2$ with initial velocity $w_0$ and moving particles in time $\sim \frac{L}{5w_0}=\frac{a}{v_0}\frac{\bar K}{(1+\bar K)^2}=s_{\bar K}$. 

We describe briefly our construction: the idea is that a fluid particle hitting the segment $S^i_2$ at some instant of time travels through the regions 
\begin{equation}
\mathcal{A}_{i1}\rightarrow\mathcal{A}_{i2}\rightarrow\dots\rightarrow\mathcal{A}_{i\bar K}.
\end{equation}

The amount of time spent in each region is $4s_{\bar K}$ (see point \ref{time:flow:lemma:1} of Lemma \ref{lem:inter:flow}). In particular, by Corollary \ref{coro:superposition:intersection}, if a fluid particle enters $\cup_{i,j}\mathcal{A}_{i,j}$ in the time $t_0$, it exits at time    
$$t_0+4\bar Ks_{\bar K}+2s_{\bar K} \in \left[t_0+\tau,t_0 +5\tau\right],$$ where we recall $\tau={a}/{v_0}$. Call for simplicity $t_{\text{inter}}=4\bar Ks_{\bar K}+2s_{\bar K}\in[\tau,5\tau]$. According to Corollary \ref{coro:superposition:intersection},  for each $i=1,2,\dots \bar K$ we denote by 
\begin{equation}
    v^{1,2,\dots,\bar K}_i:[0,+\infty)\times \R^2\rightarrow \R^2,
\end{equation}
the concatenation vector field of the regions $\mathcal{A}_{ij}$, with $j=1,\dots,
\bar K$.
We finally denote by 
\begin{equation}
    v^{U}:[0,+\infty)\times \R^2\rightarrow \R^2
\end{equation}
the sum $v^U(t,x,y)=v^{1,2,\dots,\bar K}_i(t,x,y)$ which is $2s_{\bar K}$-time periodic. Notice that, because of Corollary \ref{coro:superposition:intersection} we immediately have 

\begin{equation}\label{eq:l:ininity:norm}
          \|v^U\|_{L^\infty([0,+\infty),L^\infty(\R^2))}\leq 36C_{\text{swap}} w_0\leq  180C_{\text{swap}}v_0,
          \end{equation}
    where we have used that
    \begin{equation*}
        L'v_0=\frac{L}{5}w_0, \quad w_0\leq 5v_0,
    \end{equation*}
    and for $p\in[1,+\infty)$
     \begin{equation}\label{eq:l:p:norm}
          \|v^U\|^p_{L^\infty([0,+\infty),L^p(\cup_{i,j}\mathcal{A}_{ij}))}\leq 25(36)^pC_{\text{swap}}^pw_0^p \frac{L^2}{25}\bar K^2\leq (25\cdot 36\cdot C_{\text{swap}})^p v_0^p\cdot \frac{25}{4}a^2.
      \end{equation}
      
      Let us now define the pipe-flow in the non intersection region. Denote by
\begin{equation*}
    u=u_{0,1}+\sum_{i=1}^{\bar K} u^i_{1,2}+\sum_{j=1}^{\bar K} u^j_{3,4}+ u_{4,5}+w_{0,1}+\sum_{j=1}^{\bar K} w^j_{1,2}+\sum_{i=1}^{\bar K} w^i_{3,4}+ w_{4,5}.
\end{equation*}

We are finally ready to define the velocity field/flow of the statement.
\begin{equation*}
    v(t,x,y)=\begin{cases}
        v^U\llcorner_{(\cup_{i,j}\mathcal{A}_{ij})}(t,x,y), &t\in ((j-1)s_{\bar K},js_{\bar K}) \quad j \text{ odd}, \\
        \\
        v^U(t,x,y)+u(x,y), &t\in ((j-1)s_{\bar K},js_{\bar K}) \quad j \text{ even}, \\
        \\
        v_0\un_1, & (x,y)\in\left((-\infty,0)\cup(5a,+\infty)\right)\times(0,a),\\ & \quad t\in[js_{\bar K},(j+1)s_{\bar K}),  j=1,3,5,\dots, \quad  \\
              v_0\un_2, &(x,y)\in(2a,3a)\times\left((-\infty,-2a)\cup(3a,+\infty)\right), \\ &\quad t\in[js_{\bar K},(j+1)s_{\bar K}), j=1,3,5,\dots,  \\
              0 &\text{otherwise}.
    \end{cases}
\end{equation*}
Let us compute the amount of time a particle starting at $S_0$ takes to hit the segment $Z_4$. 
\begin{equation}
\bar t=t_{S_0S_1}+ t_{S_1S_2}+ t_{\text{inter}}+ t_{Z_3Z_4},
\end{equation}
where 
\begin{itemize}
    \item $t_{S_0S_1}$ is the time for a fluid particle starting in $S_0$ to arrive in $S_1$, which is computed as $2\tau$, according to Lemma \ref{lem:transport:tot} and the definition of the velocity field $v$;
    
\medskip 
     \item $t_{S_1S_2}$ is the time for a fluid particle starting in $S_1$ to arrive in $S_2$, which is $t_{S_1S_2}\in \left[2\tau/5, 2\tau\right]$, according to Lemma \ref{lem:transport:tot} and the definition of the velocity field $v$;
     \item $t_{\text{inter}}$ is the time spent inside the $\cup_{ij}\mathcal{A}_{ij}$, $t_{\text{inter}}\in \left[\tau,5\tau\right]$; 

     \medskip 
       \item $t_{Z_3Z_4}$ is the time for a fluid particle starting in $Z_3$ to arrive in $Z_4$, which is $t_{Z_3Z_4}=t_{S_1S_2}$, according to Lemma \ref{lem:transport:tot} and the definition of the velocity field $v$.
\end{itemize}
In particular $\bar t\in \left[3\tau, 11\tau\right]$. In order to obtain the $L^\infty_{t,x}$ and $L^\infty_tL^p_x$ estimates, we use Lemma \ref{lem:transport:tot}. Combining the result with Equations \eqref{eq:l:ininity:norm} and \eqref{eq:l:p:norm}, we find immediately that

\begin{equation*}
    \|v\|_{L^\infty(0,+\infty),L^\infty(\R^2))}\leq 36C_{\text{swap}}v_0,
\end{equation*}
 from which it immediately follows that, for $p\in[1,+\infty)$
\begin{equation*}
     \|v\|^p_{L^\infty(0,+\infty),L^p(\mathcal{A}))}\leq 36^p\cdot 8\cdot v^p_0 a^2+(25\cdot 36\cdot C_{\text{swap}})^p v_0^p\cdot \frac{25}{4}a^2\leq (C_{\text{swap}}\cdot 36\cdot 25)^p\cdot 12 \cdot v_0^pa^2.
\end{equation*}
\end{proof}
\begin{remark}
It is worth noting that in the limit $s\to 0$, the flow transitions into a non-deterministic regime, effectively giving rise to probabilistic autonomous solutions. This intriguing phenomenon offers a compelling perspective for future investigation and we believe is related to the notion of measure-valued solutions of Euler Equations \cite{BrenierDeLellisMV}.
\end{remark}
\subsubsection{Intersection of pipes with different thicknesses}

In the previous section we have constructed a flow that move across any intersection region $\mathcal{A}_a$ and it is time periodic of period $2s$ (see also Remark \ref{rmk:notation:s}). We showed that the period parameter is independent of $a$ and we have transferred the dependence on the intersection region $\mathcal{A}_a$ onto the number of parallel subpipes, governed by the parameter $\bar K$. With this in mind, we notice however that pipes with different thickness may \emph{intersect}. 

Consider an intersection region $\mathcal{A}_{a,b}$ of parameters $a,b>0$ defined (up to translation) as  
\begin{equation*}
    \mathcal A=\mathcal{A}_{a,b}=(0,5b)\times\left(-2a,3a\right),
\end{equation*}
with $a,b>0$. Without loss of generality we may assume $b<a$ and we denote by $k:=a/b$ their fraction. The two pipes have thicknesses $a$ and $b$ respectively. Notice that, thanks to Remark \ref{rmk:space:between:pipes}, we can rigorously view the region where they intersect as $\mathcal{A}$ (in the sense that no other pipes intersect this region). Moreover, by the same remark, we might also assume that $k=4^K\in \N$ even. 
The construction is a slight variation of the one performed in Proposition \ref{prop:inter:flow:subdivisions}, therefore it will be omitted for the sake of readability. The main idea is that, in order to accomodate the change of thickness of the pipe, the number of subpipes in which we subdivide our construction must take into account the relation $a/b$. The proof can be found in Appendix \ref{appendix:C}.

\begin{prop}[Intersection Flow, II]\label{prop:inter:flow:subdivisions:II}
        Let us fix $a,b,v_0,w_0>0$ positive parameters with the property that
        $$av_0=bw_0,\quad a>b$$ and consider $\mathcal{A}_{a,b}$  defined as before. 
        Then, for every $s>0$ there exists $s'<s$ (depending only on the tiling parameter $N$) and a divergence-free vector field $$v:[0,+\infty)\times \R^2 \rightarrow \R^2
        \qquad v\in L^\infty([0,+\infty),L^\infty(\R^2))\cap L^\infty([0,+\infty),\BV(\mathcal{A}_{a,b}))$$ satisfying the following properties:
        \begin{itemize}
           \item \textbf{time-periodicity} $v$ is time-periodic of period $2s'$;
           \item \textbf{flow properties}: let $\phi:[0,+\infty)\times\R^2\rightarrow\R^2$ be the flow map of $v$. Then there exist $$\bar t_1\in \left[\frac{8b}{5w_0},\frac{9b}{v_0}\right]\qquad \bar t_2\in \left[\frac{8}{5}\frac{a}{w_0}, 11\frac{a}{w_0}\right]$$ such that
            \begin{equation}\label{eq:1}
                \phi(\bar t_1,x,y)=(x+5b,y),\quad \forall (x,y)\in \lbrace 0\rbrace\times\left(0,a\right);
            \end{equation}
            \begin{equation}\label{eq:2}
                \phi(\bar t_2,x,y)= (x,y+5a),\quad\forall(x,y)\in \left(2b,3b\right)\times \left\lbrace -2a\right\rbrace;
            \end{equation}
          \item \textbf{boundary conditions}: the velocity field $v$ satisfies the following (boundary) conditions:
            \begin{equation*}
                v(t,0,y)\cdot\un_1= v(t,b,y)\cdot\un_1= v(t,4b,y)\cdot\un_1= v(t,5b,y)\cdot\un_1=v_0, 
            \end{equation*}
            for $\quad y\in(0,a), \quad t\in[js', (j+1)s'), \quad j\text{ odd}$;
            \begin{equation*}
                v(t,x,-2a)\cdot\un_2= v(t,x,-a)\cdot\un_2= v(t,x,2a)\cdot\un_2= v(t,x,3a)\cdot\un_2=w_0, 
            \end{equation*}
            for $\quad x\in(2b,3b), \quad t\in[js', (j+1)s'), \quad j \text{odd}$.
            \medskip
            \item \textbf{Norm bounds}: the velocity field satisfies the bounds
      \begin{equation}
    \|v\|_{L^\infty([0,+\infty), L^\infty(\R^2))}
 \le 180 C_{\text{swap}} \max \lbrace{v_0,w_0}\rbrace=180C_{\text{swap}}w_0.
\end{equation}
In particular there exists a positive constant $C=C(p)$ depending only on $p$ and $C_{\text{swap}}$ such that
\begin{equation}
\| v \|_{L^\infty([0, +\infty), L^p(\mathcal{A}_{a,b}))}^p \leq C(p) (\max\left\lbrace v_0^p a, w_0^p b\right\rbrace b +w_0^p \cdot b\cdot a).
\end{equation}
        \end{itemize}
 \end{prop}

\begin{remark}\label{rmk:final:estimate}
    Notice that it easily follows by the estimates of the previous result that, if $p<+\infty$,
    \begin{equation*}
        \| v \|_{L^\infty([0, +\infty), L^p(\mathcal{A}_{a,b}))}^p \leq C(p) w_0^pab.
    \end{equation*}
\end{remark}

%% file: sec3_new.tex
\section{Flows in two dimensional pipes}\label{Section:4}

The previous section dealt with a rather special building block of the construction that is genuinely different from the counterpart in three or higher dimensions. In this section we recall the result of \cite{Schiffer_Zizza25}. We choose to limit to its minimum the construction, since it is largely a two-dimensional adaptation of the one of \cite{Schiffer_Zizza25}, with the difference that pipe-flows might intersect in a rectangular region.

The core idea of the construction stays the same: we first construct flows in \emph{pipes} that connect a subcube $q$ of a cube $Q$ to its counterpart $\psi(q)$ in $\psi(Q)$, so that after time $1$ (or some time $t<1$) the fluid from $q$ is transported to its destination. For simplicity we chose the final time $T=1$ being the proof valid for any fixed arrival time $T$.

The difficulties in \cite{Schiffer_Zizza25} may be summarized as follows:
\begin{itemize}
    \item We need to make sure that flows coming from different cubes $q$ and $q'$ do not interfere with each other, i.e. trajectories do not intersect;
    \item The fluid that is present in the pipes at time $t=0$ also gets moved; but we aim to preserve the position of all fluid particles that are not in any subcube (also cf. Lemma \ref{keylemma}).
\end{itemize}
The second issue is addressed in the same fashion as in \cite{Schiffer_Zizza25}: the pipes are constructed so that after the first transportation step we may partially reverse the movement of particles in the pipes. The first one, however, is resolved differently: in dimension $\geq 3$ we could directly construct the pipes (which morally are one-dimensional objects) such that they do not intersect. In the present low-dimensional setting this is impossible and we need to use the results from the previous section.

\subsection{Notation and statement of results}
Recalling some previously introduced notation, let $N \in \N$. We consider a subdivision of $M= [0,1)^2$ into $N^2$ squares, that is for $v \in \{0,\ldots,N-1\}^2$ let
$$ Q_v:= N^{-1}( v+ [0,1)^{2}).$$
For a permutation $\sigma \colon V \to V$, we associate a map $\psi_{\sigma} \colon M \to M$ via
\begin{equation} \label{def:psisigma}
    \psi_{\sigma} (x)= N^{-1} \sigma(v) + (x- N^{-1} v) \quad \text{if } x \in \q_v,~v \in V. 
\end{equation}
Observe that we have the following equivalence of norms:
\begin{equation} \label{def:equiv}
    \Vert \psi_{\sigma} - \Id_M \Vert_{L^p(M)} \sim N^{-1} N^{-2/p} \Vert \sigma - \Id_V \Vert_{l^p(V)}.
\end{equation}

The main goal is to prove the following theorem:
\begin{thm}\label{thm:sharp:permutations}
    Fix $q,p\in[1,+\infty]$. Then there exists a constant $C=C(p,q)>0$ such that for every $\psi_\sigma\in P(M)$, there exists a divergence free vector-field $v  \in L^q([0,1], L^p([0,1]^2)\cap L^\infty([0,1],\mathrm{BV}([0,1]^2)$ satisfying
    \begin{equation}
        \|v\|_{L^q_tL^p_x}\leq C\|\psi_\sigma-Id\|_{L^p}
    \end{equation}
    and, if $\phi^v_t$ denotes the unique Regular Lagrangian Flow of $v$, it holds $\phi^v_{t=1}=\psi_\sigma$. Moreover, there exists $s_0=s_0(N)$ such that the velocity field $v$ is time-periodic of period $2s_0$, with $s_0\rightarrow 0$ as $N\to\infty$.
\end{thm}

Instead of proving that the \emph{entirety} of any square $Q_v$ might be moved to $\psi_{\sigma}(Q_v)$ we instead first show it for a subcube. We subdivide each square $Q_v$ in $\mathcal{K}^2$ subcubes $q$ (where $\mathcal{K}$ is independent of the permutation) of sidelength $\ell =N^{-1}/\mathcal{K}$. 

%We further choose a purely dimensional constant $\N \ni \kappa >> \mathcal{K}$ and define 
%\[
%\lambda \coloneqq \kappa^{-1} \ell
%\]
%which will serve as a "standard pipe width" for certain flows that we construct (see \cite{Schiffer_Zizza25}).

The main result can be summarized as follows:

\begin{lemma}[Partial construction of flows] \label{keylemma}
Let $\sigma:\lbrace 0,1,2,\dots,N-1\rbrace^2\rightarrow\lbrace 0,1,2,\dots,N-1\rbrace^2$ be a permutation of vertices. 
Fix $q,p\in[1,+\infty]$. There is an $s_0 = s_0(N)>0$ and a constant  $C=C(p,q)>0$ such that, \begin{itemize}
    \item for any $0<s<s_0$;
    \item for any $\psi_\sigma\in \mathcal{D}_N$ 
    \item and for $\psi$ defined as 
    \begin{equation} \label{eq:Qvkappa}
\psi = \begin{cases}
    \id  & \text{if } x \in  \q_v \setminus q_{v}, \\
    \psi_{\sigma} & \text{if } x \in q_{v}.
\end{cases}
\end{equation}
\end{itemize}there exists a divergence-free vector field  $u\in L^q([0,1], L^p([0,1]^2)\cap L^\infty([0,1],\mathrm{BV}([0,1]^2)$ time-periodic in time, with period $2s$, satisfying 
    \begin{equation}
        \|u\|_{L^q_tL^p_x}\leq C\|\psi_\sigma-Id\|_{L^p}
    \end{equation}
    and, for $\phi^v_t$ being the unique Regular Lagrangian Flow of $v$, it holds $\phi^v_{t=1}=\psi$.
\end{lemma}
As in \cite{Schiffer_Zizza25}, Theorem \ref{thm:sharp:permutations} then follows by applying Lemma \ref{keylemma} $\mathcal{K}^2$ times and then rescaling time, as the construction parameter $\mathcal{K}^2$ is a purely dimensional constant. We recall that the analogous of Lemma \ref{keylemma} has been proved in \cite{Schiffer_Zizza25}. Its proof follows by the construction of a certain \emph{pipe-flow} that is modeled on the discrete solution. We give now the rigorous definition of pipe-flow and we show that the existence of a pipe-flow gives the proof of Lemma \ref{keylemma}. First, we give the definition of pipe-flow in the three-dimensional setting, then we make explicit its difference with the two-dimensional setting and we carefully explain how to accomodate the differences, using the results of Section \ref{sec:intersection}.

\begin{figure}
    \centering
    \includegraphics[width=0.7\textwidth]{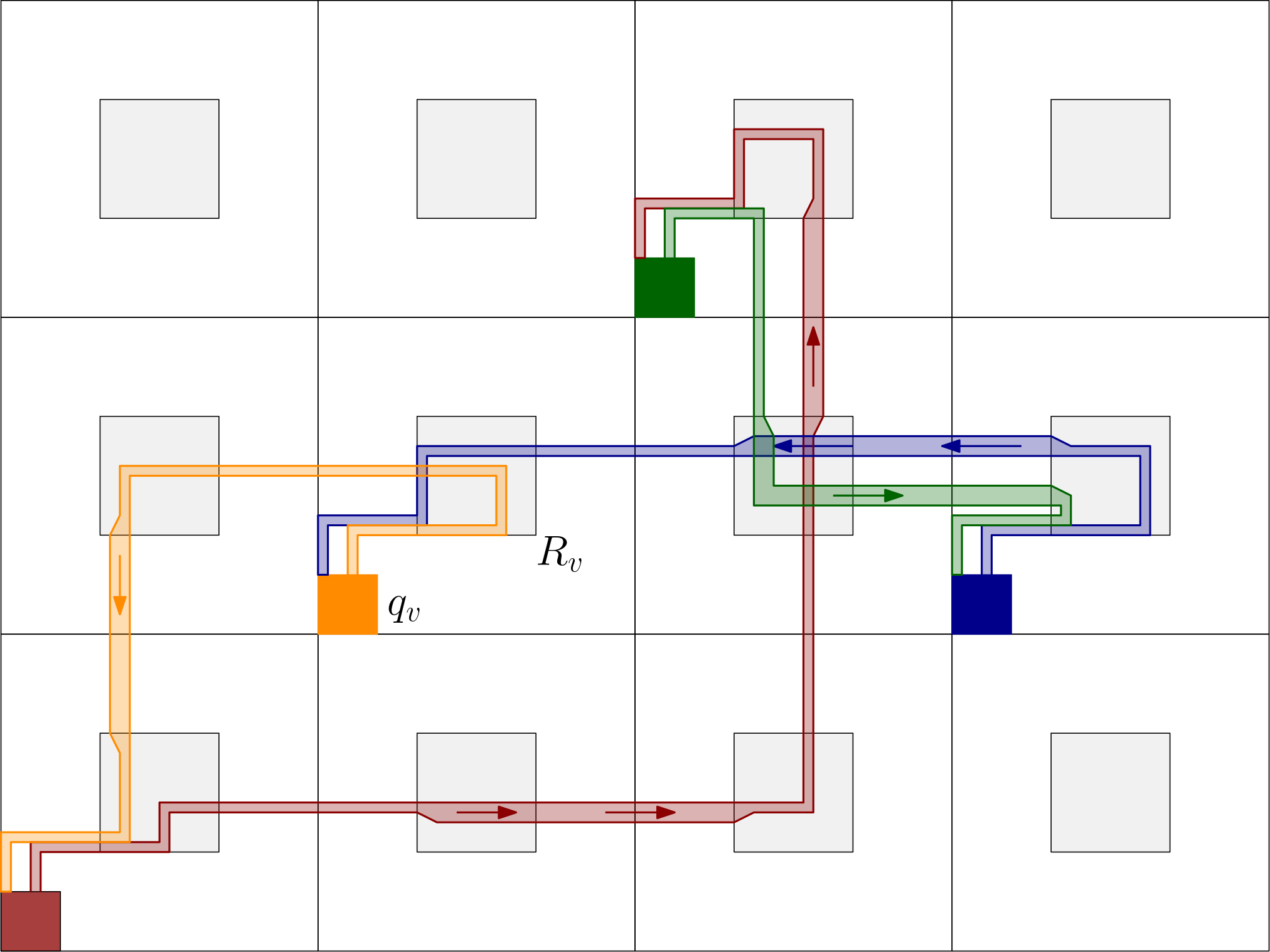}
    \caption{Schematic display of the pipe flow construction in any dimension. The flow is connected first from the cube $q_v$ to the nearest intersection region $R_v$, transported throught the 'network' and then again connected to $q_{\sigma(v)}$. Pipes might at first intersect, although (with careful construction) one can guarantee that this happens only in specific circumstances. In dimension three or higher one can then use the third dimension to correct this intersection (see \cite{Schiffer_Zizza25} and \ref{fig:3dPipes}). In 2D we replace the flow within the intersection region.}
    \label{fig:Shn1}
\end{figure}

\subsection{Pipe-Flow in $\nu\geq 3$}

Let $\psi_\sigma\in \mathcal{D}_N$ be a permutation of the tiling $\mathcal{R}_N$ as in the statement of Lemma \ref{keylemma}. Let $\lbrace \gamma_v\rbrace \subset \Gamma$ be a sequence of paths chosen to solve Problem $3.5$ in \cite{Schiffer_Zizza25}. Given $v\in V$, a pipe $\mathcal{P}_v\subset [0,1]^\nu$ is a subset of the $\nu$-dimensional cube.
\begin{definition}[Pipe-Flow in $\nu\geq 3$]\label{def:pipe:flow:3d}
    A pipe-flow modeled on $\lbrace\gamma_v\rbrace$ is a divergence-free vector field $u\in L^\infty([0,1],\BV([0,1]^\nu))\cap L^\infty([0,1],L^\infty([0,1]^\nu))$ satisfying the following properties:
    \begin{enumerate}
        \item for every vertex $v\in V$ there exists $\mathcal{P}_v$ pipe such that $\text{supp}(v)\subset\cup_v\mathcal{P}_v$;
        \item if $\mathcal{P}_v\cap \mathcal{P}_w\not=\emptyset$, then either $\mathcal{P}_v\cap \mathcal{P}_w=q_v$, $\mathcal{P}_v\cap \mathcal{P}_w=q_w$ or $\mathcal{P}_v\cap \mathcal{P}_w=q_v\cup q_w$;
        \item for every $p,q\in[1,+\infty]$ there exists a constant $C=C(\nu, p,q)>0$ such that
        $$\|u\|_{L^q_tL^p_x}\leq C\|\psi_\sigma-Id\|_{L^p};$$
        \item the time$-1$ map of the velocity field $u$ is $\phi^u_{t=1}=\psi$, where $\psi$ has been defined in the statement of Lemma \ref{keylemma}.
    \end{enumerate}
\end{definition}
Clearly, the proof of Lemma \ref{keylemma} follows by the construction of a pipe-flow. Notice that the constant $C$ in the inequality is independent of the choice of the permutation $\psi_\sigma$ or the size of the tiling $N$. We also remark that the construction of the vector field $u$ is obtained by a careful concatenation of source-sink flows (see Section $4$ in \cite{Schiffer_Zizza25} and Section \ref{Section:3} in the present paper).

In particular, the content of Sections $4$ and $5$ in \cite{Schiffer_Zizza25} can be summarized as follows:
\begin{thm}[Schiffer-Zizza \cite{Schiffer_Zizza25}]\label{thm:schiffer:zizza}
    For every $\nu\geq 3$, for every $N\in\N$, for every $\psi_\sigma\in\mathcal{D}_N$ there exists a pipe-flow $u$. 
\end{thm}

\subsection{Pipe-Flow in $\nu= 2$}
In two dimensions, there exist permutations $\psi_{\sigma}$, such that there are intersections of pipes as depicted in \ref{fig:Shn1}; these cannot be removed due to the low-dimension. We therefore deal with them differently from the three- or higher-dimensional case, see below Definition \ref{def:pipe:flow:2d}.

\begin{definition}[Pipe-Flow in $\nu=2$]\label{def:pipe:flow:2d}
    A pipe-flow modeled on $\lbrace\gamma_v\rbrace$ is a divergence-free vector field $u\in L^\infty([0,1],\BV([0,1]^2))\cap L^\infty([0,1],L^\infty([0,1]^2))$ satisfying the following properties:
    \begin{enumerate}
        \item for every vertex $v\in V$ there exists a pipe $\mathcal{P}_v$ such that $\text{supp}(v)\subset\cup_v\mathcal{P}_v$;
        \item if $\mathcal{P}_v\cap \mathcal{P}_w\not=\emptyset$, then the following situations may occur: \emph{either} one of the following occurs $\mathcal{P}_v\cap \mathcal{P}_w=q_v$, $\mathcal{P}_v\cap \mathcal{P}_w=q_w$, $\mathcal{P}_v\cap \mathcal{P}_w=q_v\cup q_w$ \emph{or} $\mathcal{P}_v\cap\mathcal{P}_w= \cup{_i} \mathcal{A}_{a_i,b_i,v,w}$, where $i\leq 2$ and each $\mathcal{A}_{a_i,b_i,v,w}$ is a rectangle satisfying the following properties:
        \begin{itemize}
            \item\label{2:pipe:1} there exist two edges $e_v=\lbrace v_1,\tilde v_1\rbrace \in E(\gamma_v), e_w=\lbrace w_1,\tilde w_1\rbrace\in E(\gamma_w)$ such that $e_v\cap e_w\not=\emptyset$ and they are \emph{orthogonal};
            \item\label{2:pipe:2} up to a rigid transformation $\mathcal{A}_{a_i,b_i,v,w}=(0,\ell\rho(e_v,\gamma_v))\times (0,\ell\rho(e_w,\gamma_w))$;
            \item\label{2:pipe:3}if $\mathcal{A}_{a_i,b_i,v,w}=(a_{i,1},a_{i,2})\times (b_{i,1},b_{i,2})$ then, denoting by $\mathcal{A}_{i,v,w}=(a_{i_1}-2(a_{i,2}-a_{i,1}),a_{i_2}+2(a_{i,2}-a_{i,1}))\times (b_{i_1}-2(b_{i,2}-b_{i,1}),b_{i_2}+2(b_{i,2}-b_{i,1}))$ we have that $\mathcal{A}_{i,v,w}\cap \mathcal{P}_v\cap\mathcal{P}_w= \mathcal{A}_{a_i,b_i,v,w}$.
        \end{itemize}
        \item For every $p,q\in[1,+\infty]$ there exists a constant $C=C(\nu, p,q)>0$ such that
        $$\|u\|_{L^q_tL^p_x}\leq C\|\psi_\sigma-Id\|_{L^p};$$
        \item the time$-1$ map of the velocity field $u$ is $\phi^u_{t=1}=\psi$, where $\psi$ has been defined in the statement of Lemma \ref{keylemma}.
    \end{enumerate}
\end{definition}

Item \ref{2:pipe:1} exactly represents the \emph{intersection region} defined previously in Section \ref{sec:intersection}. We can now give an equivalent statement to \ref{thm:schiffer:zizza} in 2D (see Theorem) \ref{thm:schiffer:zizza:2d} below; we however omit several rather straightforward adaptations to the lower-dimensional case to improve readiability and keep the presentation at a reasonable length. A more thorough presentation is given in a forthcoming version of this article.
% \st{We invite to recognize in Item \ref{2:pipe:1} the intersection region defined in Section \ref{sec:intersection}. We give the equivalent of Theorem \ref{thm:schiffer:zizza} for the dimension $\nu=2$, omitting the adaptation of the various steps of the proof of Theorem \ref{thm:schiffer:zizza} to the two-dimensional setting, in order to improve the readability of the result.}

\begin{thm}\label{thm:schiffer:zizza:2d}
    If $\nu=2$, for every $N\in\N$, for every $\psi_\sigma\in\mathcal{D}_N$ there exists a pipe-flow $u$. Moreover, there exists a time parameter $s_0=s_0(N)$ such that $u$ is time-periodic of period $2s_0$.
\end{thm}

\begin{proof}
    % \st{We sketch the steps of this proof, addressing the technical details in a forthcoming version of this paper. Let us  We fix $\psi_\sigma$, then consider $\lbrace \gamma_v\rbrace$ selection of paths solving Problem 3.2. We apply the machinery of \cite{Schiffer_Zizza25} in order to construct pipes $\mathcal{P}_v$ and the pipe-flow through concatenation of source-sink flows (see Snake flow, Gate Flow, Reservoir, Connector, \dots). \textcolor{magenta}{I don't think we need to mention the lex. order}In particular it is possible to specialize the lexicographic order and the level construction to the two-dimensional setting, with minor modifications. With a reference to Figure \ref{fig:intersections:shn:2d} it is possible to prove that we can construct pipes such that there are at most two intersection regions $\mathcal{A}_{a_i,b_i,v,w}$, with $i=1,2$, satisfying the statement of Definition \ref{def:pipe:flow:2d} (see also Remark \ref{rmk:space:between:pipes})}\footnote{A slightly more careful argument allows to extend the result to countably many intersections.}. 
    % \st{The construction can be made possible by purely topological reasoning. By definition of $\mathcal{A}_{a_i,b_i,v,w}$ we have $a_i=\ell\rho(e_v,\gamma_v)$ for some $\gamma_v\in \Gamma, e_v\in E(\gamma_v)$ and $b_i=\ell\rho(e_w,\gamma_w)$ for some $\gamma_w\in \Gamma, e_w\in E(\gamma_w)$, $v_{0,i}=\ell \frac{1}{\rho(\gamma,e)}$ and $w_{0,i}=\ell \frac{1}{\rho(\gamma',e')}$ (see also the choice of the parameters in \cite{Schiffer_Zizza25}).
    % }
    
    We fix a permutation $\sigma$ and $\psi_{\sigma} \in \mathcal{D}_N$ and consider a secltion of paths $\{\gamma_v\}_{v \in V}$ that solves Problem \ref{realproblem2}. We apply the machinery of \cite{Schiffer_Zizza25} to construct pipes $\mathcal{P}_v$ and the pipe-flow through concatenation of source-sink flows (cf. Figure \ref{fig:Shn1} for a rough sketch). In particular, by careful construction (see also the lexicographic order in \cite{Schiffer_Zizza25}) one may show that for any $v$ and $w$ there are at most two intersections regions $\mathcal{A}_{a_i,b_i,v,w}$ (see also Remark \ref{rmk:space:between:pipes}).

    By definition of $\mathcal{A}_{a_i,b_i,v,w}$ we have \[
    a_i=\ell\rho(e_v,\gamma_v), \quad b_i=\ell\rho(e_w,\gamma_w)
    \]
    for some $\gamma_v\in \Gamma, e_v\in E(\gamma_v)$ and $\gamma_w\in \Gamma, e_w\in E(\gamma_w)$, respectively. We then also set the veloocities  \[
    v_{0,i}=\ell \frac{1}{\rho(\gamma,e)}, \quad w_{0,i}=\ell \frac{1}{\rho(\gamma',e')},
    \]
    see also the choice of the parameters in \cite{Schiffer_Zizza25}.

     By closely following the proof in \cite{Schiffer_Zizza25} we may define the velocity field $\tilde u:[0,1]\times M\rightarrow\R^n$ on the set $M\setminus \cup_{i,v,w} \mathcal{A}_{a_i,b_i,v,w}$ step-by step. For any $p,q\in [1+\infty]$ we may show that there exists a constant $C=C(p,q)>0$ such that
    \begin{equation}
        \|\tilde u \|^p_{L^\infty_tL^p_x(M\setminus\cup_{i,v,w} \mathcal{A}_{a_i,b_i,v,w})}\leq C(p)N^{2-p}\sum_{\gamma,e}\rho(e,\gamma)^{-p+1}\leq C(p)\|\psi_\sigma-Id\|^p_{L^p}.
    \end{equation}
    The final velocity field is then defined as follow: Pick $s_0=s_0(N)>0$ sufficiently small (i.e. $s_0<\min_{v\in V}\rho^2(e_v,\gamma_v)$\footnote{Following the proof of Proposition \ref{prop:inter:flow:subdivisions} the construction should be accomodated into $K$ pipes. This choice is possible assuming that $\rho(e,\gamma)=4^{-k}$ for some $k\in \N$. Such a choice is possible with a loss of constant in the solution of Problem \ref{realproblem2}.}). Then
    \begin{equation}
        u(t,x)=\begin{cases}
        v^{\mathcal{A}_{a_i,b_i,v,w}}(t,x,y) &x\in \cup_{i,v,w}\mathcal{A}_{a_i,b_i,v,w}, \\
            \tilde u(t,x,y)& t\in [(j-1)s_0,js_0),\text{ j even},\\
            ~ &\quad x\in [0,1]^2\setminus  \cup_{i,v,w}\mathcal{A}_{a_i,b_i,v,w},  \\
            0 \quad & t\in [(j-1)s_0,js_0), \text{ j odd}, \\ &\quad x\in [0,1]^2\setminus \cup_{i,v,w}\mathcal{A}_{a_i,b_i,v,w}.
        \end{cases}
    \end{equation}
    We use then the result of Proposition \ref{prop:inter:flow:subdivisions:II} and the norm estimates of Remark \ref{rmk:final:estimate}, finding that (in the assumption $a_i> b_i$, that clearly can be reversed) 
    \begin{align*}
        \| u\|^p_{L^\infty([0,+\infty), \mathcal{A}_{a_i,b_i,v,w})}&\leq C(p) w_{0,i}^p a_ib_i\leq C(p)\ell^{2-p} \rho(e_w,\gamma_w)^{1-p}\leq C(p,\mathcal{K}) N^{2-p}\rho(e_w,\gamma_w)^{1-p}.
        \end{align*}
        Summing up over all possible intersection regions (recalling that two pipes intersect in at most $2$ intersection region) we recover the desired estimate. The case $p=\infty$ follows by similar observations. The only point left to address is the amount of time a fluid particle needs to reach its final destination, but it follows by similar reasoning noticing the the amount of time spent in each intersection region can be easily estimated by the amount of time spent as if the pipe never intersect (Proposition \ref{prop:inter:flow:subdivisions:II}).
\end{proof}

\begin{figure}
   \centering
   \includegraphics[width=0.6\textwidth]{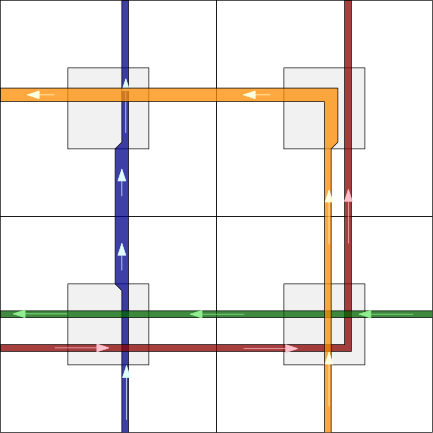}
    \caption{Construction without adaptation: Seen for itself, the flow e.g. in the blue pipe successfully connects $q_v$ to $q_{\sigma}(v)$. In dimension three or higher, the pipes themselves can be shifted (in the third dimension), keeping the spirit of the construction intact. In two dimensions, we need to adapt the flow, cf. Figures \ref{fig:Shn:2:1}--\ref{fig:Shn:2:3} below.}
    \label{fig:Shn:2:0}
\end{figure}

\begin{figure}
   \centering
   \includegraphics[width=0.6\textwidth]{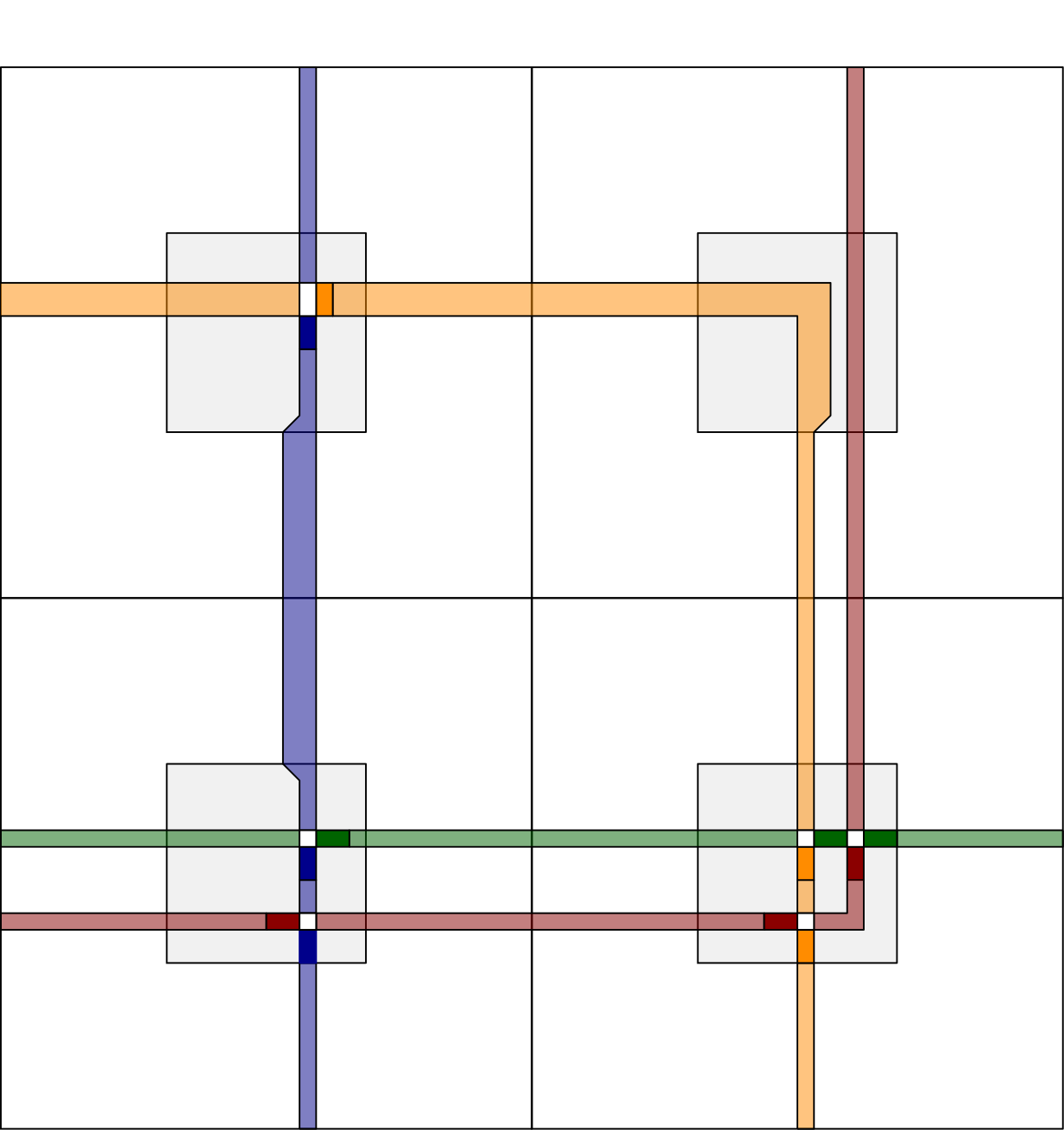}
    \caption{The configuration of the fluid at a time $t=(2j-1)s$. We pay close attention to the fluid in darker shaded region.}
    \label{fig:Shn:2:1}
\end{figure}

\begin{figure}
   \centering
   \includegraphics[width=0.6\textwidth]{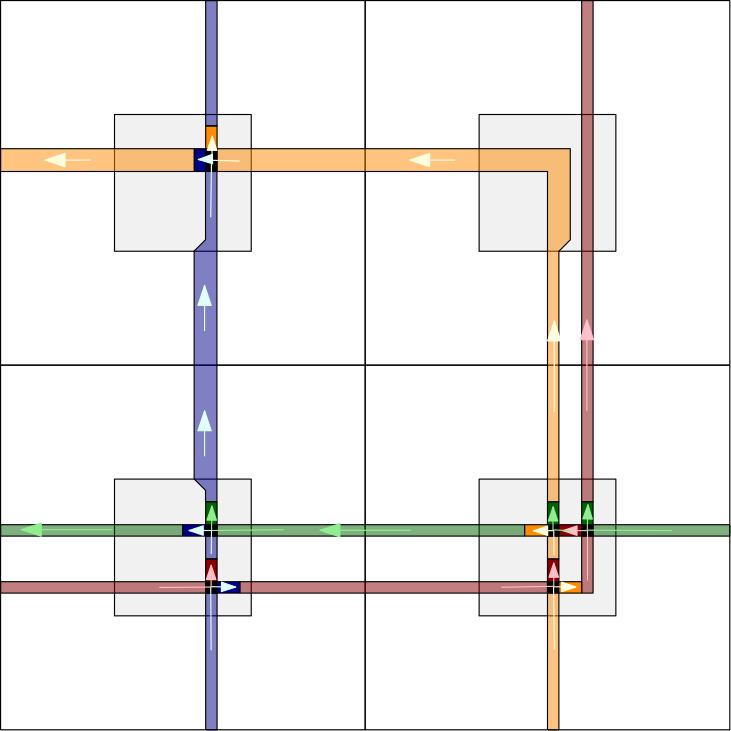}
    \caption{A rough sketch (compare to \ref{fig:swap:intro:4}) of the fluid after time $t=2js$. In the majority of the pipes, the flow is rather similar to the original one, cf. Figure \ref{fig:Shn:2:0}. In the intersection regions, however, the flow is adapted such that the initial intersection does not occur.}
    \label{fig:Shn:2:2}
\end{figure}

\begin{figure}
   \centering
   \includegraphics[width=0.6\textwidth]{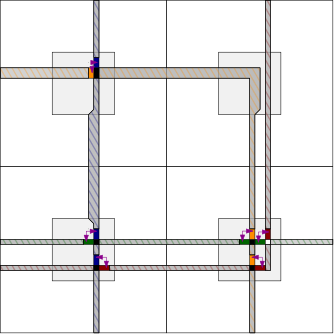}
    \caption{A rough sketc of the fluid after time $t=(2j+1)s$. The fluid is at rest in most of the pipes. The darker shaded regions (which where 'redirected' throught the modification in the intersection region get swapped in the swapping step.}
    \label{fig:Shn:2:3}
\end{figure}

  Figures \ref{fig:Shn:2:0}--\ref{fig:Shn:2:3} also show that for the construction we need to ensure two key properties (that are addressed in Section \ref{sec:intersection}): First, the intersections have to be constructed in a synchronised way (which is the first reason why we need to do the subdivision of intersections, cf. Subsection \ref{subsec:samewidth}). Second, for the bound on the $L^p$ norm of the velocity field, the velocity during the swapping procedure has to be bounded (cf. the entire Subsection \ref{subsec:refinement}).
%\begin{figure}
 %   \centering
  %  \includegraphics[scale=0.5]{intersectionsShn2d.png}
   % \caption{The light blue regions are the interchange regions $R_v$, and the light yellow regions are the highway flow regions $R_e$. The intersections can occur only in the interchange region, while in the highway flow region the flow is a superposition of horizontal/vertical shears ordered via the lexicographic order. }
    %\label{fig:intersections:shn:2d}
%\end{figure}

%% file: sec5_new.tex
\section{Proof of Sharp Shnirelman's Inequality}\label{S:section:6}
In the previous section we outlined the construction of the pipe flow and its adaptation to the intersection region. We obtained the following result (that we restate as a Theorem for reference).
\begin{thm}\label{thm:sharp:permutations}
    Fix $p,q\in[1,+\infty]$. Then there exists a constant $C=C(p,q)>0$ such that 
    \begin{itemize}
        \item for any permutation $\sigma$ and $\psi_{\sigma} \in P(M)$ of the tiling $\mathcal{R}_N$, $N\in \N$;
        \item any $s \in s(N)>0$ sufficiently small
    \end{itemize}
    there exists a $v \in [0,\infty) \times M \to \R^2$ with the following properties
    \begin{enumerate}
        \item $v$ is divergence-free,
        \item $v\in L^q([0,1], L^p([0,1]^2)\cap L^\infty([0,1],\text{BV}([0,1]^2)$;
        \item $v$ is time-periodic with period $2s$;
        \item We have the Shnirelman-type bound
        \begin{equation} \label{bound:Shnirelman}
        \|v\|_{L^q_tL^p_x}\leq C\|\psi_\sigma-Id\|_{L^p};
    \end{equation}
        \item if we denote by $\phi^v_t$ the unique Regular Lagrangian Flow of $v$, it holds that $\phi^v_{t=1}=\psi_\sigma$.
    \end{enumerate}
    % Fix $q,p\in[1,+\infty]$. Then there exists a constant $C=C(p,q)>0$ such that, for every $\psi_\sigma\in P(M)$ permutation of some tiling $\mathcal{R}_N$ with $N\in\N$, for any $s=s(N)>0$ sufficiently small, there exists $v:[0,+\infty)\times M\rightarrow \R^2$ divergence-free with $v\in L^q([0,1], L^p([0,1]^2)\cap L^\infty([0,1],\text{BV}([0,1]^2)$ time-periodic in time, with period $2s$, satisfying
    % \begin{equation}
    %     \|v\|_{L^q_tL^p_x}\leq C\|\psi_\sigma-Id\|_{L^p}
    % \end{equation}
    % and, if we denote by $\phi^v_t$ the unique Regular Lagrangian Flow of $v$, it holds $\phi^v_{t=1}=\psi_\sigma$.
\end{thm}

 \begin{remark}
     Notice that the parameter $s$, which is the period of the vector field constructed, depends on the tiling $N$, in particular $s\to 0$ as $N\to \infty$, while the constant $C$ in front of the inequality depends only on the dimension $d=2$ (and \emph{not} on $N$!).
 \end{remark}

In this section we aim to prove an analogous of the sharp Shnirelman's inequality in dimension $d=2$ (see also \cite{Schiffer_Zizza25}).
Let us define the following regularity class:
\begin{align*}\label{eq:class}
   \mathcal{C}_{q,p}(Id,f) = &\Big \{\ \phi\in AC_{\text{loc}}([0,1), S([0,1]^2)), \phi_{t=0}=Id, \lim_{t\to 1}\phi_{t}=f, \\&
   \text{ and }\dot\phi_t\circ\phi_t^{-1}\in L^q_{t,\text{loc}}([0,1),L^p([0,1]^2)) \Big\},
\end{align*}
where $\lim_{t\to 1}\phi_{t}=f$ is with respect to the strong $L^p$ topology and
\begin{equation}\label{eq:closure}
    S([0,1]^2)=\overline{\text{SDiff}([0,1]^2)}^{\|\cdot\|_{L^2}}=\lbrace \varphi:[0,1]^2\rightarrow [0,1]^2\text{ area-preserving}\rbrace.
\end{equation}
We recall that the previous identity \eqref{eq:closure} has been proved in \cite{Shnirelman},\cite{Shnirelman2},\cite{BrenierGangbo} in any dimension $\nu\geq 2$.
Our main result reads as follows:

   \begin{thm}[Sharp Shnirelman's Inequality in $\nu=2$]\label{thm:main:chapter:p:q}
    Let us consider $q,p\in[1,+\infty]$. Then there exists a positive constant $C=C(p)>0$ such that, for every $f,g\in\text{SDiff}([0,1]^2)$ there exists a flow $\varphi_t\in \mathcal{C}_{q,p}(f,g)$ such that the following sharp Shnirelman's Inequality holds
        \begin{equation}
        \inf_{\varphi_t\in \mathcal C_{q,p}(f,g)} \|\dot\varphi_t\circ\varphi_t^{-1}\|_{{L^1_tL^p_x}}\leq C\|f-g\|_{L^p([0,1]^2)}.
        \end{equation}
   \end{thm}

\begin{itemize}
   \item Without loss of generality can always assume $f=Id$.
    \item This theorem is not incompatible with the non-validity of Shnirelman's inequality for volume-preserving diffeomorphisms of the square. Fix $T>0$. Recall that we can define a geodesic distance on $\text{SDiff}(M)$ as
    \begin{equation}\label{eq:distance}
        \text{dist}_{\text{SDiff}(M)}(f,g)=\underset{\phi_t: \phi_0=f,\phi_T=g}{\inf}\mathcal{A}_{1,2}\lbrace\phi_t\rbrace,
    \end{equation}
    in the class of smooth paths of volume-preserving diffeomorphisms. For simplicity we will denote this class $\mathcal{S}(f,g)$.
     Since the class $\mathcal{C}_{1,p}(f,g)$ is bigger than the one considered in the example of Shnirelman, we trivially have
    \begin{equation}
        \underset{\phi_t\in\mathcal{C}_{q,p}(f,g)}{\inf}\mathcal A_{1,2}\lbrace\phi_t\rbrace\leq \underset{\phi_t\in\mathcal{S}(f,g)}{\inf}\mathcal A_{1,2}\lbrace\phi_t\rbrace,
    \end{equation}
    and  the two results are not contradicting each other.
    \item By convolution, we can also achieve smoothness of the vector field $v= \dot{\varphi}_t \varphi_t^{-1}$ in any compact interval $[0,T]$, $T>1$, but any type of smoothness is bound to break down at time $t=1$. In particular, if we change the definition of $C_{q,p}(Id,f)$ in such a way that $\lim_{t \to 1} = f$ is not meant in the strong $L^p$-topology, but in a $W^{\alpha,p}$ topology ($\alpha>0$), the construction breaks down, as there are new terms coming from the grid size $N^{-1}$ of the discrete approximation.
    \item Having constructed a solution, observe that we can let the frequency $s>0$ tend to zero (in \ref{thm:sharp:permutations} the parameter $s$ must be small enough, but might be chosen arbitrarily close to zero). In the limit, the ensuing vector fields converge against a Young measure. The precise characterization of the terminal limit as $t\rightarrow1$ goes beyond the scope of this work; however, we believe it is possible to embed our solution into the framework of generalized flows, or equivalently, to interpret it as a measure-valued solution to the Euler equations. See also \cite{Brenier1},\cite{Brenier99},\cite{AF},\cite{BrenierDeLellisMV}.
    
    %\textcolor{magenta}{The $L^1 L^p$ norm is not the problem. It is 
   % additional differentiability.}
    % \item Notice that $v=\dot\varphi_t\circ\varphi_t^{-1}\in L^1([0,T],L^p(\R^2))$ for every $T<1$, but this regularity breaks down at time $t=1$. In particular, as $t\rightarrow 1$ the vector field exhibits high-frequency temporal oscillations with a characteristic period $2s\rightarrow 0$. Due to this singular behavior, the trace of the vector field at $t=1$ is no longer a standard function, but can be rigorously characterized as a Young measure. The precise characterization of the terminal limit as $t\rightarrow1$ goes beyond the scope of this work; however, we believe it is possible to embed our solution into the framework of generalized flows, or equivalently, to interpret it as a measure-valued solution to the Euler equations. See also \cite{Brenier1},\cite{Brenier99},\cite{AF},\cite{BrenierDeLellisMV}.
\end{itemize}

   \begin{proof}
       Fix $f\in\text{SDiff}([0,1]^2)$. By Lax approximation theorem there exists a increasing sequence $\{N_n\}_{n \in \N}$ and permutations $P_n$ of tilings $\mathcal{R}_{N_n}$ such that $P_n\rightarrow f$ strongly in $L^p_x$, for any $p\in[1,+\infty]$. In particular, by Theorem \ref{thm:sharp:permutations}, there exists a sequence of vector fields $w_n:[0,+\infty)\times\R^2\rightarrow\R^2$ with the following properties: 
       \begin{itemize}
           \item $\mathrm{div} (w_n) =0$;
           \item $w_n\in L^\infty_t\BV_x$;
           \item the RLF associated to $w_n$, namely $\phi^{w_n}_t$, connects $Id$ to $P_n\circ P_{n-1}^{-1}$ in time $1$ (if, e.g. $N_n =2^n$, then  $P_{n-1}$ is a permutation of the tiling $\mathcal{R}_{N_n}$ thus $P_{n}\circ P_{n-1}^{-1}$ is a permutation of the tiling $\mathcal{R}_{N_n}$).
       \end{itemize}
       Moreover, there exists a positive constant $C$ independent of the construction such that 
       \begin{equation}\label{eq:norms}
           \|w_n\|_{L^\infty_tL^p_x}\leq C\|P_n -P_{n-1}\|_{L^p}, 
       \end{equation}
       where we have denoted for simplicity $P_0=Id$. Consider now the sequence of vector fields defined as follows: $v_1=w_1$,
       \begin{equation*}
           v_2(t,x)=\begin{cases}
               2w_1(2t,x), \quad t\in \left[0,\frac{1}{2}\right), \\
               2w_2(2t-1,x), \quad t\in \left[\frac{1}{2},1\right],
           \end{cases}
       \end{equation*}
       and more in general
       \begin{equation*}
           v_n(t,x)=\begin{cases}
               2w_1(2t,x), \quad t\in \left[0,\frac{1}{2}\right), \\
               4w_2(4t-2,x), \quad t\in \left[\frac{1}{2},\frac{3}{4}\right), \\
               \dots \\
               2^{n-1}w_{n}(2^{n-1}t-(2^{n-1}-1),x), \quad t\in \left[1-\frac{1}{2^{n-1}}, 1\right].
           \end{cases}
           \end{equation*}
           It is easy to prove that, up to a subsequence $\lbrace P_{n,k}\rbrace$ carefully chosen, any $L^\infty_tL^p_x$-norm of $v_n$ is uniformly bounded, or more precisely, for every $n$ it holds
           \begin{equation}
               \|v_n\|_{L^\infty_tL^p_x}\leq C\|f-Id\|_{L^p},
           \end{equation}  
          which gives the proof of the statement.
           \end{proof}

%% file: appendix.tex
\begin{appendix}
% \section*{Appendix}

\section{Proof of Lemmas \ref{lemma:2D:1},\ref{lemma:2D:2}} \label{S:appendix} \label{Ss:appendix:discrete}

We remind that the detailed explanations of proofs can be found in \cite{Schiffer_Zizza25}.
\begin{proof}[Proof of Lemma \ref{lemma:2D:1}.]
    Item \ref{lemma:2D:1:a} and \ref{lemma:2D:1:b} follow by a simple counting argument. In \ref{lemma:2D:1:a}, $e \in E(\Gamma_v)$ if and only if $v=(j,k)$ and either $j \leq i$ and $\sigma_1(v) > i$ \emph{or} $j \geq i+1$ and $\sigma_1(v) \leq i$. A similar argument works for \ref{lemma:2D:1:b}.

    Admissibility follows by the careful definition of $\rho$, cf. \eqref{def:flow:2D}. In particular the time constraint follows by $\rho(e,\gamma) \leq \ell(\gamma)^{-1} = \# E(\gamma)^{-1}$ for any path $\gamma$ and any $e \in E(\gamma)$. The capacity constraint follows by $\rho(e,\gamma) \leq F(e)^{-1}$ for each $\gamma$ such that $\gamma \in F(e)$.
\end{proof}
\begin{proof}[Proof of Lemma \ref{lemma:2D:2}]
    The $l^1$-equality again is shown by counting the same number in different ways:
     \begin{align*}
        \Vert F(\cdot) \Vert_{l^1(E)} =& \sum_{e \in E} \# \{ \gamma \colon e \in E(\gamma) \} = \sum_{v \in V} \# \{ e \colon e \in E(\gamma_v) \} = \Vert \dist(\cdot,\sigma(\cdot)) \Vert_{l^1(V)}.
    \end{align*}
    The $l^{\infty}$-inequality is more difficult. Let $K=\Vert \dist(\cdot, \sigma(\cdot)) \Vert_{l^{\infty}}$. If $e= \{(i,k),(i+1,k)\}$, we infer that $e \in E(\gamma_v)$ implies $i-K < v_1 < i+K+1$ and $v_2=k$ therefore $F(e) \leq 2K$, where we have used the observations of Lemma \ref{lemma:2D:1}. On the other hand, if $e=\{(i,k),(i,k+1)$, we have that $e \in E(\gamma_v)$ implies that $k-K < (\sigma(v))_2 < k+K-1$ and $(\sigma(v))_1 =i$, such that $F(e) \leq 2K$. We therefore get
    \[
    \Vert F(\cdot) \Vert_{l^{\infty}(E)} = \sup_{e \in E} F(e) \leq 2K = 2 \Vert \dist(\cdot,\sigma(\cdot) \Vert_{l^{\infty}(V)}.
    \]
    The $l^p$-inequality follows by interpolation, \cite{Schiffer_Zizza25}.
\end{proof}

\begin{proof}[Proof of Lemma 5.2 in \cite{Schiffer_Zizza25}, in the $2$-dimensional setting.]\label{proof:lemma:52}
    Without loss of generality we assume $s_1 > s_2$ (therefore $\mu_{S_{out}} < \mu_{S_{in}}$).
Call $\mu := \mu_{S_{out}} s_1 = \mu_{S_{in}} s_2$. Define
$$\tilde{u}(x, y) = \mu \left( \begin{array}{c} h x \cdot \frac{s_2 - s_1}{[(h-y)s_1 + y s_2]^2} \\ h \cdot \frac{1}{[(h-y)s_1 + y s_2]} \end{array} \right),$$
and
$$u(x, y) = \begin{cases} \tilde{u}(x, y) & (x, y) \in conv(S_{in}, S_{out}), \\ 0 & \text{otherwise} . \end{cases}$$
Notice that, for $x=0$ we have
$$\tilde{u}(0, y) = \mu \left( \begin{array}{c} 0 \\ \frac{h}{[(h-y)s_1 + y s_2]} \end{array} \right) .$$
with $\tilde{u}(0, 0) = \mu \left( \begin{array}{c} 0 \\ \frac{1}{s_1} \end{array} \right)$, $\tilde{u}(0, h) = \mu \left( \begin{array}{c} 0 \\ \frac{1}{s_2} \end{array} \right)$, thus it satisfies the boundary conditions. Similarly, take $(x, y)$ of the form $y = \frac{h}{s_2 - s_1} x - \frac{h s_1}{s_2 - s_1}$. Then we have
$$(h-y)s_1 + y s_2 =  h x,$$
so that
$$\tilde{u} \left( x, \frac{h}{s_2 - s_1} x - \frac{h s_1}{s_2 - s_1} \right) = \mu \left( \begin{array}{c} h x \frac{s_2 - s_1}{(h x)^2} \\ h \frac{1}{(h x)} \end{array} \right) =$$
$$= \frac{\mu}{h x} \left( \begin{array}{c} s_2 - s_1 \\ h \end{array} \right),$$
with the vector $\left( \begin{array}{c} s_2 - s_1 \\ h \end{array} \right)$ tangent to the line $y = \frac{h x}{s_2 - s_1} - \frac{h s_1}{s_2 - s_1}$. It can be easily proved that the vector field is divergence-free. To conclude the proof of the lemma we need to prove that $\exists \ t_{S_{in}, S_{out}} > 0$ such that the flow induced by $u$ satisfies the requirements of the statement.
To this end, consider the solution $\varphi$ of the ODE
$$\begin{cases} \varphi'(t) = \mu \cdot \frac{h}{(h - \varphi(t))s_1 + \varphi(t)s_2} \\ \varphi(0) = 0 \end{cases}$$

Then, for $(x, y) \in S_{out}$, the solution of the ODE
$$\begin{cases} \partial_t \psi(t) = \tilde{u}(\psi(t)) \\ \psi(0) = id \end{cases}$$

is given by
$$\psi(t, x, y) = \left( \begin{array}{c} \frac{x}{s_1} f(\varphi(t)) \\ \varphi(t) \end{array} \right)$$

where $f$ is obtained by the formula: $f(y) = \frac{s_1}{h} \cdot (h-y) + \frac{s_2}{h} \cdot y$. 

Finally observe that since
$$\frac{\mu}{s_1} \leq \varphi'(t) \leq \frac{\mu}{s_2}$$

Then
$$\frac{h s_2}{\mu} \leq t_{S_{in}, S_{out}} \leq \frac{h s_1}{\mu} .$$
The derivation of the $L^p$ bounds is straightforward. Indeed, notice that
\begin{equation*}
    |u_1|\leq \frac{u_3}{h}\max\lbrace s_1,s_2\rbrace,
\end{equation*}
thus
\begin{align*}
    \| u_3\|^p_{L^p}=\int_0^h \int_0^ {f(y)} \mu^p f(y)^{-p} dxdy\leq \mu^ ph\max\lbrace s_1^{1-p},s_2^{1-p}\rbrace\leq h\max\lbrace \mu_{S_{out} }^p\cdot s_1, \mu_{S_{in}}^p \cdot s_2\rbrace.
\end{align*}
This concludes the proof.
\end{proof}

\section{Constructions of Section \ref{sec:intersection}}\label{appendix:B}
\begin{proof}[Proof of Lemma \ref{lem:inter:flow}]
    In order to adjust the arrival time of the corner flow, it is possible to enlarge the intersection region. The quantities involved change by a factor independent of the construction. See Figure \ref{fig:building:11}.
    \begin{figure}
           \centering
           \includegraphics[scale=0.5]{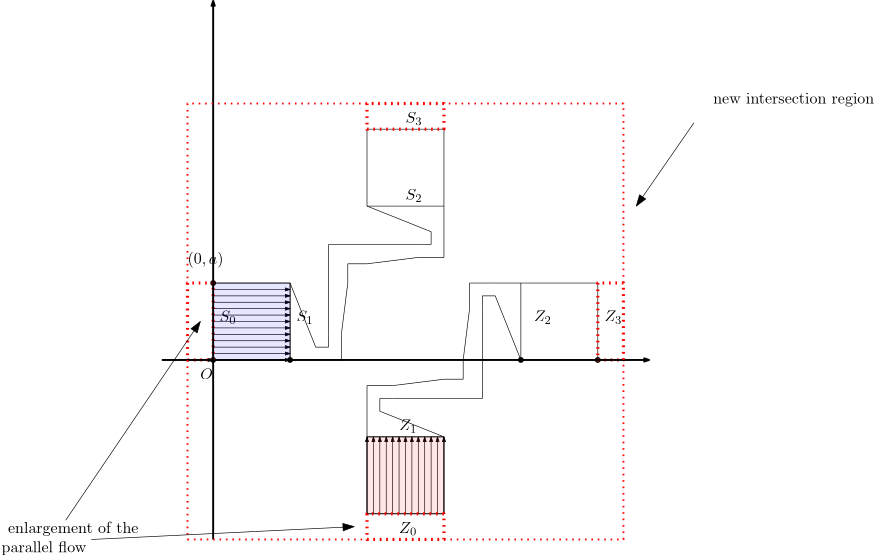}
           \caption{We can slightly deform the intersection region in order to make the time $t_{S_0S_1}=t_{S_1S_2}$, at the cost of a geometric constant independent of the parameters of the construction. }
           \label{fig:building:11}
       \end{figure}
\end{proof}

\section{}\label{appendix:C}
   \begin{proof}[Proof of Proposition \ref{prop:inter:flow:subdivisions:II}]
We refer to Figure \ref{fig:build:block:2} as a rather intuitive idea of how the dynamics works. We
define 
\begin{equation}
    M=\frac{5k}{4}+1
\end{equation}
and notice that $M$ is an integer number.
Observe that, by the choice of the parameters $k,M$ we have (since $k\geq 4$ so $a\geq 4 b$)
\begin{equation}
    \frac{a}{M}-\frac{b}{k}\geq \frac{16k}{5k+4}\frac{b}{k}-\frac{b}{k}>\frac{b}{k}.
\end{equation}
\paragraph{\textbf{Building blocks:} Horizontal Flow} 
Define the segments
\begin{align*}
  S_1 &= \bigcup_{i=1}^M S_1^i, &
  S_1^i &= \{0\}\times\Bigl(\tfrac{i-1}{M}a,\,\tfrac{i}{M}a\Bigr), \\
  S_4 &= \bigcup_{i=1}^M S_4^i, &
  S_4^i &= \{5b\}\times\Bigl(\tfrac{i-1}{M}a,\,\tfrac{i}{M}a\Bigr),
\end{align*}
and, for $i=1,\dots,M$,
\begin{align*}
  S_2^i &= \{2b\}\times\Bigl(\tfrac{(i-1)a}{M}+\tfrac{2b}{5k},\;\tfrac{(i-1)a}{M}+\tfrac{3b}{5k}\Bigr), \\
  S_3^i &= \{3b\}\times\Bigl(\tfrac{(i-1)a}{M}+\tfrac{2b}{5k},\;\tfrac{(i-1)a}{M}+\tfrac{3b}{5k}\Bigr).
\end{align*}
 
Set $\tilde w_0$ by the flux-matching condition
\[
  v_0\,\frac{a}{M} = \tilde w_0\,\frac{b}{5k}
  \quad\Longrightarrow\quad
  v_0 = \frac{M}{5k}\,\frac{b}{a}\,\tilde w_0 \leq {\frac{1}{2}\,\frac{b}{a}\,\tilde w_0}.
\]
\begin{itemize}
\item Let $\phi_{1,2}^i$ (resp.\ $u_{1,2}^i$) be the source-sink flow (resp.\ velocity field) from $(S_1^i, \un_1, v_0\un_1)$ to $(S_2^i, \un_1, \tilde w_0\un_1)$.
\item Let $\phi_{3,4}^i$ (resp.\ $u_{3,4}^i$) be the source-sink flow (resp.\ velocity field) from $(S_3^i, \un_1, \tilde w_0\un_1)$ to $(S_4^i, \un_1, v_0\un_1)$.
\end{itemize}

%\begin{figure}
%    \centering
  %  \includegraphics[scale=0.6]{buildingblockprop2.png}
 %   \caption{A look inside the area $(2b,3b)\times(0,a)$.}
%    \label{fig:build:block:prop:2}
%\end{figure}
\medskip\noindent{Vertical strips.}
Define, for $j=1,\dots,M$,
\begin{align*}
  Z_1^j &= \Bigl(2b+(j-1)\tfrac{b}{M},\;2b+j\tfrac{b}{M}\Bigr)\times\{-2a\}, \\
  Z_3^j &= \Bigl(2b+(j-1)\tfrac{b}{M},\;2b+j\tfrac{b}{M}\Bigr)\times\{3a\},
\end{align*}
and, for $i,j=1,\dots,M$,
\begin{align*}
  Z^{ij}_{\mathrm{in}}  &= \Bigl(2b+\tfrac{4j-2}{5}\tfrac{b}{k},\;2b+\tfrac{4j-1}{5}\tfrac{b}{k}\Bigr)
                            \times\Bigl\{\tfrac{(i-1)a}{M}\Bigr\}, \\
  Z^{ij}_{\mathrm{out}} &= \Bigl(2b+\tfrac{4j-2}{5}\tfrac{b}{k},\;2b+\tfrac{4j-1}{5}\tfrac{b}{k}\Bigr)
                            \times\Bigl\{\tfrac{(i-1)a}{M}+\tfrac{b}{k}\Bigr\}.
\end{align*}
 
The flux-matching condition for the vertical strips reads
\[
  \tilde w_0\,\frac{b}{5k} = w_0\,\frac{b}{M}
  \quad\Longrightarrow\quad
  \tilde w_0 = 5\,\frac{k}{M}\,w_0 \in[4w_0,5w_0].
\]
\begin{itemize}
\item $\psi_{1,2}^j$ (resp. $w_{1,2}^j$): source-sink flow (resp. vector field) from $(Z_1^j,\un_2,w_0\un_2)$ to $(Z^{1j}_{\mathrm{in}},\un_2,\tilde w_0\un_2)$.
\item $\psi_2^{ij}$ (resp. $w_2^{ij}$): source-sink flow (resp. vector field) from $(Z^{ij}_{\mathrm{out}},\un_2,\tilde w_0\un_2)$ to $(Z^{(i+1)j}_{\mathrm{in}},\un_2,\tilde w_0\un_2)$, for $i=1,\dots,M, j=1,\dots,M$.
\item $\psi_{2,3}^j$ (resp. $w_{2,3}^j$): source-sink flow (resp. vector field) from $(Z^{Mj}_{\mathrm{out}},\un_2,\tilde w_0\un_2)$ to $(Z_3^j,\un_2,w_0\un_2)$, $j=1,2,\dots, M$.
\end{itemize}

\noindent Intersection regions.

Let us define
\[ \mathcal{A}_{ij} = \left( 2b + \frac{4(j-1)}{5} \frac{b}{k}, 2b + \frac{4j+1}{5} \frac{b}{k} \right) \times \left( \frac{(i-1)a}{M}, \frac{(i-1)a}{M} + \frac{b}{k} \right) \]
for $i,j= 1 \dots, M$ and the velocity fields
\[ v_{ij} : [0, +\infty) \times \R^2 \to \mathbb{R}^2, \]
where $v_{ij}$ are the velocity fields defined in Proposition \ref{prop:inter:flow:subdivisions}. Notice that there exists $s'<s$ such that they are $2s'$- time periodic and they satisfy
\[ \| v_{ij} \|_{L^\infty([0, +\infty), L^\infty(\R^2))} \le 36 C_{\text{swap}}\tilde{w}_0. \]
Moreover $$\| v_{ij} \|_{L^\infty([0, +\infty), L^p(\mathcal{A}_{ij}))}^p \le 25 (36 C_{\text{swap}})^p \tilde{w}_0^p \left( \frac{b}{5k} \right)^2.$$

\paragraph{\textbf{Flow definition}}
Let us now define the flow
\[ u = \sum_{i=1}^M (u_{1,2}^i + u_{3,4}^i) + \sum_{j=1}^{M} (w_{1,2}^j + w_{2,3}^j) + \sum_{i=1}^M \sum_{j=1}^{M} w_2^{ij}, \] 
and the concatenation $v^i : [0, +\infty) \times \R^2 \to \mathbb{R}^2$ of $v_{ij}$, for every $i=1, \dots, M$, defined as in Corollary \ref{coro:superposition:intersection} and Remark \ref{rmk:superposition}.  In particular
$$\| v^i \|_{L^\infty([0, +\infty), L^\infty([0,1]^2))} \le 36 C_{\text{swap}} \tilde{w}_0,$$ and $$\| v^i \|_{L^\infty([0, +\infty), L^p(\cup_{j}\mathcal{A}_{i,j}))}^p \le 9 \cdot C_{\text{swap}}^p \cdot 36^p \cdot \tilde{w}_0^p \cdot M^2 \left( \frac{b}{5k} \right)^2.$$

We are finally ready to define our velocity field
\begin{equation}\label{eq:velocity field}
v(t,x,y)=\begin{cases}
        \sum_{i} v^i\llcorner_{(\cup_{i,j}\mathcal{A}_{ij})}(t,x,y), &t\in ((j-1)s',js') \quad j \text{ odd}, \\
        \\
        \sum_{i} v^i\llcorner_{(\cup_{i,j}\mathcal{A}_{ij})}(t,x,y), &t\in ((j-1)s',js') \quad j \text{ even}, (x,y)\in\mathcal{A}_{a,b} \\
        \\
        v_0\un_1, & (x,y)\in\left((-\infty,0)\cup(5a,+\infty)\right)\times(0,a),\\ & \quad t\in[js',(j+1)s'),  j=1,3,5,\dots, \quad  \\
              w_0\un_2, &(x,y)\in(2a,3a)\times\left((-\infty,-2a)\cup(3a,+\infty)\right), \\ &\quad t\in[js',(j+1)s'), j=1,3,5,\dots,  \\
              0 &\text{otherwise}.
    \end{cases}
\end{equation}
\medskip 
\textbf{Time estimates}. Before computing the relevant estimates, notice that a fluid particle starting in $S_2^i$ hits $S_3^i$ after an amount of time of $t_{S_2^i S_3^i}\in  \left[3\frac{b}{5k\tilde w_0}\cdot M, 11\frac{b}{5k\tilde w_0}\cdot M\right]$ (notice that we have used the estimate for $\bar t$ of Proposition \ref{prop:inter:flow:subdivisions}). In particular, since $\frac{4}{5}M\leq k\leq M$, we immediately have
\begin{equation}
    t_{S^i_2S^i_3}\in \left[\frac{3b}{5\tilde w_0}, 11\frac{b}{4 \tilde w_0}\right]. 
\end{equation}
A fluid particle starting in $S_1^i$ and ending up in $S_4^i$ needs the following amount of time:
\[ t_{S_1^i S_2^i} + t_{S_2^i S_3^i} + t_{S_3^i S_4^i}, \] where
\begin{itemize}
    \item $t_{S_1^i S_2^i}$ is the time that a particle in $S_1^i$ needs to hit $S_2^i$,
    
    that is $t_{S_1^i S_2^i} \in \left[ 2\cdot 2b \min \left\lbrace\frac{1}{v_0}, \frac{1}{\tilde{w}_0}\right\rbrace, 2\cdot 2b \max \left\lbrace\frac{1}{v_0}, \frac{1}{\tilde{w}_0} \right\rbrace \right]$ (see Lemma \ref{lem:transport:tot}). 
    
    \medskip 
     Notice that, because of the definition of the velocity field $v$, \eqref{eq:velocity field} the amount of time is doubled w.r.t. the formula given in Lemma \ref{lem:transport:tot}.
    \medskip 
    \item $ t_{S_3^i S_4^i}=t_{S_1^i S_2^i}$ is the time a fluid particle in $S_3^i$ needs to hit $S_4^i$.
\end{itemize}

\medskip 

Summing up we have
\[\bar t_1 := t_{S_1^i S_2^i} + t_{S_2^i S_3^i} + t_{S_3^i S_4^i} \in \left[ 8b\min\left\lbrace \frac{1}{v_0},\frac{1}{5w_0}\right\rbrace+\frac{3b}{25w_0},8b\max\left\lbrace \frac{1}{v_0},\frac{1}{4w_0}\right\rbrace+\frac{11b}{16w_0} \right]\subset \left[\frac{8b}{5w_0},\frac{9b}{v_0}\right]. \]
Similarly, the time for a fluid particle starting at $Z_1^j$ and ending up in $Z_3^j$ is given by
\[ \bar t_2 := t_{Z_1^j Z_{in}^{1j}} + \sum_{i=1}^Mt_{Z^{ij}_{in}Z^{ij}_{out}}+ \sum_{i=1}^M t_{Z^{ij}_{out} Z_{in}^{(i+1)j}} + t_{Z_{out}^{Mj} Z_3^j}. \]
Similar computations show that, if we sum up all the contributions, 

 \begin{equation}
     \bar t_2\in \left[\frac{8}{5}\frac{a}{w_0}, 11\frac{a}{w_0}\right].
 \end{equation}

\medskip 
\textbf{Norm Estimates: $L^\infty_{t,x}$-norm.}
We conclude the proof with the relevant estimates. We refer to Lemma \ref{lem:transport:tot} for the computations. We first observe that
$$\| v \|_{L^\infty([0, +\infty), L^\infty([0, 1]^2))} \le \max \{ \| u \|_{L^\infty([0, +\infty), L^\infty([0, 1]^2)}, \| v^i \|_{L^\infty([0, +\infty), L^\infty([0, 1]^2)))} \}.$$
Then we have 
\begin{align*}
\| u \|_{L^\infty([0, +\infty), L^\infty([0, 2b] \times (0, a) \cup [3b, 5b] \times (0, a)))} &
\le \max \{ v_0, \tilde{w}_0 \} \left( 1 + \frac{2}{2b} \max \left\lbrace \frac{a}{M}, \frac{b}{5k} \right\rbrace \right)  \\&
\le \tilde{w}_0 \left( 1 + \frac{1}{b} \cdot \frac{2a}{2k} \right) \le 2 \tilde{w}_0 \\&\le 10 w_0;
\end{align*}
\begin{align*}
\| u \|_{L^\infty([0, +\infty), L^\infty([2b, 3b] \times [-2a, 0] )} &
\le \max \{ \tilde{w}_0, w_0 \} \left( 1 + \frac{2}{2a} \max \left\lbrace\frac{b}{M}, \frac{b}{5k} \right\rbrace \right)  \\&
\le 5 w_0 \left( 1 + \frac{1}{a} \cdot \frac{b}{2k} \right) \le 10 w_0.
\end{align*}
\begin{align*}
\| u \|_{L^\infty([0, +\infty), L^\infty([2b, 3b] \times [a, 3a] )} &
\le \max \{ \tilde{w}_0, w_0 \} \left( 1 + \frac{2}{2a} \max \left\lbrace\frac{b}{M}, \frac{b}{5k} \right\rbrace \right)  \\&
\le 5 w_0 \left( 1 + \frac{1}{a} \cdot \frac{b}{2k} \right) \le 10 w_0.
\end{align*}
 Finally, if we consider the contribution given by the source-sink flow $\psi_2^{i,j}$ we have that
\begin{align*}
    \| u \|_{L^\infty([0, +\infty), L^\infty((2b,3b)\times(0,a))}\leq 5w_0 \left(1+\frac{2}{{\left(\frac{a}{M}-\frac{b}{k}\right)}} \frac{b}{k}\right)\leq {15} w_0,
\end{align*}
where we have used that $\frac{a}{M}-\frac{b}{k}>\frac{b}{k}$.
In particular,
$$\| v \|_{L^\infty([0, +\infty), L^\infty([0, 1]^2))} \le \max \{ {15} w_0, 36 C \tilde{w}_0 \} \le 36 C \tilde{w}_0 \le 180 C w_0.$$

\textbf{Norm Estimates: $L^\infty_{t}L^p_{x}$-norm.}
For the $L^\infty_t L^p_x$-norm one similarly has (by Lemma \ref{lem:inter:flow})
  \begin{align*}
&\| u \|_{L^\infty([0, +\infty), L^p([0, 2b] \times (0, a) \cup [3b, 5b] \times (0, a)))}^p \\ &\le  2M \left( 1 + \frac{2}{2b} \max \left\lbrace\frac{a}{M}, \frac{b}{5k} \right\rbrace\right)^p \cdot \max \{ v_0^p \cdot \frac{a}{M}, \tilde{w}_0^p \cdot \frac{b}{5k} \} \cdot 2b \\&
\le 2\cdot3^p \cdot \max \{ v_0^p a, \tilde{w}_0^p \frac{b}{10} \} \cdot 2b \le 4\cdot3^p \cdot \max \{ v_0^p a, 5^p w_0^p \frac{b}{10} \} b,
\end{align*}
 \begin{align*}
\| u \|_{L^\infty([0, +\infty), L^p([2b, 3b] \times [-2a, 0] \cup [2b, 3b] \times [a, 3a]))}^p &\le 2M \left( 1 + \frac{2}{2a} \max \{ \frac{b}{M}, \frac{b}{5k} \} \right)^p \cdot \\ &
\cdot \max \{ w_0^p \cdot \frac{b}{M}, \tilde{w}_0^p \cdot \frac{b}{5k} \} \cdot 2a \\ &\le 4 \cdot 2^p \cdot 5 \cdot 5^p (w_0^p \cdot b) \cdot a\\& 
\le 20 \cdot 10^p (w_0^p \cdot b) \cdot 2a.
\end{align*}
\begin{align*}
    \| u \|_{L^\infty([0, +\infty), L^p([2b, 3b] \times [0,a] }^p\leq M\cdot M \cdot (15w_0)^p \frac{b}{5k}\cdot\left (\frac{a}{M}-\frac{b}{k}\right)\leq (15 w_0^p)b\cdot a.
\end{align*}
 Moreover we recall that
$$\| v^i \|_{L^\infty([0, +\infty), L^p(\cup_{j}\mathcal{A}_{i,j}))}^p \le 9 \cdot C^p \cdot 36^p \cdot \tilde{w}_0^p \cdot M \cdot M\left( \frac{b}{5k} \right)^2.$$
Finally, summing up all the previous contributions we have:
\begin{align*}
\| v \|_{L^\infty([0, +\infty), L^p(\mathcal{A}_{a,b}))}^p \leq C(p) (\max\left\lbrace v_0^p a, w_0^p b\right\rbrace b +w_0^p \cdot b\cdot a).
\end{align*}

This concludes the proof.

\end{proof}

\end{appendix}